\documentclass[12pt]{article}
\usepackage{e-jc}

\usepackage{amsmath}
\usepackage{amssymb}
\usepackage{amsthm}
\usepackage{longtable}
\usepackage{lscape}

\allowdisplaybreaks
\theoremstyle{plain}

\newtheorem{thm}{Theorem}
\newtheorem{cor}[thm]{Corollary}
\newtheorem*{JTP}{Jacobi's triple product identity}
\newtheorem*{QPI}{Quintuple product identity}
\newtheorem*{BT}{Bailey's transform}
\newtheorem*{BL}{Bailey's lemma}
\newtheorem{lemma}[thm]{Lemma}

\newtheorem{id}[thm]{Identity}

\theoremstyle{definition}
\newtheorem{example}[thm]{Example}

\newcommand{\half}{\frac{1}{2}}
\newcommand{\ft}{\Psi}

\newcommand{\ts}{\psi}
\newcommand{\s}{\varphi}
\newcommand{\sn}[1]{\textbf{S. #1}}
\newcommand{\ftn}[1]{\textbf{F. #1}}

\date{\dateline{April 6, 2008}{May 24, 2008}{\small May 31, 2008} Revision of August 28, 2015
}

\title{Rogers-Ramanujan-Slater Type Identities}

\author{James Mc Laughlin \\
\small Department of Mathematics, 124 Anderson Hall,\\[-0.8ex]
\small West Chester University, West Chester, PA 19383 USA\\[-0.8ex]
\small \texttt{jmclaughlin@wcupa.edu}\\
\hfil\break\\
Andrew V. Sills \\
\small Department of Mathematical Sciences,\\[-0.8ex]
\small 203 Georgia Avenue, Georgia Southern University, Statesboro, GA 30460 USA\\[-0.8ex]
\small \texttt{asills@georgiasouthern.edu}\\
\hfil\break\\
Peter Zimmer\\
\small Department of Mathematics, 124 Anderson Hall,\\[-0.8ex]
\small West Chester University, West Chester, PA 19383 USA\\[-0.8ex]
\small \texttt{pzimmer@wcupa.edu}}

\numberwithin{equation}{subsection}

\begin{document}
\maketitle

\begin{abstract}
In this survey article, we present an expanded version of Lucy Slater's famous list of
identities of the Rogers-Ramanujan type, including identities of similar type, which were
discovered after the publication of Slater's papers, and older identities (such as those in Ramanujan's lost notebook) which were not included in Slater's papers.  We attempt to supply the
earliest known reference for each identity.  Also included are identities of false theta functions, along with their relationship to Rogers-Ramanujan type identities.
We also describe several ways in which pairs/larger sets of
identities may be related, as well as dependence relationships between identities.
\end{abstract}

\section{Introduction}
\subsection{Theta Functions}
\emph{Ramanujan's theta function}~\cite[p. 11, Eq. (1.1.5)]{AB05} is defined as
  \begin{equation}
     f(a,b):= \sum_{n=-\infty}^\infty a^{n(n+1)/2}b^{n(n-1)/2},
  \end{equation}
for $|ab|<1$. It is called a theta function, despite the lack of a
theta in the notation, because it is equivalent, via change of
variable, to the theta function of Jacobi~\cite[p. 463]{WW27}:
\begin{equation}
  \vartheta(z,w) := \sum_{n=-\infty}^\infty (-1)^n w^{n^2} e^{2niz},
\end{equation}
where $|w|<1$.

  The following special cases of
$f(a,b)$ arise so often that they were given their own special notation
by Ramanujan~\cite[p. 11]{AB05}:
\begin{align}
    \s(q) &:= f(q,q)\\
    \ts(q) &:= f(q,q^3)\\
    f(-q) &:= f(-q,-q^2).
\end{align}
Ramanujan further defines
\begin{equation}
    \chi(q) := \frac{f(-q^2,-q^2)}{ f(-q)}.
\end{equation}

 One of the most important results in the theory of theta functions is that they can be expressed as infinite products:
\begin{JTP}
For $|ab|<1$,
\begin{equation}\label{jtpeq}
   f(a,b) = (-a, -b, ab ; ab)_\infty,
 \end{equation}
where
 \[ (A;w)_\infty := \prod_{n=0}^\infty (1-Aw^n), \]
and
 \[ (A_1, A_2, \dots, A_r; w)_\infty := (A_1;w)_\infty (A_2;w)_\infty \cdots (A_r;w)_\infty. \]
\end{JTP}

  An immediate corollary of ~\eqref{jtpeq} is thus
\begin{cor}
\begin{align}
 f(-q) &= (q;q)_\infty  &\mbox{ \hskip 1cm (Euler~\cite{E1748}))}\\
 \s(-q) &= \frac{(q;q)_\infty}{(-q;q)_\infty} &\mbox{ \hskip 1cm (Gau\ss~\cite{G66})}\\
 \ts(-q) &= \frac{(q^2;q^2)_\infty}{(-q;q^2)_\infty }&\mbox{ \hskip 1cm (Gau\ss~\cite{G66})}\\
 \chi(-q) &= (q;q^2)_\infty
\end{align}
\end{cor}

  Sometimes a linear combination of two theta series can be expressed as a
single infinite product.  The quintuple product identity has
probably been (re)discovered more than any other identity in the
theory of $q$-series.  The earliest published occurrence appears to
be due to Fricke~\cite[p. 207, Eq. (6)]{F16}. See Cooper's survey
paper for a history and many proofs of the quintuple product
identity~\cite{C06}.
  \begin{QPI}
  \begin{align}
Q(w,x) &:=
  f(-wx^3, -w^2 x^{-3}) + x f(-w x^{-3}, -w^2 x^3) \notag \\
  &= \frac{ f(wx^{-1},x) f(-wx^{-2}, -wx^2)}{ f(-w^2,-w^4)} \notag\\
  &= (-w x^{-1}, -x, w; w)_\infty (w x^{-2}, w x^2; w^2)_\infty.
   \end{align}
  \end{QPI}
Bailey~\cite[p. 220, Eq. (4.1)]{B51} showed how certain linear combinations of theta series
can be simplified to a single theta series:
   \begin{equation}
   f( qz^2, q^3 z^{-2}) + z f(q^3 z^2, q z^{-2}) = f(z, q z^{-1} ).
   \end{equation}

\subsection{Bailey Pairs and Bailey's lemma}
Of central importance are the two Rogers-Ramanujan identities, which
each assert the equality of  a certain $q$-series with
a ratio of theta functions:
  \begin{align}
 G(q):=  \sum_{n=0}^\infty \frac{ q^{n^2} }{ (q;q)_n } &= \frac{f(-q^2,-q^3) }{f(-q)} \label{RR1}\\
 H(q):=  \sum_{n=0}^\infty \frac{ q^{n(n+1)}}{(q;q)_n} & = \frac{f(-q,-q^4) }{f(-q)} \label{RR2},
   \end{align}
 where
 \[ (a;q)_n := \left\{ \begin{array}{ll}
     (1-a)(1-aq)\cdots(1-aq^{n-1}) &\mbox{if $n\in \mathbb Z_+$}\\
     1 &\mbox{if $n=0$}\\
     \left[ (1-aq^{-1})(1-aq^{-2})\cdots (1-aq^{n})\right]^{-1} &\mbox{if $n\in \mathbb Z_-$.}
     \end{array} \right.
  \]
Note that it is customary to write $(A)_n$ for $(A;q)_n$ and
$(A)_\infty$ for $(A;q)_\infty$.

Using~\eqref{jtpeq}, and after some simplification, the right hand sides of~\eqref{RR1} and~\eqref{RR2} can
be expressed as, respectively, the infinite products
\[ \prod_{n=1}^\infty \frac {1}{(1-q^{5n-4})(1-q^{5n-1})}  \] and
\[ \prod_{n=1}^\infty \frac {1}{(1-q^{5n-3})(1-q^{5n-2}) }. \]

  Normally, an ``identity of Rogers-Ramanujan type" asserts the equality of a certain
$q$-series with a ratio of a theta series (or perhaps a linear combination of several
theta series) to one of $f(-q^m)$, $\s(-q^m)$, or $\ts(\pm q^m)$ for some positive integer $m$.
Sometimes, however, the term is applied to other types of $q$-series---product identities,
or polynomial generalizations thereof.
In particular, we limit ourselves to considering
$q$-series in which the power of $q$ is quadratic in the summation variable.  L. J. Rogers~\cite{R94, R17}
and S.  Ramanujan~\cite{AB05, AB07} discovered a number of identities of this type.
W. N. Bailey studied Rogers's work and as a result was able to
simplify and extend it in~\cite{B47, B49}.  L. J. Slater used and extended
Bailey's methods to obtain a large collection of identities of Rogers-Ramanujan type
in~\cite{S51, S52}.

\begin{BT} If
\begin{equation}\beta_n=\sum_{r=0}^n \alpha_r u_{n-r} v_{n+r} \label{BPdef}\end{equation} and
\begin{equation}\gamma_n = \sum_{r=n}^\infty \delta_r u_{r-n} v_{r+n}, \end{equation} then
\[ \sum_{n=0}^\infty \alpha_n \gamma_n = \sum_{n=0}^\infty  \beta_n \delta_n. \]
\end{BT}
\begin{proof}
\begin{equation*}
  \sum_{n=0}^\infty \alpha_n \gamma_n = \sum_{n=0}^\infty \sum_{r=n}^\infty
     \alpha_n \delta_r u_{r-n} v_{r+n}
     = \sum_{r=0}^\infty \sum_{n=0}^r \delta_r \alpha_n u_{r-n} v_{r+n}
    = \sum_{r=0}^\infty \delta_r \beta_r
\end{equation*}
\end{proof}
Bailey does not use his transform in the most general setting, but rather
specializes $u_n = 1/(q)_n$ and $v_n= 1/(xq)_n$.
Such a pair $(\alpha_n, \beta_n)$ which
satisfies~\eqref{BPdef} is called a \emph{Bailey pair} in the subsequent literature.

\begin{BL} If $\left(  \alpha_n (x,q) , \beta_n (x,q) \right)$ satisfies~\eqref{BPdef},
again with $u_n=1/(q)_n$ and $v_n=1/(xq)_n$, then so does
$\left( \alpha'_n (x,q),  \beta'_n (x,q) \right)$, where
\begin{equation} \alpha'_r(x,q) = \frac{ (\rho_1)_r (\rho_2)_r}{(xq/\rho_1)_r (xq/\rho_2)_r}
\left( \frac{xq}{\rho_1 \rho_2} \right)^r \alpha_r
\end{equation} and
\begin{equation}
  \beta'_n (x,q) = \frac{1}{(xq/\rho_1)_n (xq/\rho_2)_n }
      \sum_{j=0}^n \frac{(\rho_1)_j (\rho_2)_j (xq/\rho_1\rho_2)_{n-j} }{ (q;q)_{n-j} }
      \left( \frac{xq}{\rho_1 \rho_2} \right)^j \beta_j(x,q).
  \end{equation}
Equivalently, if
$\left(  \alpha_n (x,q) , \beta_n (x,q) \right)$ is a Bailey pair, then
\begin{multline}
\frac{1}{(xq/\rho_1)_n (xq/\rho_2)_n }
      \sum_{j= 0}^n \frac{(\rho_1)_j (\rho_2)_j (xq/\rho_1\rho_2)_{n-j} }{ (q)_{n-j} }
      \left( \frac{xq}{\rho_1 \rho_2} \right)^j \beta_j(x,q)
    \\= \sum_{r=0}^n \frac{ (\rho_1)_r (\rho_2)_r}{(q)_{n-r} (xq)_{n+r}
    (xq/\rho_1)_r (xq/\rho_2)_r }
\left( \frac{xq}{\rho_1 \rho_2} \right)^r \alpha_r (x,q). \label{SBL}
\end{multline}
\end{BL}
\begin{proof} See Andrews~\cite[pp. 25--27, Theorem 3.3]{A86}. \end{proof}

\begin{cor}
\begin{align}
\sum_{n=0}^\infty x^n q^{n^2} \beta_n(x,q) &= \frac{1}{(xq)_\infty}
  \sum_{r=0}^\infty x^r q^{r^2} \alpha_r (x,q) \label{aPBL} \\
\sum_{n=0}^\infty x^n q^{n^2} (-q;q^2)_n \beta_n(x,q^2) &= \frac{(-xq;q^2)_\infty}{(xq^2;q^2)_\infty}
  \sum_{r=0}^\infty \frac{x^r q^{r^2} (-q;q^2)_r }{(-xq;q^2)_r} \alpha_r (x,q^2) \label{aTBL} \\
\sum_{n=0}^\infty x^n q^{n(n+1)/2} (-1)_n \beta_n(x,q)
    &= \frac{(-xq)_\infty}{(xq)_\infty}
  \sum_{r=0}^\infty \frac{x^r q^{r(r+1)/2} (-1)_r }{(-xq)_r} \alpha_r (x,q) \label{aSBL}
 \end{align}
\end{cor}
\begin{proof}
Let $n,\rho_2\to\infty$ in~\eqref{SBL} to obtain
\begin{multline} \label{SWBL}
 \sum_{j= 0}^\infty \frac{ (-1)^j  x^j q^{j(j+1)/2} (\rho_1)_j}{\rho_1^j} \beta_j (x,q)\\
= \frac{(xq/\rho_1)_\infty}{(xq)_\infty}
\sum_{r=0}^\infty \frac{ (-1)^r x^r q^{r(r+1)/2} (\rho_1)_r }{\rho_1^r(xq/\rho_1)_r  } \alpha_r(x,q).
\end{multline}
To obtain~\eqref{aPBL}, let $ \rho_1\to\infty$ in~\eqref{SWBL}.
To obtain~\eqref{aTBL}, in~\eqref{SWBL}, set $\rho_1=-\sqrt{q}$ and
replace $q$ by $q^2$ throughout.
To obtain~\eqref{aSBL}, in~\eqref{SWBL} set $\rho_1=-q$.\end{proof}

\begin{cor}
\begin{align}
\sum_{n=0}^\infty q^{n^2} \beta_n(1,q) &= \frac{1}{f(-q)}
  \sum_{r=0}^\infty  q^{r^2} \alpha_r (1,q) \label{PBL} \tag{PBL} \\
\sum_{n=0}^\infty q^{n^2} (-q;q^2)_n \beta_n(1,q^2) &= \frac{1}{\ts(-q)}
  \sum_{r=0}^\infty q^{r^2}  \alpha_r (1,q^2) \label{TBL}\tag{TBL}\\
\sum_{n=0}^\infty  q^{n(n+1)/2} (-1)_n \beta_n(1,q)
    &= \frac{2}{\s(-q)}\sum_{r=0}^\infty  \frac{q^{r(r+1)/2}}{1+q^r}  \alpha_r (1,q) \label{S1BL}
    \tag{S1BL}
\end{align}
\end{cor}
\begin{proof}
  To obtain~\eqref{PBL},~\eqref{TBL}, and~\eqref{S1BL}, set $x=1$
in~\eqref{aPBL},
  ~\eqref{aTBL}, and~\eqref{aSBL} respectively.
\end{proof}

\begin{cor}
\begin{align}
\frac{1}{1-q}\sum_{n=0}^\infty  q^{n(n+1)/2} (-q)_n \beta_n(q,q)
    &= \frac{1}{\s(-q)}
  \sum_{r=0}^\infty  q^{r(r+1)/2}  \alpha_r (q,q) \label{S2BL} \tag{S2BL}  \\
\frac{1}{1-q}\sum_{n=0}^\infty  (-1)^n q^{n(n+1)/2} (q)_n \beta_n(q,q)
    &=
  \sum_{r=0}^\infty  (-1)^r q^{r(r+1)/2}  \alpha_r (q,q). \label{FBL} \tag{FBL}
 \end{align}
\end{cor}
\begin{proof}
To obtain~\eqref{S2BL}, set $x=q$ and $\rho_2=-q$ in~\eqref{SWBL}.
To obtain~\eqref{FBL}, set $x=\rho_2 = q$ in~\eqref{SWBL}.
\end{proof}

\begin{example}
The pair $(\alpha_n(x,q),\beta_n(x,q))$, where
  \[ \alpha_n(x,q) = \frac{ (-1)^n x^n q^{n(3n-1)/2} (1-xq^{2n}) (x)_n}
{ (1-x)(q)_n} \] and
  \[ \beta_n(x,q)= \frac{1}{(q)_n}\] form a Bailey pair.
\end{example}
\begin{proof}
\begin{align*}
  \beta_n (x,q) &= \sum_{r=0}^n \frac{ \alpha_r (x,q) }{(q)_{n-r} (xq)_{n+r}}\\
   &= \frac{1}{(q)_n (xq)_n} \sum_{r=0}^n
   \frac{ (q^{-n})_r}{(xq^{n+1})_r } (-1)^r q^{nr- r(r-1)/2} \alpha_r (x,q)\\
  &= \frac{1}{(q)_n (xq)_n} \sum_{r=0}^n
   \frac{ (x)_r (1-xq^{2r}) (q^{-n})_r  }{ (q)_r (1-x) (xq^{n+1})_r } x^r
   q^{2r^2} \\
  &= \frac{1}{(q)_n (xq)_n} \lim_{b\to\infty} \sum_{r=0}^n
    \frac{ (x)_r (q\sqrt{x})_r (-q\sqrt{x})_r (q^{-n})_r
    (b)_r^2 }{ (\sqrt{x})_r (-\sqrt{x})_r (xq^{n+1})_r (xq/b)_r^2 }
    \left( \frac{xq^{n+1}}{b^2} \right)^r \\
   &= \frac{1}{(q)_n (xq)_n} \lim_{b\to\infty} \frac{(xq)_n (xq/b^2)_n }
   { (xq/b)_n ^2 } \mbox{  (by Jackson's ${}_6\phi_5$ sum;
see~\cite[Eq. (II.21)]{GR04})}\\
   &= \frac{1}{(q)_n}
\end{align*}
\end{proof}
 Now that we have a Bailey pair in hand, it can be inserted into the
various limiting cases of Bailey's lemma to yield series-product identities.
Notice that
\begin{equation} \label{RRbeta}
\beta_n (1,q) = 1/(q)_n
\end{equation}
and
\begin{equation} \label{RRalpha}
\alpha_n (1,q) =
  \left\{ \begin{array}{ll}
      (-1)^n q^{n(3n-1)/2} (1+q^{n}) &\mbox{ when $n>0$}\\
      1 &\mbox{ when $n=0$.}
        \end{array} \right.
\end{equation}

 Upon inserting~\eqref{RRbeta} and~\eqref{RRalpha} into~\eqref{PBL},
we obtain
\begin{align*}
\sum_{n=0}^\infty \frac{q^{n^2}}{(q)_n}
 &=  \frac{1}{f(-q)} \left(
    1+\sum_{n=1}^\infty (-1)^n q^{n(5n-1)/2} (1+q^{n}) \right) \\
 &=  \frac{1}{f(-q)}\sum_{n=-\infty}^\infty (-1)^n q^{n(5n-1)/2} \\
 &=  \frac{ f(-q^2,-q^3)}{f(-q)} \qquad\mbox{ (by~\eqref{jtpeq})}
\end{align*}
the extremes of which are the first Rogers-Ramanujan identity.

 Upon inserting~\eqref{RRbeta} and~\eqref{RRalpha} into~\eqref{TBL},
we obtain
\begin{align*}
\sum_{n=0}^\infty \frac{q^{n^2} (-q;q^2)_n}{(q^2;q^2)_n}
 &=  \frac{1}{\ts(-q)} \left(
    1+\sum_{n=1}^\infty (-1)^n q^{n(4n-1)} (1+q^{2n}) \right) \\
 &=  \frac{1}{\ts(-q)}\sum_{n=-\infty}^\infty (-1)^n q^{n(4n-1)} \\
 &=  \frac{ f(-q^3,-q^5)}{\ts(-q)} \qquad\mbox{ (by~\eqref{jtpeq})}
\end{align*}
the extremes of which are the first Ramanujan-Slater/G\"ollnitz-Gordon identity.

Finally, inserting~\eqref{RRbeta} and~\eqref{RRalpha} into~\eqref{S1BL},
we obtain
\begin{align*}
\sum_{n=0}^\infty \frac{q^{n(n+1)/2} (-1)_n}{(q)_n}
 &=  \frac{2}{\s(-q)} \left(
    1+\sum_{n=1}^\infty (-1)^n q^{2r^2} (1+q^{2n}) \right) \\
 &=  \frac{1}{\s(-q)}\sum_{n=-\infty}^\infty (-1)^n q^{2r^2} \\
 &=  \frac{ f(-q^2, -q^2)}{\s(-q)} \qquad\mbox{ (by~\eqref{jtpeq})}
\end{align*}
the extremes of which are a special case of Lebesgue's identity~\eqref{leb},
which
appears in Slater's list~\cite{S52} as Eq. (12).

\subsection{Some More General $q$-Series Identities}
A number of Rogers-Ramanujan type identities are special cases of the identities in
this subsection.

\begin{id}[Euler]
\begin{gather} \label{eul}
 (-z)_\infty =\sum_{n=0}^\infty \frac{z^n q^{n(n-1)/2} }{(q)_n} \tag{E}
\end{gather} {(\cite{E1748}; cf.
 Andrews~\cite[p. 19, Eq. (2.2.6)]{A76}) }
\end{id}
\begin{id}[Cauchy]
\begin{gather}
 \frac{1}{(z)_\infty}  = \sum_{n=0}^\infty \frac{ z^n q^{n(n-1)}}{(q)_n (z)_n}
     \label{cauchy} \tag{C}
\end{gather}
(\cite[p. 20, Eq. (2.2.8)]{A76} )
\end{id}
\begin{id}[Lebesgue]
\begin{gather}
 (aq;q^2)_\infty (-q)_\infty  =\sum_{n=0}^\infty \frac{ q^{n(n+1)/2} (a)_n}{(q)_n}
 \label{leb} \tag{L}
\end{gather}
\mbox{(Lebesgue~\cite{L40}; cf. Andrews~\cite[p. 21, Cor. 2.7]{A76})}
\end{id}
\begin{id}[Heine's $q$-analog of Gau\ss's sum]
\begin{gather}
\frac{(c/a)_\infty }{(c)_\infty }
= \sum_{n=0}^\infty \frac{ (-1)^n c^n q^{n(n-1)/2} (a)_n }{a^n (q)_n (c)_n }
\label{qGauss} \tag{H}
\end{gather}
(\cite{H47}; cf. Gasper and Rahman~\cite[Eq. (II.8)]{GR04})
\end{id}
\begin{id}[Andrews' $q$-analog of Gau\ss's ${}_2 F_1(\frac 12)$ sum]
\begin{gather}
\frac{(aq;q^2)_\infty (bq;q^2)_\infty}{(q;q^2)_\infty (abq;q^2)_\infty}
= \sum_{n=0}^\infty \frac{ (a)_n (b)_n q^{n(n+1)/2} }{(q)_n (abq;q^2)_n}
\label{q2ndGauss} \tag{AG}
\end{gather}
(Andrews~\cite[p. 526, Eq. (1.8)]{A73}; cf. Gasper and Rahman~\cite[Eq. (II.11)]{GR04})
\end{id}
\begin{id}[Andrews' $q$-analog of Bailey's ${}_2 F_1(\frac 12)$ sum]
\begin{gather}
\frac{(cq/b;q^2)_\infty (bc;q^2)_\infty}{(c)_\infty}
= \sum_{n=0}^\infty \frac{ (b)_n (q/b)_n c^n q^{\binom{n}{2}} }
{(c)_n (q^2;q^2)_n} \label{qBailey} \tag{AB}
\end{gather}
(Andrews~\cite[p. 526, Eq. (1.9)]{A73}; cf. Gasper and Rahman~\cite[Eq. (II.10)]{GR04})
\end{id}

\begin{id}[Ramanujan]
\begin{align}
\frac{f(aq^3, a^{-1}q^3)}{f(-q^2)}
&=
\sum_{n=0}^\infty \frac{ q^{2n^2} (-a^{-1}q; q^2)_n (-aq; q^2)_n}{(q^2;q^2)_{2n}}
 \label{ram1}  \tag{R1}\\
\frac{f(aq^2, a^{-1}q^2)}{\ts(-q)}
&=\sum_{n=0}^\infty \frac{ q^{n^2} (-a^{-1}q ; q^2)_n (-aq;q^2)_n}{(q;q^2)_n (q^4;q^4)_{n}}
  \label{ram2} \tag{R2}
\end{align}
\end{id}
The preceding result of Ramanujan~\cite[p. 33]{R88} yields
infinitely many identities of Rogers-Ramanujan-Slater type  when $a$
is set to $\pm q^r$ for $r\in\mathbb Q$.

\section{Rogers-Ramanujan Type Identities}
The subsections are numbered to correspond to the modulus associated with
the product side of the identities, and identities are numbered sequentially
within each subsection.  Just as Mozart's compositions are identified according
to their listing in the K\"ochel catalog~\cite{K62} with a designation of the form ``K.$n$,"
Rogers-Ramanujan type identities will likely always be associated with their
appearance in Slater's list~\cite{S52}.  Accordingly, each identity below that
appears in~\cite{S52} is designated with a  ``Slater number" \textbf{S.$n$}.
The designation ``S.$n-$" means that $q$ has been replaced by $-q$ in
the $n$th identity in Slater's list.
The designation ``\textbf{S}.${n}$c" means that a corrected form of the $n$th identity
in Slater's list is being presented.
The earliest known occurrence of each identity is indicated.

\subsection{$q$-series Expansions of Constants and Theta Functions}
\begin{align}
0&= \sum_{n=0}^\infty \frac{ (-1)^n q^{n(n-1)/2}}{(q)_n}
   &\mbox{(\eqref{eul} with $z=-1$)} \\
1&= \sum_{n=0}^\infty \frac{ (-1)^n q^{n^2}}{(q;q^2)_{n+1} }
  &\mbox{(Rogers~\cite[p. 333 (4)]{R17})} \\
1 & = \sum_{n=0}^\infty \frac{ q^{n(n+1)} (q^2;q^2)_{n+1}}{
(-q^3;q^3)_{n+1}(q)_n } &\mbox{(M.-S.-Z.~\cite[Eq. (2.3)]{MSZ08})}\\
2&=\sum_{n=0}^\infty \frac{ q^{n(n-1)/2} }{(-q)_n}\\
f(-q) &= \sum_{n=0}^\infty (-1)^n q^{n(3n-1)/2} (1+q^n)
    &\mbox{(Euler; \textbf{S. 1}}) \\
 \frac{1}{f(-q)} &= \sum_{n=0}^\infty \frac{ q^{n^2}}{(q)_n^2} &\mbox{(Euler~\cite{E1748}; \eqref{cauchy}
 with $z=q$)}
\end{align}
\begin{multline}
\s(-q) = \sum_{n=0}^\infty \frac{(-1)^n q^{n(n+1)/2} (q)_n}{(-q)_n}\\
   \mbox {(Starcher~\cite[p. 805, (3.6)]{S31}; \eqref{qGauss} with $a=-c=q$)}
\end{multline}
\begin{multline}
\frac{1}{\s(-q)}  = \sum_{n=0}^\infty \frac{ q^{n(n+1)/2} (-1)_n}{(q)_n^2}\\
 \mbox{(Starcher~\cite[p. 805, (3.7)]{S31}; \eqref{qGauss} with $a=-1, c=q$)}
\end{multline}

\begin{align}
\ts(q) & = 1 - \sum_{n=1}^\infty \frac{ (-1)^n q^{n^2+n-1} (1-q)}{(q^2;q^2)_n (1-q^{2n-1})}
& \mbox{(Starcher~\cite[p. 807, (3.14)]{S31})} \\
\frac{1}{\psi(q)}& = \sum_{n=1}^\infty \frac{ (-1)^n (q;q^2)_n q^{n^2}}{ (q^2;q^2)_n^2}
&\mbox{(Ramanujan~\cite[p. 84, Entry 4.2.6]{AB07})}\\
(-q)_\infty & = \sum_{n=0}^\infty \frac{ q^{n(3n-1)/2} (1+q^{2n}) (-q)_{n-1}   }{(q)_n  }
& \mbox{(Starcher~\cite[p. 809, (3.29)]{S31})}
\end{align}

\subsection{Mod 2 Identities}
\begin{align}
\frac{ f(-q,-q) }{f(-q)}
&= \sum_{n=0}^\infty \frac{(-1)^n q^{n^2}}{(q^2;q^2)_n}
   &\mbox{(\eqref{eul} with $z=-\sqrt{q}$; \textbf{S. 3}) } \\
\frac{ f(-q,-q)}{\ts(-q)}
&=  \sum_{n=0}^\infty \frac{(-1)^n q^{n^2} (-q;q^2)_n }{(q^4;q^4)_n} \notag\\
   &\mbox{(Ramanujan~\cite[p. 104, Entry 5.3.6]{AB07}; \textbf{S. 4})} 
   \end{align}
   \begin{align}
\frac{ f(1,q^2) }{ \ts(-q)  }
    &=\sum_{n=0}^\infty \frac{q^{n(n-1)} (-q;q^2)_n}{(q)_{2n}}
    &\mbox{\eqref{qGauss}} \label{A7t}
\\
   &=
   2\sum_{n=0}^\infty \frac{ q^{n(n+1)} (-q;q^2)_n}{(q)_{2n+1} }
   &\mbox{\eqref{qGauss}} \label{F3t}
\end{align}

\subsection{Mod 3 Identities}
\begin{align}
\frac{ f(-q) }{ \s(-q) }
   &= \sum_{n=0}^\infty \frac{q^{n(n+1)/2}} {(q)_n} &\mbox{(\eqref{eul} with $z=q$;
   \textbf{S. 2}) } \\
\frac{ f(-q) }{ \s(-q^2)}
&=    \sum_{n=0}^\infty \frac{ (-1)^n q^{n(2n+1)} } {(-q;q^2)_{n+1}
   (q^2;q^2)_n}
         &\mbox{(\textbf{S. 5})} \\
\frac{ f(q,q^2) }{ f(-q) }
&=   \sum_{n=0}^\infty \frac{ q^{n^2} (-1)_n}{(q)_n (q;q^2)_n}
    &\mbox{(Ramanujan~\cite[p. 85, Ent. 4.2.8]{AB07}; \textbf{S. 6}c)}  \\
    &= \sum_{n=0}^\infty \frac{ q^{n^2} (-q)_n}{ (q)_n (q;q^2)_{n+1} }
    &\mbox{(Ramanujan~\cite[p. 86, Entry 4.2.9]{AB07})}\\
\frac{ f(q,q^2) }{ \ts(q) }
&= \sum_{n=0}^\infty \frac{ q^{2n^2} (q;q^2)_n^2}{ (q^2;q^2)_{2n} }
 &\mbox{(Ramanujan~\cite[p. 102, Entry 5.3.3]{AB07}) }\\
\frac{f(-q, q^2) }{f(-q) }
&=
 \sum_{n=0}^\infty   \frac{ q^{n(n+1)}(-1;q^2)_{n}  } {(q)_{2n}}
     &\mbox{(Ramanujan~\cite[p. 86, Ent. 4.2.10]{AB07}; \textbf{S. 48})}
\end{align}

\subsection{Mod 4 Identities}
\begin{align}
\frac{ f(-q,-q^3) }{f(-q) }
     &=\sum_{n=0}^\infty \frac{ q^{n(n+1)} }{(q^2;q^2)_n}
        &\mbox{(\eqref{eul} with $z=q$; \textbf{S. 7})}  \\
\frac{ f(-q^2,-q^2)}{ \ts(-q) }
&=
   \sum_{n=0}^\infty \frac{ q^{n^2} (q;q^2)_n }{ (q^4;q^4)_n }
  &\mbox{(\eqref{qGauss} with $a^2 = -c =q$; \sn{4-})}  \\
\frac{ f(-q,-q^3) }{\s(-q) }
&=\sum_{n=0}^\infty \frac{q^{n(n+1)/2}  (-q)_n }{(q)_n}
     &\mbox {(\eqref{leb} with $a=-q$; \textbf{S. 8})} \\
    &=  \sum_{n=0}^\infty \frac{q^{n(n+1)}  (-q;q^2)_n }{(q)_{2n+1} }
     &\mbox{(\eqref{qGauss} with $-aq=c=q^{3/2}$; \sn{51})}  
     \end{align}
     \begin{align}
\frac{ f(-q^2,-q^2) }{\s(-q) }
&=\sum_{n=0}^\infty \frac{ q^{n(n+1)/2}   (-1)_n} {(q)_n} \notag\\
    &    \mbox{(\eqref{leb}; Ramanujan\cite[p. 38, Entry 1.7.14]{AB07}; \sn{12})}\\
     &=\sum_{n=0}^\infty \frac{ q^{n(n+1)/2} (-q)_{n+1} }{(q)_n } \notag\\
    &\mbox{(\eqref{leb} with $a=-q^2$)}  \\
    &= \sum_{n=0}^\infty \frac{ q^{n^2}   (-q^2;q^2)_n} {(q)_{2n+1} } \notag\\
      &\mbox{(\eqref{qGauss}; Ramanujan~\cite[p. 37, Entry 1.7.13]{AB07})}\\
    &= \sum_{n=0}^\infty \frac{ q^{n^2}   (-1;q^2)_n} {(q)_{2n} } \notag\\
   &\mbox{(\eqref{qGauss} with $a=-1, c=\sqrt{q}$)}
\end{align}
\begin{align}
\frac{ f(q,q^3)}{ f(-q^2) }
&= \sum_{n=0}^\infty \frac{q^{n(2n+1)}}{(q)_{2n+1}}
      &\mbox{(Jackson~\cite[p. 179, line -3]{J28}; \sn{9})}
 \end{align}
 \begin{align}
\frac{ f(q,q^3) }{ \ts(-q) }
&= \sum_{n=0}^\infty \frac{q^{n^2} (-1)_{2n} }{(q^2;q^2)_{n} (q^2;q^4)_n}   &\mbox{(\sn{10c})} \\
\frac{f(q,-q^3) }{\ts(-q) }
&=
 \sum_{n=0}^\infty   \frac{q^{n^2} (-1;q^4)_n (-q;q^2)_n } {(q^2;q^2)_{2n}}
    &\mbox{(\sn{66})}  \\
\frac{f(-q,q^3)}{\ts(-q)}
&=
  \sum_{n=0}^\infty   \frac{q^{n(n+2)} (-1;q^4)_n (-q;q^2)_n } {(q^2;q^2)_{2n}}
 &\mbox{(\sn{67})}  \\
\frac{ f(q,q^3) }{ \s(-q^2) }
&= \sum_{n=0}^\infty \frac{ q^{n(n+1)}(-q)_{2n}}{(q;q^2)_{n+1} (q^4;q^4)_n} &\mbox{(\sn{11})} \\
\frac{ f(q,-q^3) }{\s(-q^2) }
&=
   \sum_{n=0}^\infty   \frac{ q^{n(n+1)} (-q^2;q^4)_n} {(q)_{2n+1} (-q;q^2)_n}
     &\mbox{(\sn{65})} \\
\frac{f(-q^2,-q^2) }{ \s(-q^2) }
&=  \sum_{n=0}^\infty \frac{ q^{n(n+1)} (q^2;q^2)_{n+1}}{(-q^3;q^3)_{n+1} (q)_n} &\mbox{(M.-S.-Z.~\cite[Eq. (2.3)]{MSZ08})}
  \end{align}

\subsection{Mod 5 Identities}
\begin{align}
H(q)=
\frac{f(-q,-q^4) }{f(-q) }
&= \sum_{n=0}^\infty \frac{q^{n(n+1)} } {(q)_n}
&\mbox{(Rogers~\cite[p. 330 (2)]{R94}; \sn{14})} \\
  &= \sum_{n=0}^\infty \frac{ q^{n^2} (q;q^2)_{n+1}}{(q)_n (q;q^2)_n}
& \mbox{(M.-S.-Z. \cite[Eq. (2.5)]{MSZ08})}
\end{align}
\begin{align}
G(q)=
\frac{f(-q^2,-q^3) }{f(-q) }
&=\sum_{n=0}^\infty \frac{q^{n^2}} {(q)_n}
     &\mbox{(Rogers~\cite[p. 328 (2)]{R94}; \sn{18})} \\
&     = \sum_{n=0}^\infty \frac{ q^{n(n+1)} (-q)_{n+1} } {(q^2;q^2)_n}
& \mbox{(M.-S.-Z.  \cite[Eq. (2.6)]{MSZ08}) }
\end{align}
\begin{align}
\chi(-q) H(q) =
\frac{f(-q,-q^4) }{ f(-q^2) }
&=
    \sum_{n=0}^\infty \frac{(-1)^n q^{n(3n-2)}} {(-q;q^2)_n (q^4;q^4)_n}
\notag \\
            & \mbox{(Rogers~\cite[p. 330 (5)]{R17}; \sn{15})} \\
&  =  \sum_{n=0}^\infty\frac{(-1)^n q^{n(3n+2)}} {(-q;q^2)_{n+1} (q^4;q^4)_n}
\notag\\
            & \mbox{(Ramanujan~\cite[p. 252, Eq. (11.2.7)]{AB05})}  \\
\chi(-q)G(q)  =
\frac{ f(-q^2,-q^3) }{ f(-q^2) }
&=
  \sum_{n=0}^\infty \frac{(-1)^n q^{3n^2}} {(-q;q^2)_n (q^4;q^4)_n}\notag\\
    &\mbox{(Rogers~\cite[p. 339, Ex. 2]{R94}; \sn{19})}
            \\
\chi(-q^2) H(q)
=\frac{ f(-q,-q^4)}{ \ts(-q) }
  &=\sum_{n=0}^\infty \frac{q^{n(n+2)}} {(q^4;q^4)_n} \notag \\
          &\mbox{(Rogers~\cite[p. 331, abv (7)]{R17}; \sn{16})}  \\
\chi(-q^2) G(q) =\frac{ f(-q^2,-q^3) }{ \ts(-q) }
&=
     \sum_{n=0}^\infty \frac{q^{n^2}} {(q^4;q^4)_n} \notag\\
  & \mbox{(Rogers~\cite[p. 330]{R94}; \sn{20})}
\end{align}
\begin{align}
\frac{ f(1,q^5) }{ \ts(q) }
      &=   2 \sum_{n=0}^\infty   \frac{(-1)^n q^{n(n+2)} (q;q^2)_n} {(-q;q^2)_{n+1} (q^4;q^4)_n}
      \notag\\
      &\mbox{(M.-S.-Z.~\cite[Eq. (2.7)]{MSZ08}) }\\
\frac{ f(q,q^4) }{\ts(q) }
&=
      \sum_{n=0}^\infty \frac{(-1)^n q^{n(n+2)}  (q;q^2)_n }
  {(-q;q^2)_n (q^4;q^4)_n} \notag\\
  &\mbox{(B.-M.-S.~\cite[Eq. (2.17)]{BMS07})}   \\
\frac{ f(q^2,q^3) }{ \ts(q) }
&=
      \sum_{n=0}^\infty \frac{(-1)^n q^{n^2}  (q;q^2)_n } {(-q;q^2)_n (q^4;q^4)_n} &\mbox{(\sn{21})}\\
 \frac{H(q) \chi(-q)}{\chi(-q^2) } =\frac{ f(-q,-q^4) }{\s(-q^2) }
&=
    \sum_{n=0}^\infty \frac{q^{n(n+1)}} {(q^2;q^2)_n (-q;q^2)_{n+1}} &\mbox{(\sn{17})}  \\
 \frac{G(q) \chi(-q)}{\chi(-q^2) } =\frac{f(-q^2,-q^3) }{ \s(-q^2) }
&=
            \sum_{n=0}^\infty \frac{q^{n(n+1)}} {(q^2;q^2)_n (-q;q^2)_{n}}&\mbox{(\sn{99-})}
 \end{align}

 \subsection{Mod 6 Identities}
 \begin{align}
\frac{ f(-q^3,-q^3)}{ f(-q^2) }
&=
   \sum_{n=0}^\infty \frac{ q^{2n^2} (q;q^2)_n}{(-q)_{2n} (q^2;q^2)_n} \notag\\
    &\mbox{(B.-M.-S. \cite[Eq. (2.13)]{BMS07})}\\
\frac{ f(-q,-q^5) }{ \ts(-q) }
  & = \sum_{n=0}^\infty \frac {q^{n(n+2)} (-q;q^2)_n} {(q^4;q^4)_n} \notag\\
     & \mbox{(Ramanujan~\cite[p. 87, Entry 4.2.11]{AB07}, Stanton~\cite{S01})}
     \label{RamStanton}
  \end{align}
 \begin{align}
\frac{ f(-q^3,-q^3) }{\ts(-q) }
&=
   \sum_{n=0}^\infty \frac {q^{n^2} (-q;q^2)_n} {(q^4;q^4)_n}&
   \mbox{(Ramanujan~\cite[p. 85, Ent. 4.2.7]{AB07}; \sn{25})}\\
\frac{ f(-q,-q^5)}{ \ts(q) }
&=
  \sum_{n=0}^\infty \frac{ (-1)^n q^{n^2}}{(q^2;q^2)_n}
       &\mbox{(\eqref{eul} with $z=-\sqrt{q}$; \sn{23})}
 \end{align}
 \begin{align}
\frac{f(-q,-q^5) }{ \s(-q) }
  &= \sum_{n=0}^\infty \frac{ q^{n(n+1)}  (-q)_n} {(q;q^2)_{n+1} (q)_n} \notag
\\ &
\mbox{(Ramanujan~\cite[p. 87, Entry 4.2.12]{AB07}, Bailey~\cite[p. 72, Eq. (10)]{B36};
\sn{22})}
\end{align}
\begin{align}
\frac{ f(-q^3,-q^3) }{ \s(-q) }
     &=      \sum_{n=0}^\infty \frac{ q^{n^2}  (-q)_{n}} {(q;q^2)_{n+1} (q)_n}&
     \mbox{(\sn{26})} \\
  &=    \sum_{n=0}^\infty \frac{ q^{n^2}  (-1)_{n}} {(q;q^2)_{n} (q)_n} &\mbox{(M.-S.-Z.~\cite{MSZ08})}
\end{align}
 \begin{align}
 \frac{ f(q,q^5) }{ f(-q^2) }
    &=      \sum_{n=0}^\infty \frac{ q^{2n(n+1)} (-q;q^2)_{n}} {(q)_{2n+1} (-q^2;q^2)_n}
    &\mbox{(\sn{27})} \\
    &= \sum_{n=0}^\infty \frac{ q^{2n^2} (-q^{-1};q^2)_n (-q^3;q^2)_n}
{(q^2;q^2)_{2n}} \notag\\
    &\mbox{(\eqref{ram1} with $a=q$; Stanton~\cite{S01})}  \\
    &= \sum_{n=0}^\infty \frac{ q^{2n(n-1)} (-q;q^2)_n}{(q)_{2n} (-1;q^2)_{n+1}}
   \end{align}
\begin{align}
\frac{ f(q^3,q^3) }{ f(-q^2)}&=\sum_{n=0}^\infty \frac{ q^{2n^2}
(-q;q^2)_n}{(q;q^2)_n (q^4;q^4)_n} \notag\\
&\mbox{(B.-M.-S.~\cite[Eq. (2.13)]{BMS07})}\\
\frac{ f(q,q^5) }{ \s(-q^2)}&=\sum_{n=0}^\infty \frac{q^{n(n+1)}
(-q^2;q^2)_{n}}{(q)_{2n+1}}\notag
\\
&\mbox{(Ramanujan~\cite[p. 254, (11.3.5)]{AB05},~\cite[p. 88, Ent. 4.2.13]{AB07}; \sn{28})} \\
\frac{f(q^3,q^3) }{\s(-q^2)} &=\sum_{n=0}^\infty \frac{q^{n(n+1)}
(-1;q^2)_n}{(q)_{2n}} &\mbox{(\sn{48-})}
 \\
\frac{ f(q^2,q^4)}{\ts(-q) } &=\sum_{n=0}^\infty \frac{q^{n^2} (-q;q^2)_{n} } {(q)_{2n}} \notag\\
       &\mbox{(Ramanujan~\cite[p. 178 (3.1)]{A81},~\cite[p. 101, Entry 5.3.2]{AB07}; \sn{29})} \\
\frac{1}{(q,q^4,q^5;q^6)_\infty} &= \sum_{n=0}^\infty \frac{ q^{n(3n-1)/2} (q^2;q^6)_n}{(q)_{3n}}\notag\\
&\mbox{(special case of \eqref{qGauss}; cf. Corteel, Savage, and Sills~\cite[Eq. (1)]{CS08})}\\
\frac{1}{(q,q^2,q^5;q^6)_\infty} &= \sum_{n=0}^\infty \frac{ q^{n(3n+1)/2} (q^4;q^6)_n}{(q)_{3n+1}}\notag\\
&\mbox{(special case of \eqref{qGauss}; cf. Corteel, Savage, and Sills~\cite[Eq. (2)]{CS08})}
\end{align}

\subsection{The Rogers-Selberg Mod 7 Identities}
\begin{align}
\frac{ f(-q,-q^6) }{ f(-q^2) }
    &=  \sum_{n=0}^\infty   \frac{ q^{2n(n+1)} } {(q^2;q^2)_n (-q)_{2n+1}}
   &\mbox{(Rogers~\cite[p. 331 (6)]{R17}; \sn{31})}
   \\
\frac{ f(-q^2,-q^5) }{ f(-q^2) }
    &=  \sum_{n=0}^\infty   \frac{ q^{2n(n+1)} } {(q^2;q^2)_n (-q)_{2n}}
       &\mbox{(Rogers~\cite[p. 342]{R94}; \sn{32})}\\
\frac{ f(-q^3,-q^4) }{ f(-q^2) }
  &=   \sum_{n=0}^\infty   \frac{ q^{2n^2} } {(q^2;q^2)_n (-q)_{2n}}
    &\mbox{(Rogers~\cite[p. 339]{R94}; \sn{33})}
\end{align}

\subsection{Mod 8 Identities}
\subsubsection{Triple Products}
\begin{align}
\frac{ f(-q,-q^7)}{ \ts(-q) }
&=
  \sum_{n=0}^\infty   \frac{q^{n(n+2)} (-q;q^2)_n  } {(q^2;q^2)_n}\notag \\
    &\mbox{(Ramanujan~\cite[p. 37, Entry 1.7.12]{AB07}; \sn{34})}  \\
 \frac{1}{(q^2,q^3,q^7;q^8)_\infty} &=
  \sum_{n=0}^\infty   \frac{q^{n(n+1)} (-q;q^2)_n  } {(q^2;q^2)_n }\notag \\
     &\mbox{(G\"ollnitz~\cite[(2.24)]{G67};
 \eqref{leb} with $a=-q^{1/2}$)} \label{OddGollnitz2}  \\
\frac{1}{(q,q^5,q^6;q^8)_\infty} &=
  \sum_{n=0}^\infty   \frac{q^{n(n+1)} (-q^{-1};q^2)_n  } {(q^2;q^2)_n }\notag
 \\
    &\mbox{(G\"ollnitz~\cite[(2.22)]{G67};
  \eqref{leb} with $a=-q^{-1/2}$)} \label{OddGollnitz1} \\
\frac{ f(-q^3,-q^5)}{\ts(-q) }
&=   \sum_{n=0}^\infty   \frac{q^{n^2} (-q;q^2)_n  } {(q^2;q^2)_n }\notag \\
  &\mbox{(Ramanujan~\cite[p. 36, Entry 1.7.11; p. 88, Entry 4.2.15]{AB07}; \sn{36})} \\
\frac{ f(-q,-q^7) }{ \s(-q) }
    &=   \sum_{n=0}^\infty   \frac{  q^{n(n+3)/2} (-q;q^2)_n (-q)_n}
  {(q)_{2n+1}} \notag\\
   &\mbox{(Ramanujan~\cite[p. 32, Entry 1.7.6]{AB07}; \sn{35})}  \\
\frac{ f(-q^2,-q^6) }{ \s(-q) }
  &=
   \sum_{n=0}^\infty \frac{ q^{n(n+1)/2}
   (-q^2;q^2)_{n}  } {(q;q^2)_{n+1} (q)_n}  \notag \\
   &\mbox{(Ramanujan~\cite[p. 32, Entry 1.7.5]{AB07})} \label{D6s1} \\
\frac{ f(-q^3,-q^5) }{\s(-q) } &=
\sum_{n=0}^\infty   \frac{  q^{n(n+1)/2} (-q;q^2)_n (-q)_n} {(q)_{2n+1}}
\notag\\
        &\mbox{(Ramanujan~\cite[p. 34, Entry 1.7.8]{AB07}; \sn{37})}  \\
\frac{ f(-q^4,-q^4)}{\s(-q)}
     &=   \sum_{n=0}^\infty \frac{ q^{n(n+1)/2} (-1;q^2)_n}{(q)_n (q;q^2)_n}
\notag\\
&\mbox{(Ramanujan~\cite[p. 31, Entry 1.7.4]{AB07})} \\
\frac{ f(q,q^7)}{f(-q^2) }
&=   \sum_{n=0}^\infty   \frac{q^{2n(n+1)} } {(q)_{2n+1}}
   &\mbox{(\sn{38})}\\
\frac{ f(q^3,q^5) }{ f(-q^2) }
  &=   \sum_{n=0}^\infty   \frac{q^{2n^2} } {(q)_{2n}} \notag\\
  &\mbox{(Jackson~\cite[p. 170, 5th Eq.]{J28}; \sn{39})}
  \end{align}
  \begin{align}
  \frac{ f(q,q^7) }{ \ts(-q^2) }
    &=   \sum_{n=0}^\infty \frac{ q^{2n^2} (-q^{-1};q^4)_{n} (-q^5;q^4)_n }
   {(q^8;q^8)_n (q^2;q^4)_n} \notag\\
    &\mbox{(\eqref{ram2} with  $a=q^{3/2}$)} \\
\frac{ f(q^3,q^5) }{ \ts(-q^2) }
    &=   \sum_{n=0}^\infty \frac{ q^{2n^2} (-q;q^2)_{2n} }
 {(q^8;q^8)_n (q^2;q^4)_n} \notag  \\
     &\mbox{(\eqref{ram2} $a=q^{\frac 12}$; Gessel-Stanton~\cite[p. 197, Eq. (7.24)]{GS83})}
 \label{gs8}\\
\frac{ f(-q,-q^7)}{ \s(-q^4)}
   &= \sum_{n=0}^\infty \frac{ q^{2n(n+1)} (-q^4;q^4)_n (q;q^2)_{2n+1} }
    {(q^4;q^4)_{2n+1}}  \notag\\
     &\mbox{(\eqref{q2ndGauss} with $a=q^{3/4}, b=q^{5/4}$); Gessel-Stanton~\cite[p. 197, Eq. (7.25)]{GS83})}  \\
 \frac{ f(-q^3,-q^5)}{ \s(-q^4)}
   &= \sum_{n=0}^\infty \frac{ q^{2n(n+1)} (-q^4;q^4)_n (q^{-1};q^2)_{2n} }
    {(q^4;q^4)_{2n}} \notag \\
     &\mbox{(\eqref{q2ndGauss} with $a=q^{-1/4}, b=q^{1/4}$)}
 \end{align}

 \subsubsection{Quintuple products}
 \begin{align}
\frac{ Q(q^4,q )}{f(-q) }
 &= \sum_{n=0}^\infty \frac{ q^{n^2} (-1;q^2)_n }{(q)_{2n}}
&\mbox{(\eqref{qGauss} with $a=-1$, $c=\sqrt{q}$; \sn{47})}
 \end{align}

\subsection{Mod 9 Identities}
\begin{align}
\frac{ f(-q,-q^8) }{ f(-q^3) }
&=
\sum_{n=0}^\infty   \frac{q^{3n(n+1)} (q)_{3n+1}  } {(q^3;q^3)_n
  (q^3;q^3)_{2n+1}} \\  &\mbox{(Bailey~\cite[p. 422,  Eq. (1.7)]{B47}; \sn{40c})}
  \notag\\
\frac{f(-q^2,-q^7) }{ f(-q^3) }
&=
   \sum_{n=0}^\infty   \frac{q^{3n(n+1)} (q)_{3n} (1-q^{3n+2})  } {(q^3;q^3)_n
  (q^3;q^3)_{2n+1}} \\ &\mbox{(Bailey~\cite[p. 422,  Eq. (1.8)]{B47}; \sn{41c})}
  \notag \\
\frac{ f(-q^4,-q^5) }{ f(-q^3) }
 &=
   \sum_{n=0}^\infty   \frac{q^{3n^2} (q)_{3n}  } {(q^3;q^3)_n
  (q^3;q^3)_{2n}} \\
 &\mbox{(\eqref{ram1} $a=-q^{\frac 12}$; Bailey~\cite[p. 422, Eq. (1.6)]{B47};
  \sn{42c})} \notag\\
\frac{f(-q^3,-q^6) }{\ts(-q) }
&=
\sum_{n=0}^\infty \frac{ q^{n^2} (-1;q^6)_{n} (-q;q^2)_n}
   { (-1;q^2)_{n} (q^2;q^2)_{2n}}  \\
    &\mbox{(M.-S. \cite[p. 767, Eq. (1.13)]{MS07})}  \notag  \\
\frac{ f(-q^3,q^6) }{ \ts(-q) }
&=
1+\sum_{n=1}^\infty \frac
    {q^{n^2} (q^6;q^6)_{n-1} (-q; q^2)_{n}  }
    {(q^2;q^2)_{2n} (q^2;q^2)_{n-1}}
    &\mbox{(\sn{113})}
\end{align}

\subsection{Mod 10 Identities}
\subsubsection{Triple Products}
\begin{align}
\frac{f(-q,-q^9) }{\s(-q) }
  &= \sum_{n=0}^\infty   \frac{  q^{n(n+3)/2} (-q)_{n}} {(q;q^2)_{n+1}
  (q)_{n}} \notag\\
&\mbox{(Rogers~\cite[p. 330 (4), line 2]{R17}; \sn{43})}  \\
\frac{ f(-q^3,-q^7) }{ \s(-q) }
&=
    \sum_{n=0}^\infty   \frac{  q^{n(n+1)/2} (-q)_{n}} {(q;q^2)_{n+1}
  (q)_{n}} \notag\\
 &\mbox{(Rogers~\cite[p. 330 (4), line 1]{R17}; \sn{45})}  \\
\frac{ f(-q^5,-q^5) }{ \s(-q) }
  &= \sum_{n=0}^\infty   \frac{  q^{n(n+1)/2} (-1)_{n}} {(q;q^2)_{n}
  (q)_{n}}  \notag\\
  &\mbox{(Rogers~\cite[p. 330 (4), line 3, corrected]{R17})} \label{C1s1}  \\
\frac{H(q^2)}{\chi(-q)} = \frac{ f(-q^2,-q^8) }{ f(-q) }
  &= \sum_{n=0}^\infty   \frac{  q^{3n(n+1)/2} } {(q;q^2)_{n+1}
  (q)_{n}} \notag\\
&\mbox{(Rogers~\cite[p. 330 (2), line 2]{R17}; \sn{44})}  \\
\frac{G(q^2)}{\chi(-q)} =\frac{ f(-q^4,-q^6) }{ f(-q) }
&=   \sum_{n=0}^\infty   \frac{  q^{n(3n-1)/2} } {(q;q^2)_{n}
  (q)_{n}} \notag\\
&\mbox{(Rogers~\cite[p. 341, Ex. 1]{R94}; \sn{46})}  \\
&=\sum_{n=0}^\infty   \frac{  q^{n(3n+1)/2} } {(q;q^2)_{n+1}
  (q)_{n}}\label{C5p}\notag\\
   &\mbox{(Rogers~\cite[p. 330 (2), line 1]{R17})}
\end{align}

\subsubsection{Quintuple Products}
\begin{align}
\frac{G(q^2)}{\chi(-q)}=\frac{ Q(q^5,-q)}{ \s(-q) }
  &=  \sum_{n=0}^\infty   \frac{ q^{n(3n+1)/2} (-q)_n} {(q)_{2n+1}}
             &\mbox{(\sn{62})} \\
  &=   \sum_{n=0}^\infty   \frac{ q^{n(3n-1)/2} (-q)_n} { (q)_{2n}} \notag\\
    &\mbox{(Rogers~\cite[p. 332 (12), 1st Eq.]{R17})} \\
\frac{H(q^2)}{\chi(-q)} =\frac{ Q(q^5,-q^2) }{\s(-q) }
  &=\sum_{n=0}^\infty   \frac{ q^{3n(n+1)/2} (-q)_n} {(q)_{2n+1}} \notag\\
  &\mbox{(Rogers~\cite[p. 332 (12), 2nd Eq.]{R17}; \sn{63})}
\end{align}

\addtocounter{subsection}{1}

\subsection{Mod 12 identities}
\subsubsection{Triple Products}
\begin{align}
\frac{ f(-q,-q^{11}) }{ f(-q) }
    &=\sum_{n=0}^\infty
     \frac{ q^{n(n+2)}(-q^2;q^2)_{n} (1-q^{n+1}) } {(q)_{2n+2} }
   &\mbox{(\sn{49c})}\\
\frac{ f(-q^2,-q^{10}) }{ f(-q) }
    &=  \sum_{n=0}^\infty   \frac{  q^{n(n+2)} (-q;q^2)_{n}}
  {(q)_{2n+1}} \notag \\
   &\mbox{(Ramanujan~\cite[p. 65, Entry 3.4.4]{AB07}; \sn{50})}\\
\frac{ f(-q^3,-q^9) }{ f(-q) }
    &=  \sum_{n=0}^\infty \frac{ q^{n(n+1)} (-q^2;q^2)_n}{(q)_{2n+1}}
  &\mbox{(\sn{28})}\\
\frac{ f(-q^4,-q^8) }{ f(-q) }
&=
    \sum_{n=0}^\infty \frac{ q^{n(n+1)} (-q;q^2)_n}{(q)_{2n+1}} \notag\\
    &\mbox{(\eqref{qGauss} with $aq=-c=q^{3/2}$; \sn{51})}  \\
\frac{ f(-q^5,-q^7) }{ f(-q) }
&=
\sum_{n=0}^\infty   \frac{q^{n^2} (-q^2;q^2)_{n-1} (1+q^n)  } {(q)_{2n}}
   &\mbox{(\sn{54c})} \\
\frac{ f(-q^6,-q^6) }{ f(-q) }
&=
  \sum_{n=0}^\infty   \frac{  q^{n^2} (-q;q^2)_{n}}
  {(q)_{2n}}\notag\\
&\mbox{(Ramanujan~\cite[Ent. 11.3.1]{AB05}; \sn{29})}  \\
\frac{ f(-q,-q^{11}) }{ f(-q^4) }
   &=
     \sum_{n=0}^\infty   \frac{ q^{4n(n+1)} (q;q^2)_{2n+1} }
   {(q^4;q^4)_{2n+1}} &\mbox{(\sn{55})} \\
\frac{ f(-q^5,-q^7) }{ f(-q^4) }
   &=   \sum_{n=0}^\infty   \frac{ q^{4n^2} (q;q^2)_{2n}} {(q^4;q^4)_{2n}}\
 \notag \\
   &\mbox{(\eqref{ram1} with $a=-\sqrt{q}$; \sn{53})} \\
\frac{ f(-q^3,-q^9) }{\ts(-q) }
&=
   \sum_{n=0}^\infty \frac{ q^{2n(n+1)} (-q;q^2)_n}{(q;q^2)_{n+1} (q^4;q^4)_n}
   &\mbox{(\sn{27})}\\
\frac{ f(-q^4,-q^8) }{ \ts(-q) }
&=
  \sum_{n=0}^\infty \frac{ q^{n(2n-1)} (-q;q^2)_n}{ (q^2;q^2)_{n} (q^2;q^4)_n}
   &\mbox{(\sn{52})}
\end{align}
\begin{align}
\frac{ f(-q^5,-q^7) }{ \ts(-q^3) }
&=
  \sum_{n=0}^\infty \frac{ q^{3n^2} (q^2;q^2)_{3n} }
  {(q^{12};q^{12})_n (q^3;q^3)_{2n} } \notag\\
     &\mbox{(\eqref{ram2} $a=-q^{1/3}$; Dyson~\cite[p. 9, Eq. (7.5)]{B49})} \\
\frac{ f(-q,-q^{11})}{ \ts(-q^3) }
&=
 \sum_{n=0}^\infty \frac{ q^{3n^2}(q^2;q^2)_{3n+1} }
  {(q^{12};q^{12})_n (q^3;q^3)_{2n} (q^{6n}-q^2) } \notag\\
     &\mbox{(Dyson~\cite[p. 9, Eq. (7.6)]{B49})}
  \end{align}
  \begin{align}
\frac{ f(q,q^{11}) }{ f(-q) }
&=
     \sum_{n=0}^\infty   \frac{  q^{n(n+2)}(-q)_{n} } {(q;q^2)_{n+1}
  (q)_{n+1} } &\mbox{(\sn{56})} \\
 \frac{ f(q^3, q^9)}{f(-q)} &= \frac{1+q^3}{(1-q)(1-q^2)} +
  \sum_{n=1}^\infty \frac{ q^{n(n+2)} (-q)_{n-1} (-q)_{n+2} }{ (q)_{2n+2} }
  \qquad\mbox{\cite[Eq. (2.10)]{MSZ08}}
  \end{align}
  \begin{align}
\frac{ f(q^5,q^7) }{ f(-q) }
&= 1 + \sum_{n=1}^\infty   \frac{  q^{n^2} (-q)_{n-1} } {(q;q^2)_{n}
          (q)_{n} } &\mbox{(\sn{58c})}  \\
\frac{ f(q,q^{11}) }{ f(-q^4)  }
&=  \sum_{n=0}^\infty   \frac{q^{4n(n+1)} (-q;q^2)_{2n+1}}
  {(q^4;q^4)_{2n+1}}&\mbox{(\sn{57})} \\
\frac{f(q^3,q^9) }{f(-q^4) }
&= \sum_{n=0}^\infty \frac{ q^{n(n+2)} (-q;q^2)_n}
  {(q^4;q^4)_n} & \mbox{(Ramanujan~\cite[p. 105, Entry 5.3.7]{AB07})}\\
\frac{ f(q^5,q^7) }{ f(-q^4) }
   &=   \sum_{n=0}^\infty   \frac{ q^{4n^2} (-q;q^2)_{2n}} {(q^4;q^4)_{2n}}
  &\mbox{(\eqref{ram1} with $a=\sqrt{q}$; \sn{53-})}
\end{align}

\subsubsection{Quintuple Products}
\begin{align}
\frac{ Q(q^6,-q)}{ \s(-q)}
&=
 \sum_{n=0}^\infty \frac{ q^{n(n+1)/2} (-1;q^3)_n (-q)_n}{ (q)_{2n} (-1)_n }  
 \\
&  \mbox{(M.-S.~\cite[p. 767, Eq. (1.22)]{MS07})}\notag\\
\frac{ Q(q^6, -q^2)}{ \s(-q)}
&=
  \sum_{n=0}^\infty \frac{ q^{n(n+1)/2} (-q^3;q^3)_n}{(q)_{2n+1}} \\
 & \mbox{(M.-S.~\cite[p. 768, Eq. (1.24)]{MS07})}\notag\\
\frac{ Q(q^6,q)}{ \s(-q)}
&=
 1+\sum_{n=1}^\infty \frac{ q^{n(n+1)/2} (q^3;q^3)_{n-1} (-q)_n}{ (q)_{2n} (q)_{n-1} } \\
& \mbox{(M.-S.~\cite[p. 768, Eq. (1.27)]{MS07})}\notag\\
\frac{ Q(q^6, q^2)}{ \s(-q)}
&=
  \sum_{n=0}^\infty \frac{ q^{n(n+1)/2} (q^3;q^3)_n (-q)_n}{(q)_{2n+1} (q)_n}\\
  &\mbox{(Dyson~\cite[p. 434, Eq. (D2)]{MS07})}\notag
\end{align}

\addtocounter{subsection}{1}

\subsection{Mod 14 Identities}
\subsubsection{Triple Products}
\begin{align}
\frac{ f(-q^2,-q^{12}) }{ f(-q) }
&=
 \sum_{n=0}^\infty   \frac{ q^{n(n+2)} } {(q;q^2)_{n+1} (q)_{n} }
     &\mbox{(Rogers~\cite[p. 329 (1)]{R17}; \sn{59})} \\
\frac{ f(-q^4, -q^{10})}{ f(-q) }
   &=    \sum_{n=0}^\infty   \frac{ q^{n(n+1)} } {(q;q^2)_{n+1} (q)_{n} } &
   \mbox{(Rogers~\cite[p. 329 (1)]{R17}; \sn{60})} \\
\frac{ f(-q^6,-q^8) }{f(-q) }
&=
\sum_{n=0}^\infty   \frac{ q^{n^2} } {(q;q^2)_{n} (q)_{n} } &
    \mbox{(Rogers~\cite[p. 341, Ex. 2]{R94}; \sn{61})}
\end{align}

\subsubsection{Quintuple Products}
\begin{align}
\frac{Q(q^7,-q) }{\s(-q) }
&=
\sum_{n=0}^\infty   \frac{q^{n(n+1)/2} (-q)_n}  { (q)_{2n}  }
    &\mbox{(\sn{81})} \\
\frac{ Q(q^7,-q^2) }{\s(-q) }
&=
 \sum_{n=0}^\infty   \frac{q^{n(n+1)/2} (-q)_n}  { (q)_{2n+1}  } 
 &\mbox{(\sn{80})} \\
\frac{ Q(q^7,-q^3) }{ \s(-q) }
 &=   \sum_{n=0}^\infty   \frac{q^{n(n+3)/2} (-q)_n}  { (q)_{2n+1}  } 
    &\mbox{(\sn{82})}
 \end{align}

\subsection{Mod 15 Identities}
\begin{align}
\frac{ f(-q,-q^{14})}{f(-q^5)}  &=
  1-\sum_{n=1}^\infty \frac{  q^{5n^2-4} (q;q^5)_{n-1} (q^4;q^5)_{n+1}}{ (q^5;q^5)_{2n} }
  &\mbox{(\eqref{ram1} with $a=-q^{13/5}$)}\\
\frac{ f(-q^2,-q^{13})}{f(-q^5)}  &=
  1-\sum_{n=1}^\infty \frac{  q^{5n^2-3} (q^2;q^5)_{n-1} (q^3;q^5)_{n+1}}{ (q^5;q^5)_{2n} }
  &\mbox{(\eqref{ram1} with $a=-q^{11/5}$)}\\
\frac{ f(-q^3,-q^{12})}{f(-q^5)}  &=
  1-\sum_{n=1}^\infty \frac{  q^{5n^2-2} (q^3;q^5)_{n-1} (q^2;q^5)_{n+1}}{ (q^5;q^5)_{2n} }
  &\mbox{(\eqref{ram1} with $a=-q^{9/5}$)}\\
\frac{ f(-q^4,-q^{11})}{f(-q^5)}  &=
  1-\sum_{n=1}^\infty \frac{  q^{5n^2-1} (q;q^5)_{n-1} (q;q^5)_{n+1}}{ (q^5;q^5)_{2n} }
  &\mbox{(\eqref{ram1} with $a=-q^{7/5}$)}\\
\frac{ f(-q^6,-q^9)}{f(-q^5)}  &=
  \sum_{n=0}^\infty \frac{  q^{5n^2} (q;q^5)_n (q^4;q^5)_n}{ (q^5;q^5)_{2n} }
  &\mbox{(\eqref{ram1} with $a=-q^{3/5}$)}\\
\frac{ f(-q^7,-q^8)}{f(-q^5)}  &=
  \sum_{n=0}^\infty \frac{  q^{5n^2} (q^2;q^5)_n (q^3;q^5)_n}{ (q^5;q^5)_{2n} }
  &\mbox{(\eqref{ram1} with $a=-q^{1/5}$)}
\end{align}

\subsection{Mod 16 Identities}
\subsubsection{Triple Products}
\begin{align}
\frac{ f(-q^2,-q^{14}) }{ \ts(-q) }
   &= \sum_{n=0}^\infty   \frac{  q^{n(n+2)} (-q;q^2)_n (-q^4;q^4)_n}
  {(-q^2;q^2)_{n+1} (q^2;q^4)_{n+1} (q^2;q^2)_n} &\mbox{(\sn{68})} \\
\frac{ f(-q^4,-q^{12})  }{\ts(-q) }
   &=  \sum_{n=0}^\infty   \frac{ q^{n(n+2)} (-q^2;q^4)_n }
  {(q;q^2)_{n+1} (q^4;q^4)_n } &\mbox{(\sn{70})}  \\
\frac{ f(-q^6,-q^{10}) }{ \ts(-q) }
   &=
      \sum_{n=0}^\infty\frac{ q^{n^2} (-q^4;q^4)_{n-1}  }
    {(q)_{2n}  (-q^2;q^2)_{n-1} }&\mbox{(\sn{71})}
    \end{align}
    \begin{multline}
\frac{ f(-q^8,-q^8) }{ \ts(-q) }
   =  \sum_{n=0}^\infty   \frac{ q^{n^2} (-q^2;q^4)_n }
  {(q;q^2)_{n}  (q^4;q^4)_n} \\
 \mbox{(S.~\cite[Eq. (5.5)]{S07};
\eqref{qBailey} $b=ic = iq^{1/2}$) }
\end{multline}
\begin{align}
\frac{ f(q^2,q^{14}) }{ \ts(-q) }
   &=
     \sum_{n=0}^\infty   \frac{(-q^2;q^2)_n  q^{n(n+2)}} {(q)_{2n+2} }
\qquad\mbox{(\sn{69})}  \\
  &=   \sum_{n=0}^\infty   \frac{ q^{n(n+1)/2 + n} (-q)_n} {(q)_{n+1}} 
     \qquad\mbox{(Gessel-Stanton~\cite[p. 196, Eq. (7.15)]{GS83})}  \\
   &= 1+\frac{q}{(1-q)(1-q^2)} + \sum_{n=2}^\infty \frac{ q^{n^2-2} (-q^2;q^2)_{n-2} (1+q^{2n+2}) }
   {(q)_{2n}} \\
   &\mbox{(M.-S.-Z.~\cite{MSZ08})}\notag\\
\frac{ f(q^6,q^{10}) }{ \ts(-q) }
   &=        1 + \sum_{n=1}^\infty   \frac{q^{n^2} (-q)_{2n-1}  }
  {(q^2;q^4)_{n} (q^2;q^2)_n }
   \qquad\mbox{(\sn{72})}  \\
&=    1+\sum_{n=1}^\infty \frac{ q^{n(n+1)/2} (-q)_{n-1} }{(q)_n}
    \qquad\mbox{(Gessel-Stanton~\cite[p. 196, Eq. (7.13)]{GS83})}
   \end{align}
\subsubsection{Quintuple Products}
\begin{align}
\frac{ Q(q^8,-q) }{ f(-q) }
&=
\sum_{n=0}^\infty \frac{ q^{2n^2} }{(q)_{2n}} &\mbox{(\sn{83})}  \\
\frac{ Q(q^8,-q^2)}{f(-q) }
&=
   \sum_{n=0}^\infty \frac{ q^{n(2n+1)} }{(q)_{2n+1}} &\mbox{(\sn{84})}
\\
&=    \sum_{n=0}^\infty \frac{ q^{n(2n-1)} }{(q)_{2n}}
&\mbox{(Starcher~\cite[p. 809, Eq. (3.29)]{S31}; \sn{85})}  \\
\frac{ Q(q^8,-q^3)}{ f(-q) }
&=
     \sum_{n=0}^\infty \frac{ q^{2n(n+1)} }{(q)_{2n+1}} &\mbox{(\sn{86})}
\end{align}

\addtocounter{subsection}{1}

\subsection{Mod 18 Identities}
\subsubsection{Triple Products}
\begin{align}
\frac{ f(-q^3,-q^{15}) }{ \s(-q) }
  &= \sum_{n=0}^\infty   \frac{q^{n(n+3)/2} (q^3;q^3)_{n} (-q)_{n+1} }
  {(q)_{2n+2} (q)_{n} } \notag\\
   &\mbox{(Dyson~\cite[p. 434 (D1)]{B47}; \sn{76})} \\
\frac{ f(-q^6,-q^{12}) }{\s(-q) }
  &=\sum_{n=0}^\infty \frac{ q^{n(n+1)/2} (-q)_n (q^3;q^3)_n}
{ (q)_n  (q)_{2n+1}} \notag\\
    &\mbox{(Dyson~\cite[p. 434 (D2)]{B47}; \sn{77c})} \\
\frac{ f(-q^9,-q^9) }{ \s(-q) }
  &=     1 + \sum_{n=1}^\infty   \frac{ q^{n(n+1)/2} (q^3;q^3)_{n-1}
(-1)_{n+1} }
  {(q)_{2n} (q)_{n-1} }\notag\\ &\mbox{(Dyson~\cite[p. 434 (D3)]{B47}; \sn{78})}
\end{align}
\subsubsection{Quintuple Products}
\begin{align}
\frac{Q(q^9,-q)}{f(-q)} &=
   \sum_{n=0}^\infty \frac{ q^{n(n+1)} (-1;q^3)_n }{(-1)_n (q)_{2n}}
&\mbox{(Loxton~\cite[p. 158, Eq. (P12)]{L84}}\\
\frac{Q(q^9,-q^2)}{f(-q)} &=
   \sum_{n=0}^\infty \frac{ q^{n^2} (-1;q^3)_n }{(-1)_n (q)_{2n}}
&\mbox{(Loxton~\cite[p. 159, (P12 bis))]{L84}}\\
  \frac{Q(q^9,-q^3)}{f(-q)} &=
   \sum_{n=0}^\infty \frac{ q^{n(n+1)} (-q^3;q^3)_n }{(-q)_n (q)_{2n+1}}
&\mbox{(M.-S.~\cite[p. 766, Eq. (1.5)]{MS07})}\\
\frac{Q(q^9,-q^4)}{f(-q)} &=
   \sum_{n=0}^\infty \frac{ q^{n(n+2)} (-q^3;q^3)_n (1-q^{n+1}) }
{(-q)_n (q)_{2n+2}}
&\mbox{(M.-S.~\cite[p. 767, Eq. (1.6)]{MS07})}\\
\frac{Q(q^9,q)}{f(-q)} &=
 1+  \sum_{n=1}^\infty \frac{ q^{n^2} (q^3;q^3)_{n-1} (2+q^n) }{(q)_{n-1} (q)_{2n}}
&\mbox{(M.-S.~\cite[p. 766, Eq. (1.7)]{MS07})}\\
\frac{Q(q^9,q^2)}{f(-q)} &=
   1+\sum_{n=1}^\infty \frac{ q^{n^2} (q^3;q^3)_{n-1} (1+2q^n) }{(q)_{n-1} (q)_{2n}}
&\mbox{(M.-S.~\cite[p. 766, Eq. (1.8)]{MS07})}\\
  \frac{Q(q^9,q^3)}{f(-q)} &=
   \sum_{n=0}^\infty \frac{ q^{n(n+1)} (q^3;q^3)_n }{(q)_n (q)_{2n+1}}
&\mbox{(Dyson~\cite[p. 433, Eq. (B3)] {B47})}\\
\frac{Q(q^9,q^4)}{f(-q)} &=
   \sum_{n=0}^\infty \frac{ q^{n(n+2)} (q^3;q^3)_n }
{(q)_n^2 (q^{n+2})_{n+1}}
&\mbox{(M.-S.~\cite[p. 766, Eq. (1.10)]{MS07})}
\end{align}

\addtocounter{subsection}{1}

\subsection{Mod 20 Identities}
\subsubsection{Triple Products}
\begin{align}
\frac{H(q^4)}{\chi(-q)} =\frac{ f(-q^4,-q^{16}) }{\ts(-q)  }
&= \sum_{n=0}^\infty   \frac{q^{n(n+2)}} {(q)_{2n+1}}
   &\mbox{(Rogers~\cite[p. 330 (3), 2nd Eq.]{R17})}
   \label{RogMod20}\\
\frac{G(q^4)}{\chi(-q)} =\frac{ f(-q^8, -q^{12})  }{ \ts(-q) }
&=   \sum_{n=0}^\infty   \frac{q^{n^2}} {(q)_{2n}}
     &\mbox{(Rogers~\cite[p. 330 (3), 1st Eq.]{R94}; \sn{79})}
\end{align}
\subsubsection{Quintuple Products}
\begin{align}
\frac{G(-q)}{\chi(-q)} =
\frac{ Q(q^{10},-q) }{ f(-q) }
&=
\sum_{n=0}^\infty \frac{ q^{n(n+1)} }{ (q)_{2n} }  \notag\\
       &\mbox{(Rogers~\cite[p. 332, Eq. (13)]{R94}; \sn{99})}
  \\
\frac{G(q^4)}{\chi(-q)} =\frac{ Q(q^{10},-q^2) }{ f(-q) }
&=
 \sum_{n=0}^\infty \frac{ q^{n^2} }{ (q)_{2n} } \notag\\
 &\mbox{(Rogers~\cite[p. 331 above (5)]{R94}; \sn{98})} \\
\frac{H(-q)}{\chi(-q)} =\frac{ Q(q^{10},-q^3) }{ f(-q) }
&=
 \sum_{n=0}^\infty \frac{ q^{n(n+1)} }{ (q)_{2n+1} } \notag\\
  & \mbox{(Rogers~\cite[p. 331, Eq. (6)]{R94}; \sn{94})} \\
\frac{H(q^4)}{\chi(-q)} =\frac{ Q(q^{10},-q^4)}{ f(-q) }
&=
  \sum_{n=0}^\infty \frac{ q^{n(n+2)} }{ (q)_{2n+1} } \notag\\
   &\mbox{(Rogers~\cite[p. 331, Eq. (7)]{R94}; \sn{96})}\\
\frac{ Q(q^{10},-q) }{ \ts(-q) }
&=
 \sum_{n=0}^\infty \frac{ q^{3n^2} (-q;q^2)_n }{ (q^2;q^2)_{2n} }
 &\mbox{(\sn{100c})} \\
\frac{ Q(q^{10},-q^3) }{ \ts(-q) }
&=
   \sum_{n=0}^\infty \frac{ q^{n(3n-2)} (-q;q^2)_n }{ (q^2;q^2)_{2n} }
   &\mbox{(\sn{95})} \\
  &=   \sum_{n=0}^\infty \frac{ q^{n(3n+2)} (-q;q^2)_{n+1} }{ (q^2;q^2)_{2n+1} }
    &\mbox{(\sn{97c})}
\end{align}

\setcounter{subsection}{23}
\subsection{Mod 24 Identities}
\begin{align}
\frac{ Q(q^{12},-q) }{\ts(-q) } &=
  \sum_{n=0}^\infty \frac{ q^{n(n+2)} (-q;q^2)_n (-1;q^6)_n }{ (q^2;q^2)_{2n}
  (-1;q^2)_n } \notag \\
   &\mbox{(M.-S.~\cite[p. 767, Eq. (1.12)]{MS07})} \\
\frac{ Q(q^{12},-q^2)}{\ts(-q) }
&=
 \sum_{n=0}^\infty \frac{ q^{n^2} (-q^3;q^6)_n}{ (q^2;q^2)_{2n} } \notag\\
        &\mbox{(Ramanujan~\cite[p. 105, Entry 5.3.8]{AB07})} \\
\frac{Q(q^{12},-q^3) }{\ts(-q) }
&=
\sum_{n=0}^\infty \frac{ q^{n^2} (-q;q^2)_n (-1;q^6)_{n} }
{ (-1;q^2)_{n} (q^2;q^2)_{2n}}  \notag\\
    &\mbox{(M.-S.~\cite[p. 767, Eq. (1.13)]{MS07})} \\
\frac{Q(q^{12},-q^4) }{\ts(-q) }
&=
 \sum_{n=0}^\infty \frac{ q^{n(n+2)} (-q^3;q^6)_n}{ (q^2;q^2)_{2n} (1-q^{2n+1})}
 \notag\\
     &\mbox{(M.-S.~\cite[p. 767, Eq. (1.14)]{MS07})} \\
\frac{ Q(q^{12},-q^5) }{\ts(-q)} &=
   \sum_{n=0}^\infty \frac{ q^{n(n+2)} (-q;q^2)_{n+1} (-q^6;q^6)_n (1-q^{2n+2})}
{(-q^2;q^2)_n (q^2;q^2)_{2n+2}}   \notag\\
&\mbox{(M.-S.~\cite[p. 767, Eq. (1.15)]{MS07})} \\
\frac{ Q(q^{12},q)}{\ts(-q)}
&=1+ \sum_{n=1}^\infty \frac{ q^{n^2} (-q;q^2)_n (q^6;q^6)_{n-1} (2+q^{2n}) }
  { (q^2;q^2)_{2n} (q^2;q^2)_{n-1} }\notag\\
  &\mbox{(M.-S.~\cite[p. 767, Eq. (1.17)]{MS07})}\\
\frac{ Q(q^{12},q^2) }{ \ts(-q) }
&=
 \sum_{n=0}^\infty \frac{ q^{n^2} (q^3;q^6)_n}{ (q;q^2)_n^2 (q^4;q^4)_{n} } \notag\\
        &\mbox{(Ramanujan~\cite[p. 105, Entry 5.3.9]{AB07})}\\
\frac{ Q(q^{12},q^3)}{\ts(-q)}
&=1+ \sum_{n=1}^\infty \frac{ q^{n^2} (-q;q^2)_n (q^6;q^6)_{n-1} (1+2q^{2n}) }
  { (q^2;q^2)_{2n} (q^2;q^2)_{n-1} } \notag\\
  &\mbox{(M.-S.~\cite[p. 767, Eq. (1.18)]{MS07})}\\
\frac{ Q(q^{12},q^4) }{ \ts(-q) }
&=
  \sum_{n=0}^\infty \frac
    {q^{n(n+2)}  (q^3;q^6)_{n} (-q; q^2)_{n+1}  }
    {(q^2;q^2)_{2n+1} (q;q^2)_n}  &\mbox{(\sn{110c})} \\
\frac{ Q(q^{12}, q^5) }{ \ts(-q) }
&=
   \sum_{n=0}^\infty \frac
    {q^{n(n+2)} (q^6;q^6)_{n} (-q; q^2)_{n+1} (1-q^{2n+2}) }
    {(q^2;q^2)_{2n+2} (q^2;q^2)_n}
    &\mbox{(\sn{108c})} \\
\frac{Q(q^{12},-q)}{\s(-q^2)}
&= \sum_{n=0}^\infty \frac{  q^{n(n+1)} (-q^2;q^2)_n (-q^3;q^6)_n}{ (q)_{2n} (-q)_{2n+1}
(-q;q^2)_n } \notag\\
 &\mbox{(M.-S.~\cite[p. 767, Eq. (1.21)]{MS07})}\\
\frac{Q(q^{12},-q^3)}{\s(-q^2)}
&= \sum_{n=0}^\infty \frac{  q^{n(n+1)} (-q^2;q^2)_n (-q^3;q^6)_n}{ (q^2;q^2)_{2n+1}
(-q;q^2)_n } \notag\\
  &\mbox{(M.-S.~\cite[p. 767, Eq. (1.23)]{MS07})}\\
\frac{Q(q^{12},-q^5)}{\s(-q^2)}
&= \sum_{n=0}^\infty \frac{  q^{n(n+3)} (-q^2;q^2)_n (-q^3;q^6)_n}{ (q^2;q^2)_{2n+1}
(-q;q^2)_n } \notag\\
 &\mbox{(M.-S.~\cite[p. 768, Eq. (1.25)]{MS07})}\\
\frac{Q(q^{12},q)}{\s(-q^2)}
&= \sum_{n=0}^\infty \frac{  q^{n(n+1)} (-q^2;q^2)_n (q^3;q^6)_n}{ (q)_{2n+1} (-q)_{2n}
(q;q^2)_n }  \notag\\
&\mbox{(M.-S.~\cite[p. 768, Eq. (1.26)]{MS07})}\\
\frac{ Q(q^{12},q^3) }{ \s(-q^2) }
&=
  \sum_{n=0}^\infty \frac
    { q^{n(n+1)} (q^3;q^6)_{n} (-q^2; q^2)_{n} }{(q^2;q^2)_{2n+1} (q;q^2)_n}
    &\mbox{(\sn{107})}\\
\frac{Q(q^{12},q^5) }{  \s(-q^2) }
 &=
        \sum_{n=0}^\infty \frac
    { q^{n(n+3)} (q^3;q^6)_{n} (-q^2; q^2)_{n} }{(q^2;q^2)_{2n+1} (q;q^2)_n} \notag\\
    &\mbox{(M.-S.~\cite[p. 768, Eq. (1.30)]{MS07})}
 \end{align}

\setcounter{subsection}{26}
\subsection{Dyson's Mod 27 Identities}
\begin{align}
\frac{ f(-q^3,-q^{24})  }{f(-q) }
&=
   \sum_{n=0}^\infty   \frac{ q^{n(n+3)} (q^3;q^3)_{n} }
  { (q)_{2n+2} (q)_{n} }  &\mbox{(Dyson~\cite[p. 434 (B1)]{B47}; \sn{90})}\\
\frac{ f(-q^6,-q^{21}) }{f(-q) }
&=    \sum_{n=0}^\infty   \frac{ q^{n(n+2)} (q^3;q^3)_{n} }
  { (q)_{2n+2} (q)_{n} }  &\mbox{(Dyson~\cite[p. 434 (B2)]{B47}; \sn{91})}\\
\frac{ f(-q^9,-q^{18}) }{ f(-q) }
&=
     \sum_{n=0}^\infty   \frac{ q^{n(n+1)} (q^3;q^3)_{n} }
  { (q)_{2n+1} (q)_{n} }  &\mbox{(Dyson~\cite[p. 434 (B3)]{B47}; \sn{92})}\\
\frac{ f(-q^{12},-q^{15}) }{ f(-q) }
&=
      1+ \sum_{n=1}^\infty   \frac{ q^{n^2} (q^3;q^3)_{n-1} }
  { (q)_{2n-1} (q)_{n} }  &\mbox{(Dyson~\cite[p. 434 (B4)]{B47}; \sn{93})}
\end{align}

\subsection{Mod 28 Identities}
\begin{align}
\frac{ Q(q^{14},-q) }{ \ts(-q) }
&=
\sum_{n=0}^\infty \frac  {q^{n(n+2)} (-q; q^2)_{n}  } {(q^2;q^2)_{2n} }
      &\mbox{(\sn{118})}  \\
\frac{ Q(q^{14},-q^3) }{ \ts(-q) }
&=
\sum_{n=0}^\infty \frac  {q^{n^2} (-q; q^2)_{n}  } {(q^2;q^2)_{2n} }
     &\mbox{(\sn{117})}   \\
\frac{ Q(q^{14},-q^5) }{ \ts(-q) }
&=
       \sum_{n=0}^\infty \frac  {q^{n(n+2)} (-q; q^2)_{n+1}  } {(q^2;q^2)_{2n+1} }
       &\mbox{(\sn{119})}
\end{align}

\setcounter{subsection}{31}
\subsection{Mod 32 Identities}
\begin{align}
\frac{ Q(q^{16},-q^2)}{f(-q)}
&= 1 + \sum_{n=1}^\infty \frac{ q^{n^2} (-q^2;q^2)_{n-1}}{(q)_{2n}}
&\mbox{(\sn{121})}\\
\frac{ Q(q^{16}, -q^4)}{f(-q)}
 &= \sum_{n=0}^\infty \frac{ q^{n(n+1)} (-q;q^2)_n }{(q)_{2n+1}}
&\mbox{(\eqref{qGauss} with $aq=-c=q^{3/2}$) } \\
\frac{ Q(q^{16},-q^6)}{f(-q)}
&=  \sum_{n=0}^\infty \frac{ q^{n(n+2)} (-q^2;q^2)_{n}}{(q)_{2n+2}} &\mbox{(\sn{123})}
\end{align}

\setcounter{subsection}{35}
\subsection{Mod 36 Identities}
\subsubsection{Triple Products}
\begin{align}
\frac{ f(-q^3,-q^{33})}{\ts(-q)}
&=
  \sum_{n=0}^\infty \frac
    {q^{n(n+4)} (q^6;q^6)_{n} (-q; q^2)_{n+1}  }
    {(q^2;q^2)_{2n+2} (q^2;q^2)_{n}}     \notag\\
    &\mbox{(Dyson~\cite[p. 434, Eq. (C1)]{B47}; \sn{116})}  \\
\frac{ f(-q^9,-q^{27})}{\ts(-q)}
&=
  \sum_{n=0}^\infty \frac
    {q^{n^2} (q^6;q^6)_{n} (-q; q^2)_{n+1}  }
    {(q^2;q^2)_{2n+2} (q^2;q^2)_{n}}     \notag\\
      &\mbox{(Dyson~\cite[p. 434, Eq. (C2)]{B47}; \sn{115c})} \\
\frac{ f(-q^{12},-q^{24})}{\ts(-q)}
&=
  \sum_{n=0}^\infty \frac
    {q^{n(n+2)}  (q^3;q^6)_{n} (-q; q^2)_{n+1}  }
    {(q^2;q^2)_{2n+1} (q;q^2)_n} \notag\\ &\mbox{(\sn{110c}}) \\
\frac{ f(-q^{15},-q^{21})}{\ts(-q)}
&=
1+\sum_{n=1}^\infty \frac
    {q^{n^2} (q^6;q^6)_{n-1} (-q; q^2)_{n}  }
    {(q^2;q^2)_{2n-1} (q^2;q^2)_{n}}  \notag\\   &\mbox{(Dyson~\cite[p. 434, Eq. (C3)]{B47}; \sn{114})}
\end{align}
\subsubsection{Quintuple Products}
\begin{align}
\frac{ Q(q^{18}, q)}{f(-q^2)}
&= \sum_{n=0}^\infty \frac{ q^{2n(n+1)} (q^3;q^6)_n }{(q^2;q^2)_{2n}(q;q^2)_{n+1}  }
&\mbox{(M.-S.~\cite{MS08})} \\
\frac{ Q(q^{18}, q^{3})}{f(-q^2)}
&=
  \sum_{n=0}^\infty \frac{ q^{2n^2} (q^3;q^6)_n }{(q^2;q^2)_{2n}(q;q^2)_n  }
  &\mbox{(Ramanujan~\cite[p. 104, Ent. 5.3.4]{AB07})}\\
\frac{ Q(q^{18}, q^{5})}{f(-q^2)}
&=
    \sum_{n=0}^\infty \frac
    {q^{2n(n+1)} (q^3; q^6)_{n}  }{(q^2;q^2)_{2n+1} (q;q^2)_n }
   &\mbox{(\sn{124})}  \\
\frac{ Q(q^{18}, q^{7})}{f(-q^2)}
&=
   \sum_{n=0}^\infty \frac
    {q^{2n(n+2)} (q^3; q^6)_{n}  }{(q^2;q^2)_{2n+1} (q;q^2)_n }
    &\mbox{(\sn{125})}
\end{align}

\section{$q$-series expansions of sums of ratios of theta functions}
Slater's list contains quite a few identities between $q$-series and \emph{sums} of
two or more ratios of theta functions.  
\setcounter{subsection}{7}
\subsection{Mod 8 Identity}
\begin{multline}
\frac{ f(-q,-q^7)+ f(-q^3,-q^5)}{\s(-q^4)} =
  \sum_{n=0}^\infty \frac{ q^{2n(n-1)} (-q^4;q^4)_n (q;q^2)_{2n} } { (q^4;q^4)_{2n} }\\
     \mbox{(M.-S.-Z.~\cite[Eq. (2.4)]{MSZ08})}
\end{multline}

\addtocounter{subsection}{1}
\subsection{Mod 10 Identity}
\begin{multline}  
\frac{ f(-q^5,-q^5) - f(-q,-q)}{2q\s(-q)} =\frac{f(-q^3,-q^7)  f(-q^4,-q^{16})_\infty}{\s(-q) f(-q^8,-q^{12})}
 \\  = \sum_{n=0}^\infty \frac{ q^{n(n+3)/2} (-q)_n }{ (q;q^2)_{n+1} (q)_{n+1} }
   \mbox{  (B.-M.-S.~\cite[Eq. (2.22)]{BMS07})}
  \end{multline}

 \addtocounter{subsection}{1}
\subsection{Mod 12 Identities}
\begin{align}
\frac{f(-q^5,-q^7) + q f(-q,-q^{11})}{f(-q)}
&=\sum_{n=0}^\infty \frac{ q^{n^2} (-1;q^2)_n }{(q)_{2n}}
&\mbox{(\sn{47})}\\
\frac{f(-q^5,-q^7) - q f(-q,-q^{11})}{f(-q)}
&=
 \sum_{n=0}^\infty   \frac{ q^{n(n+1)}(-1;q^2)_{n}  } {(q)_{2n}}
     &\mbox{(\sn{48})}
\end{align}

\setcounter{subsection}{14}
\subsection{Mod 15 Identities}
\begin{align}
\frac{ f(q^7,q^{8}) - q f(q^2,q^{13}) }{\s(-q)}
  &=  \sum_{n=0}^\infty   \frac{ q^{n(3n+1)/2} (-q)_n} {(q)_{2n+1}}
             &\mbox{(\sn{62})} \\
  &=   \sum_{n=0}^\infty   \frac{ q^{n(3n-1)/2} (-q)_n} { (q)_{2n}}
    &\mbox{(Rogers~\cite[p. 332]{R17})}  \label{A5s2}
\end{align}

\begin{equation}
\frac{ f(q^4,q^{11}) - q f(q,q^{14}) }{\s(-q)}
  =\sum_{n=0}^\infty   \frac{ q^{3n(n+1)/2} (-q)_n} {(q)_{2n+1}}
  \mbox{(Rogers~\cite[p. 332]{R17}; \sn{63}) }
\end{equation}

\subsection{Mod 16 Identities}
\begin{align}
\frac{ f(q^6,q^{10}) + q f(q^2, q^{14}) }{\s(-q^2) }
&=
   \sum_{n=0}^\infty   \frac{ q^{n(n+1)} (-q)_{2n} } {(q)_{2n+1} (q^4;q^4)_n}
     &\mbox{(\sn{64})} \\
\frac{ f(-q^6,-q^{10}) + q f(-q^2, -q^{14}) }{\s(-q^2) }
&=
   \sum_{n=0}^\infty   \frac{ q^{n(n+1)} (-q^2;q^4)_n} {(q)_{2n+1} (-q;q^2)_n}
     &\mbox{(\sn{65})}\\
\frac{ f(-q^6,-q^{10}) + q f(-q^2, -q^{14}) }{\ts(-q) }
&=
 \sum_{n=0}^\infty   \frac{q^{n^2} (-1;q^4)_n (-q;q^2)_n } {(q^2;q^2)_{2n}}
    &\mbox{(\sn{66})}  \\
\frac{ f(-q^6,-q^{10}) - q f(-q^2, -q^{14}) }{\ts(-q) }
&=
  \sum_{n=0}^\infty   \frac{q^{n(n+2)} (-1;q^4)_n (-q;q^2)_n } {(q^2;q^2)_{2n}}
 &\mbox{(\sn{67})}
\end{align}

\addtocounter{subsection}{1}
\subsection{Mod 18 Identities}
\begin{equation}
\frac{ f(-q^6, -q^{12}) + f(-q^9,-q^9)}{\s(-q)}
= 2 + \sum_{n=1}^\infty \frac{ q^{n(n-1)/2} (-q)_n (q^3;q^3)_{n-1} }{(q)_n (q)_{2n-1} }
    \quad\mbox{(\sn{73})}
\end{equation}
\begin{multline}
\frac{ f(-q^6, -q^{12}) + f(-q^9,-q^9) -q f(-q^3,-q^{15})}{\s(-q)}
\\ = 1 + \sum_{n=1}^\infty \frac{ q^{n(n-1)/2} (-q)_n (q^3;q^3)_{n-1} }{(q)_{n-1} (q)_{2n} }
   \quad\mbox{(\sn{74})}
\end{multline}
\begin{equation}
 \frac{  f(-q^9,-q^9) -q f(-q^3,-q^{15})}{\s(-q)}
= 1 + \sum_{n=1}^\infty \frac{ q^{n(n+1)/2} (-q)_n (q^3;q^3)_{n-1} }{(q)_{n-1} (q)_{2n} }
   \quad\mbox{(\sn{75})}
\end{equation}

\setcounter{subsection}{20}
\subsection{Mod 21 Identities}
\begin{align}
\frac{ f(q^{10},q^{11}) - q f(q^4,q^{17})}{\s(-q)}
&=
\sum_{n=0}^\infty   \frac{q^{n(n+1)/2} (-q)_n}  { (q)_{2n}  }
      &\mbox{(Rogers~\cite[p. 331 (1)]{R17}; \sn{81})} \\
\frac{ f(q^{8},q^{13}) - q^2 f(q,q^{20})}{\s(-q)}
&=
 \sum_{n=0}^\infty   \frac{q^{n(n+1)/2} (-q)_n}  { (q)_{2n+1}  }
 &\mbox{(Rogers~\cite[p. 331 (1)]{R17}; \sn{80})} \\
\frac{ f(q^{5},q^{16}) - q f(q^2,q^{19})}{\s(-q)}
 &=   \sum_{n=0}^\infty   \frac{q^{n(n+3)/2} (-q)_n}  { (q)_{2n+1}  }
     &\mbox{(Rogers~\cite[p. 331 (1)]{R17}; \sn{82})}
\end{align}

\setcounter{subsection}{23}
\subsection{Mod 24 Identities}
\begin{align}
\frac{ f(q^{11},q^{13}) -q f(q^5, q^{10})}{f(-q)}
&=
\sum_{n=0}^\infty \frac{ q^{2n^2} }{(q)_{2n}} &\mbox{(\sn{83})}  \\
\frac{ f(q^{10},q^{14}) -q^2 f(q^2, q^{22})}{f(-q)}
&=
   \sum_{n=0}^\infty \frac{ q^{n(2n+1)} }{(q)_{2n+1}} &\mbox{(\sn{84})}  \\
&=    \sum_{n=0}^\infty \frac{ q^{n(2n-1)} }{(q)_{2n}} &\mbox{(\sn{85})}   \\
\frac{ f(q^{7},q^{17}) -q^2 f(q, q^{23})}{f(-q)}
&=
     \sum_{n=0}^\infty \frac{ q^{2n(n+1)} }{(q)_{2n+1}} &\mbox{(\sn{86})}  \\
\frac{ f(q^{8},q^{16}) +q f(q^4, q^{20})}{f(-q^2)}
    &=      \sum_{n=0}^\infty \frac{ q^{2n(n+1)} (-q;q^2)_{n}} {(q)_{2n+1} (-q^2;q^2)_n}
    &\mbox{(\sn{87})}
\end{align}

\setcounter{subsection}{26}
\subsection{Mod 27 Identities}
\begin{align}
\frac{f( q^{15},q^{12} )-qf(q^6, q^{21} ) }{f{-q}} &=
  \sum_{n=0}^\infty \frac{ q^{n(n+1)} (-1;q^3)_n  }{ (-1)_n (q)_{2n} } \notag\\
 &\mbox{(M.-S.~\cite[p. 766, Eq. (1.3)]{MS07})}\\
\frac{f(q^{12},q^{15}) - q^2 f(q^3,q^{24})}{f(-q)} &=
   \sum_{n=0}^\infty \frac{ q^{n^2} (-1;q^3)_n }{(-1)_n (q)_{2n}} \notag\\
&\mbox{(M.-S.~\cite[p. 766, Eq. (1.4)]{MS07})}\\
\frac{ f(q^9,q^{18}) - q^3 f(1,q^{27})}{f(-q)} &=
   \sum_{n=0}^\infty \frac{ q^{n(n+1)} (-q^3;q^3)_n }{(-q)_n (q)_{2n+1}} \notag\\
&\mbox{(M.-S.~\cite[p. 766, Eq. (1.5)]{MS07})}\\
\frac{ f(q^{6},q^{21}) - q^4 f(q^{-3},q^{30})}{f(-q)} &=
   \sum_{n=0}^\infty \frac{ q^{n(n+2)} (-q^3;q^3)_n (1-q^{n+1}) }
{(-q)_n (q)_{2n+2}} \notag\\ &\mbox
{(M.-S.~\cite[p. 766, Eq. (1.6)]{MS07})}\\
\frac{ f(-q^6,-q^{21}) -q^2 f(-q^3,-q^{24}) }{f(-q) }
&=  \sum_{n=1}^\infty   \frac{ q^{n^2 -1} (q^3;q^3)_{n-1} (1-q^{n+1}) }
  { (q)_{2n} (q)_{n-1} }  &\mbox{(\sn{88c})}\\
\frac{ f(-q^{12},-q^{15}) -q f(-q^6,-q^{21}) }{f(-q)}
&= 1+ \sum_{n=1}^\infty   \frac{ q^{n(n+1)} (q^3;q^3)_{n-1}  }
  { (q)_{2n} (q)_{n-1} }  &\mbox{(\sn{89c})}
\end{align}

\setcounter{subsection}{29}
\subsection{Mod 30 Identities}
\begin{align}
\frac{ f(q^{17},q^{13}) - q f(q^7, q^{23}) }{f(-q)}
&=
\sum_{n=0}^\infty \frac{ q^{n(n+1)} }{ (q)_{2n} }
       &\mbox{(Rogers~\cite[p. 333]{R94}; \sn{99})}
\\
\frac{ f(q^{14},q^{16}) - q^2 f(q^4, q^{26}) }{f(-q)}
&=
 \sum_{n=0}^\infty \frac{ q^{n^2} }{ (q)_{2n} }
&\mbox{(\sn{98})}
\\
\frac{ f(q^{11},q^{19}) - q^3 f(q, q^{29}) }{f(-q)}
&=
 \sum_{n=0}^\infty \frac{ q^{n(n+1)} }{ (q)_{2n+1} }
  & \mbox{(Rogers~\cite[p. 334]{R94}; \sn{94})}
 \\
\frac{ f(q^{8},q^{22}) - q^4 f(q^{-2}, q^{32}) }{f(-q)}
&=
  \sum_{n=0}^\infty \frac{ q^{n(n+2)} }{ (q)_{2n+1} }
  &\mbox{(\sn{96})}
  \end{align}
  \begin{align}
\frac{ f(q^{17},q^{13}) - q f(q^7, q^{23}) }{\ts(-q)}
&=
 \sum_{n=0}^\infty \frac{ q^{3n^2} (-q;q^2)_n }{ (q^2;q^2)_{2n} }
 &\mbox{(\sn{100c})}
\\
\frac{ f(q^{11},q^{19}) - q^3 f(q, q^{29}) }{\ts(-q)}
&=
   \sum_{n=0}^\infty \frac{ q^{n(3n-2)} (-q;q^2)_n }{ (q^2;q^2)_{2n} }
   &\mbox{(\sn{95})}\\
  &=   \sum_{n=0}^\infty \frac{ q^{n(3n+2)} (-q;q^2)_{n+1} }{ (q^2;q^2)_{2n+1} }
    &\mbox{(\sn{97c})}
\end{align}

\addtocounter{subsection}{1}
\subsection{Mod 32 Identities}
\begin{align}
\frac{ f(q^{8},q^{24})  -q^3 f(1,q^{32}) }{\s(-q) }
&=
\sum_{n=0}^\infty   \frac{  q^{n(n+3)/2} (-q^2;q^2)_{n} (-q)_{n+1}} {(q)_{2n+2}}
        &\mbox{(\sn{103})}   \\
\frac{ f(q^{10},q^{22})  -q f(q^6,q^{26}) }{\s(-q) }
&=
\sum_{n=0}^\infty   \frac{  q^{n(n+3)/2} (-q;q^2)_{n} (-q)_{n}} {(q)_{2n+1}}
        &\mbox{(\sn{106})}   \\
        \frac{ f(q^{12},q^{20}) -q^2 f(q^4,q^{28})  }{\s(-q) }
&= \sum_{n=0}^\infty   \frac{  q^{n(n+1)/2} (-q^2;q^2)_{n} (-q)_{n+1} (1-q^{n+1})} {(q)_{2n+2}}
         \\
\frac{ f(q^{14},q^{18})  -q^3 f(q^2,q^{30}) }{\s(-q) }
&=
\sum_{n=0}^\infty   \frac{  q^{n(n+1)/2} (-q;q^2)_{n} (-q)_{n}} {(q)_{2n+1}}
        &\mbox{(\sn{105})}   \\
\frac{ f(q^{16},q^{16})  -q f(q^8,q^{24}) }{\s(-q) }
&=
1+\sum_{n=1}^\infty   \frac{  q^{n(n+1)/2} (-q^2;q^2)_{n-1} (-q)_{n}} {(q)_{2n}}
        &\mbox{(\sn{104})}
        \end{align}
\begin{align}
\frac{ f(q^{16},q^{16}) + f(q^{12},q^{20})-q f(q^8,q^{24}) - q^2 f(q^4,q^{28}) }{\s(-q) }
&=
2+\sum_{n=1}^\infty   \frac{  q^{n(n-1)/2} (-q^2;q^2)_{n-1} (-q)_n} {(q)_{2n}} \nonumber\\
       &\mbox{(\sn{101c})}\\
      \frac{ f(q^{12},q^{20}) +q f(q^{8},q^{24})-q^2 f(q^4,q^{28}) - q^4 f(1,q^{32}) }{\s(-q) }
&=
\sum_{n=0}^\infty   \frac{  q^{n(n+1)/2} (-q^2;q^2)_{n} (-q)_{n+1}} {(q)_{2n+2}} \nonumber\\
        &\mbox{(\sn{102})}
\end{align}

\addtocounter{subsection}{3}
\subsection{Mod 36 Identities}
\begin{align}
\frac{ f(-q^{15},-q^{21})+q^3 f(-q^3,-q^{33})}{\s(-q^2)}
&=
  \sum_{n=0}^\infty \frac
    { q^{n(n+1)} (q^3;q^6)_{n} (-q^2; q^2)_{n} }{(q^2;q^2)_{2n+1} (q;q^2)_n}
    &\mbox{(\sn{107})}\\
\frac{ f(-q^{9},-q^{27})+q^5 f(-q^{-3},-q^{39})}{\s(-q^2)}
 &=
        \sum_{n=0}^\infty \frac
    { q^{n(n+3)} (q^3;q^6)_{n} (-q^2; q^2)_{n} }{(q^2;q^2)_{2n+1} (q;q^2)_n}\notag\\
    &\mbox{(\cite[p. 768, Eq. (1.30)]{MS07})} \\
\frac{ f(-q^{9},-q^{27})+q^3 f(-q^3,-q^{33})}{\ts(-q)}
&=
 \sum_{n=0}^\infty \frac
    {q^{n(n+2)}  (q^6;q^6)_{n} (-q; q^2)_{n+2}  }
    {(q^2;q^2)_{2n+2} (q^2;q^2)_{n}}  &\mbox{(\sn{112}) }\\
\frac{ f(-q^{12},-q^{24})+q^4 f(-1,-q^{36})}{\ts(-q)}
&=
  \sum_{n=0}^\infty \frac
    {q^{n(n+2)}  (q^3;q^6)_{n} (-q; q^2)_{n+1}  }
    {(q^2;q^2)_{2n+1} (q;q^2)_n}  &\mbox{(\sn{110}})\\
\frac{ f(-q^{15},-q^{21})-q f(-q^9,-q^{27})}{\ts(-q)}
&=
 1+\sum_{n=1}^\infty \frac
    {q^{n(n+2)}  (q^6;q^6)_{n-1} (-q; q^2)_{n}  }
    {(q^2;q^2)_{2n} (q^2;q^2)_{n-1}}  &\mbox{(\sn{111})  }\\
\frac{ f(-q^{15},-q^{21})-q^3 f(-q^3,-q^{33})}{\ts(-q)}
&=
1+\sum_{n=1}^\infty \frac
    {q^{n^2} (q^6;q^6)_{n-1} (-q; q^2)_{n}  }
    {(q^2;q^2)_{2n} (q^2;q^2)_{n-1}}
    &\mbox{(\sn{113})}
 \end{align}

\addtocounter{subsection}{5}
\subsection{Mod 42 Identities}
\begin{align}
\frac{f(q^{17},q^{25}) -q f(q^{11},q^{31}) }{\ts(-q)}
&=
\sum_{n=0}^\infty \frac  {q^{n(n+2)} (-q; q^2)_{n}  } {(q^2;q^2)_{2n} }
      &\mbox{(\sn{118})}   \\
\frac{f(q^{19},q^{23}) -q^3 f(q^{5},q^{37}) }{\ts(-q)}
&=
\sum_{n=0}^\infty \frac  {q^{n^2} (-q; q^2)_{n}  } {(q^2;q^2)_{2n} }
     &\mbox{(\sn{117})}  \\
\frac{f(q^{13},q^{29}) -q^5 f(q^{-1},q^{43}) }{\ts(-q)}
&=
       \sum_{n=0}^\infty \frac  {q^{n(n+2)} (-q; q^2)_{n+1}  } {(q^2;q^2)_{2n+1} }
       &\mbox{(\sn{119})}
  \end{align}

\addtocounter{subsection}{5}
\subsection{Mod 48 Identities}
\begin{align}
\frac{ f(q^{22},q^{26})-q f(q^{14},q^{34}) }{f(-q)}
&= 1 + \sum_{n=1}^\infty \frac{ q^{n(n+1)} (-q^2;q^2)_{n-1} }{(q)_{2n}}
&\mbox{(\sn{120})}
\\
\frac{ f(q^{22},q^{26})-q^2 f(q^{10},q^{38}) }{f(-q)}
&= 1 + \sum_{n=1}^\infty \frac{ q^{n^2} (-q^2;q^2)_{n-1} }{(q)_{2n}} &\mbox{(\sn{121})}
\\
\frac{ f(q^{10},q^{38})-q^3 f(q^{2},q^{46})}{f(-q)}
&=  \sum_{n=0}^\infty \frac{ q^{n(n+3)} (-q^2;q^2)_{n}}{(q)_{2n+2}}
&\mbox{(\sn{122})}
\\
\frac{ f(q^{20},q^{28})-q^4 f(q^4, q^{44})}{f(-q)}
&= \sum_{n=0}^\infty \frac{ q^{n(n+1)} (-q;q^2)_n}{(q)_{2n+1}} \notag\\
 &\mbox{(\eqref{qGauss} with $aq=-c=q^{3/2}$)} \\
\frac{ f(q^{34},q^{14})-q^6 f(q^{-2},q^{50})}{f(-q)}
&=  \sum_{n=0}^\infty \frac{ q^{n(n+2)} (-q^2;q^2)_{n}}{(q)_{2n+2}} &\mbox{(\sn{123})}
\end{align}

\addtocounter{subsection}{5}
\subsection{Mod 54 Identities}
\begin{align}
\frac{ f(-q^{27},-q^{27}) + q^3 f(-q^9,-q^{54})}{ f(-q^2)}
&=
  \sum_{n=0}^\infty \frac{ q^{2n^2} (q^3;q^6)_n }{(q^2;q^2)_{2n}(q;q^2)_n  }
  \\
\frac{ f(-q^{33},-q^{21}) + q^5 f(-q^3,-q^{51})}{ f(-q^2)}
&=
    \sum_{n=0}^\infty \frac
    {q^{2n(n+1)} (q^3; q^6)_{n}  }{(q^2;q^2)_{2n+1} (q;q^2)_n }
   &\mbox{(\sn{124})}  \\
\frac{ f(-q^{39},-q^{15}) + q^7 f(-q^{-3},-q^{57})}{ f(-q^2)}
&=
   \sum_{n=0}^\infty \frac
    {q^{2n(n+2)} (q^3; q^6)_{n}  }{(q^2;q^2)_{2n+1} (q;q^2)_n }
    &\mbox{(\sn{125c})}
\end{align}

\addtocounter{subsection}{9}
\subsection{Mod 64 Identities}
\begin{align}
\frac{ f(q^{28},q^{36}) - q^3 f(q^{12},q^{52})}{ \ts(-q)}
&=
   1+\sum_{n=1}^\infty \frac
    {q^{n^2} (-q; q^2)_{n}  (-q^4;q^4)_{n-1} }{(q^2;q^2)_{2n} }
    &\mbox{(\sn{126})} \\
\frac{ f(q^{28},q^{36}) - q f(q^{20},q^{44})}{ \ts(-q)}
&=
   1+\sum_{n=1}^\infty \frac
    {q^{n(n+2)} (-q; q^2)_{n}  (-q^4;q^4)_{n-1} }{(q^2;q^2)_{2n} }
    &\mbox{(\sn{127})  }\\
\frac{ f(q^{20},q^{44}) - q^5 f(q^{4},q^{60})}{ \ts(-q)}
&=
   \sum_{n=0}^\infty \frac
    {q^{n(n+2)} (-q; q^2)_{n+1}  (-q^4;q^4)_{n} }{(q^2;q^2)_{2n+2} }
    &\mbox{(\sn{128c})}  \\
\frac{ f(q^{12},q^{52}) - q^3 f(q^{4},q^{60})}{ \ts(-q)}
&=
   \sum_{n=0}^\infty \frac
    {q^{n(n+4)} (-q; q^2)_{n+1}  (-q^4;q^4)_{n} }{(q^2;q^2)_{2n+2} }
    &\mbox{(\sn{129})}
 \end{align}

\begin{align}
\frac{ f(q^{32},q^{32}) +qf(q^{24},q^{40})- q^5 f(q^{8},q^{56}) -q^8 f(1,q^{64})}{ \ts(-q)}
\notag\\=
   \sum_{n=0}^\infty \frac
    {q^{n^2} (-q; q^2)_{n+1}  (-q^2;q^4)_{n} }{(q^2;q^2)_{2n+1} }
    \qquad \qquad\mbox{(\sn{130})}
\end{align}

\setcounter{subsection}{107}
\subsection{Mod 108 Identities}
\begin{align}
\frac{ f(-q^3,-q^{24}) -2q^5 f(q^{15},q^{93}) + 2 q^{13} f(q^{-3},q^{111})}{f(-q)}
&= \sum_{n=0}^\infty \frac{ q^{n(n+3)}  (-q^3;q^3)_n}{(-q)_n (q)_{2n+2} }
\notag\\ &\mbox{(B.-M.-S.~\cite[Eq. (3.18)]{BMS07})} \\
\frac{ f(-q^6,-q^{21}) -2q^4 f(q^{21},q^{87}) + 2 q^{11} f(q^{3},q^{105})}{f(-q)}
&= \sum_{n=0}^\infty \frac{ q^{n(n+2)}  (-q^3;q^3)_n}{(-q)_n (q)_{2n+2} }
\notag\\ &\mbox{(B.-M.-S.~\cite[Eq. (3.17)]{BMS07})} \\
\frac{ f(-q^9,-q^{18}) -2q^3 f(q^{27},q^{81}) + 2 q^{9} f(q^{9},q^{99})}{f(-q)}
&= \sum_{n=0}^\infty \frac{ q^{n(n+1)}  (-q^3;q^3)_n}{(-q)_n (q)_{2n+1} }
\notag\\ &\mbox{(B.-M.-S.~\cite[Eq. (3.16)]{BMS07})} \\
\frac{ f(-q^{12},-q^{15}) -2q^2 f(q^{33},q^{75}) + 2 q^{7} f(q^{15},q^{93})}{f(-q)}
&= 1+\sum_{n=1}^\infty \frac{ q^{n^2}  (-q^3;q^3)_{n-1}}{(-q)_n (q)_{2n-1} }
\notag\\ &\mbox{(B.-M.-S.~\cite[Eq. (3.15)]{BMS07})}
\end{align}

\setcounter{subsection}{143}
\subsection{Mod 144 Identities}
\begin{align}
\frac{ f(-q^{3},-q^{33}) -2q^7 f(q^{18},q^{126}) + 2q^{12} f(q^{6}, q^{138 }  )  }{ \ts(-q) }
&= \sum_{n=0}^\infty \frac{ q^{n(n+4)} (-q;q^2)_{n+1} (-q^6;q^6)_{n}}{(-q^2;q^2)_n (q^2;q^2)_{2n+2} }
\notag\\ &\mbox{(B.-M.-S.~\cite[Eq. (3.39)]{BMS07})} \\
\frac{ f(-q^{15},-q^{21}) -2q^3 f(q^{42},q^{102}) + 2q^{10} f(q^{18}, q^{126 }  )  }{ \ts(-q) }
&= 1+\sum_{n=1}^\infty \frac{ q^{n^2} (-q;q^2)_n (-q^6;q^6)_{n-1}}{(-q^2;q^2)_n (q^2;q^2)_{2n-1} }
\notag\\ &\mbox{(B.-M.-S.~\cite[Eq. (3.38)]{BMS07})}
\end{align}

\section{Equivalent Products in Slater's List}
There is no canonical way to represent ratios of theta functions,
and it often happens that two seemingly different products are in
fact equivalent. 
The table below summarizes such occurrences in Slater's
list.  If a particularly nice form of the infinite product was observed, 
then that form of the product is recorded in the table.

{\small
\begin{longtable}{|c| l |}
\hline\hline
$\prod_{n\geq 1} (1-q^n)$ & (S. 1)\\ \hline
$\prod_{n\geq 1} (1+q^n)$ & (S. 2) = (S. 5-) = (S. 9) = \\
 & (S. 52) = (S. 84) = (S. 85) \\ \hline
$\prod_{n\geq 1} (1-q^{2n-1})$ & (S. 3) = (S. 23)\\ \hline
$\prod_{n\geq 1} (1-q^{2n-1})(1-q^{4n-2})$ & (S. 4) \\ \hline
$\prod_{n\geq 1} (1+q^n)$ & (S. 5-) = (S. 9) = (S. 52) \\
& = (S. 84) = (S. 85)\\ \hline
$\prod_{n\geq 1} (1+q^{3n-1})(1+q^{3n-2})(1-q^{3n})/(1-q^n)$  & (S. 6)\\ \hline
$\prod_{n\geq 1} (1+q^{2n})$ & (S. 7)\\ \hline
$\prod_{n\geq 1} 1/(1-q^n); n\not\equiv 0\pmod 4$ & (S. 8) = (S. 11) = \\
& =(S. 51) = (S. 64)\\ \hline
$\prod_{n\geq 1} (1+q^n)$& (S. 9) = (S. 5-) = (S. 52) \\
& = (S. 84) = (S. 85)\\ \hline
$\prod_{n\geq 1} (1-q^n)(1+q^{2n-1})$ & (S. 10) = (S. 47)\\ \hline
$\prod_{n\geq 1} 1/(1-q^n); n\not\equiv 0 \pmod {4}$  & (S. 11) = (S. 8) \\
& = (S. 51) = (S. 64)\\ \hline
$\prod_{n\geq 1} (1+q^{2n-1})/(1-q^{2n-1})$ & (S. 12)\\ \hline
& (S. 13) = (S. 8) + (S. 12) \\ \hline
$H(q) = \prod_{n\geq 1} 1/(1-q^n); n\equiv \pm 2\pmod 5$ & (S. 14)\\ \hline
$H(q)/(-q)_\infty$ & (S. 15)\\ \hline
$H(q)/(-q^2;q^2)_\infty$ & (S. 16)\\ \hline
& (S. 17)  = (S. 94-) \\ \hline
$G(q) = \prod_{n\geq 1} 1/(1-q^n); n\equiv \pm 1\pmod 5$ & (S. 18)\\ \hline
$G(q)/(-q)_\infty$ & (S. 19) \\ \hline
$G(q)/(-q^2;q^2)_\infty$ & (S. 20)\\ \hline
$\prod_{n\geq 1} (1-q^{2n-1})$ & (S. 23) = (S. 3)\\ \hline
  & (S. 24) = (S. 30-) \\ \hline
$\prod_{n\geq 1} 1/(1-q^n); n\equiv \pm 1,\pm 4\pm 5 \pmod {12}$ 
&(S. 27)= (S. 87)\\  \hline
$\prod_{n\geq 1} 1/(1-q^n); n\not\equiv 0,\pm 3 \pmod {12}$ &(S. 28)\\  \hline
&(S. 30) = (S. 24-)\\ \hline
$\prod_{n\geq 1} 1/(1-q^n); n\equiv \pm 3, 4\pmod{8} $& (S. 34)\\ \hline
&(S. 35) = (S. 106)\\ \hline
$\prod_{n\geq 1} 1/(1-q^n); n\equiv \pm 1, 4\pmod{8} $& (S. 36)\\ \hline
&(S. 37) =(S. 105)\\ \hline
$\prod_{n\geq 1} 1/(1-q^n); n\equiv\pm 1,\pm 4, \pm 6,\pm 7  \pmod{16}$&(S. 38) = (S. 86)\\ \hline
$\prod_{n\geq 1} 1/(1-q^n); n\equiv\pm 2,\pm 3,\pm 4, \pm 5  \pmod{16}$&(S. 39) = (S. 83)\\ \hline
$\prod_{n\geq 1} 1/(1-q^n); n\not\equiv 0,\pm 2 \pmod {10}$&(S. 44) = (S. 63)\\ \hline
$\prod_{n\geq 1} 1/(1-q^n); n\not\equiv 0,\pm 4\pmod {10}$ &(S. 46) = (S. 62)\\ \hline
$\prod_{n\geq 1} (1-q^n)(1+q^{2n-1})$&(S. 47) = (S.10) \\
& = (S. 54)$+q\times$(S. 49)\\ \hline
&(S. 48) = (S.54)$-q\times$(S. 49)\\ \hline
$\prod_{n\geq 1} 1/(1-q^n); n\not\equiv 0, \pm 1\pmod {12}$  & (S. 49)\\ \hline
$\prod_{n\geq 1} 1/(1-q^n); n\not\equiv 0,\pm 2\pmod {12}$ & (S. 50)\\ \hline
$\prod_{n\geq 1} 1/(1-q^n); n\not\equiv 0 \pmod {4}$ &(S. 51) = (S. 8) 
\\ & = (S. 11) = (S. 64)\\ \hline
$\prod_{n\geq 1} (1+q^n)$&(S. 52) = (S. 5-) \\
& = (S. 9) = (S. 84) \\ & = (S. 85)\\ \hline
$\prod_{n\geq 1} 1/(1-q^n); n\not\equiv 0,\pm 5\pmod {12}$ & (S. 54)\\ \hline
&(S. 55) = (S. 57-) \\ \hline
&(S. 57) = (S. 55-)\\ \hline
$\prod_{n\geq 1} 1/(1-q^n); n\not\equiv 0,\pm 2\pmod {14}$ & (S. 59)\\ \hline
$\prod_{n\geq 1} 1/(1-q^n); n\not\equiv 0,\pm 4\pmod {14}$ & (S. 60)\\ \hline
$\prod_{n\geq 1} 1/(1-q^n); n\not\equiv 0,\pm 6\pmod {14}$ & (S. 61)\\ \hline
$\prod_{n\geq 1} 1/(1-q^n); n\not\equiv 0,\pm 4\pmod {10}$ &(S. 62) = (S. 46)\\ \hline
$\prod_{n\geq 1} 1/(1-q^n); n\not\equiv 0,\pm 2\pmod {10}$ &(S. 63) = (S. 44)\\ \hline
$\prod_{n\geq 1} 1/(1-q^n); n\not\equiv 0 \pmod {4}$ &(S. 64) = (S. 8) \\
& = (S. 11) = (S. 51)\\ \hline
&(S. 65) \\ & = (S. 37)$+\sqrt{q}\times$(S. 35) 
\\ \hline
&(S. 66)\\ &= (S. 71)$+q\times$(S. 68)\\ \hline
&(S. 67)\\ &= (S. 71)$-q\times$(S. 68)\\ \hline
$\prod_{n\geq 1} 1/(1-q^n); n\not\equiv 0,\pm 4, \pm 6 \pm 10, 16\pmod{32}$&(S. 69)= (S. 123)\\
\hline
$\prod_{n\geq 1} 1/(1-q^n); n\mbox{ odd or }n \equiv 8 \pmod {16} $ & (S. 70)\\ \hline
$\prod_{n\geq 1} 1/(1-q^n); n\not\equiv 0,\pm 2, \pm 12, \pm 14, 16\pmod{32}$ &(S. 72)=(S. 121)\\
\hline
&(S. 73)= (S. 77)$+$(S. 78)\\ \hline
&(S. 74)= (S. 77)$+$(S. 78)\\ & $-q\times$(S. 76)\\ \hline
&(S. 75)\\ &=(S. 78)$-q\times$(S. 76)\\ \hline
&(S. 78)\\ &= (S. 75)$+q\times$(S. 76)\\ \hline
$\prod_{n\geq 1} 1/(1-q^n); n\mbox{ odd or }n\equiv \pm 4\pmod {20}$&(S. 79)= (S. 98)\\ \hline
$\prod_{n\geq 1} 1/(1-q^n); n\equiv\pm 2,\pm 3,\pm 4, \pm 5  \pmod{16}$&(S. 83) = (S. 39)\\ \hline
$\prod_{n\geq 1} (1+q^n)$ &(S. 84) = (S. 5-) \\ &= (S. 9) = (S. 52) = (S. 85)\\ \hline
$\prod_{n\geq 1} (1+q^n)$&(S. 85) = (S. 5-) \\ &= (S. 9) = (S. 52) = (S. 84) \\ \hline
$\prod_{n\geq 1} 1/(1-q^n); n\equiv\pm 1,\pm 4, \pm 6,\pm 7  \pmod{16}$&(S. 86) = (S. 38)\\ \hline
$\prod_{n\geq 1} 1/(1-q^n); n\equiv \pm 1,\pm 4\pm 5 \pmod {12}$ &(S. 87) = (S. 27)\\ \hline
&(S. 88)= (S. 91)$-q^2\times$(S. 90)\\ \hline
&(S. 89)= (S. 93)$-q\times$(S. 91)\\ \hline
$\prod_{n\geq 1} 1/(1-q^n); n\not\equiv 0,\pm 3 \pmod {27}$ & (S. 90)\\ \hline
$\prod_{n\geq 1} 1/(1-q^n); n\not\equiv 0,\pm 6 \pmod {27}$ & (S. 91)\\ \hline
$\prod_{n\geq 1} 1/(1-q^n); n\not\equiv 0 \pmod{9}$ & (S. 92)\\ \hline
$\prod_{n\geq 1} 1/(1-q^n); n\not\equiv 0,\pm {12} \pmod {27}$ & (S. 93)\\ \hline
$\prod_{n\geq 1} 1/(1-q^n); n\not\equiv 0, 10, \pm 3, \pm 4, \pm 7 \pmod {20}$ & (S. 94) = (S. 17-)\\
\hline
&(S. 95)= (S. 97)\\ \hline
$\prod_{n\geq 1} 1/(1-q^n); n\mbox{ odd or }n\equiv \pm 8\pmod {20}$&(S. 96)\\ \hline
&(S. 97) = (S. 95)\\ \hline
$\prod_{n\geq 1} 1/(1-q^n); n\mbox{ odd or }n\equiv \pm 4\pmod {20}$&(S. 98) = (S. 79)\\
\hline
$\prod_{n\geq 1} 1/(1-q^n); n\equiv \pm 2, \pm 3, \pm 4, \pm 5, \pm 6, \pm 7
\pmod{20}$ &(S. 99) \\ \hline
&(S. 105) = (S. 37)\\ \hline
&(S. 106) = (S. 35)\\ \hline
&(S. 108) = (S. 115)\\ & $-q^2\times$(S. 116)\\ \hline
$\prod_{n\geq 1} 1/(1-q^n); n\not\equiv 0, 2, 6, 10 \pmod {12}$ & (S. 110c) \\ \hline
&(S. 112) \\ &= (S. 115)$+q^3\times$(S. 116)\\ \hline
$\prod_{n\geq 1} 1/(1-q^n); n\not\equiv 0, \pm 2, \pm 6, \pm {10}, \pm{14}, \pm {15}, 18 \pmod {36}$  &(S. 114)\\ \hline
$\prod_{n\geq 1} 1/(1-q^n); n\not\equiv 0, \pm 2, \pm 6, \pm 9, \pm {10}, \pm{14}, 18 \pmod {36}$  &(S. 115)\\ \hline
$\prod_{n\geq 1} 1/(1-q^n); n\not\equiv 0, \pm 2, \pm 3,  \pm 6, \pm {10}, \pm{14}, 18 \pmod {36}$  &(S. 116)\\ \hline
$\prod_{n\geq 1} 1/(1-q^n); n\not\equiv 0,\pm 2, \pm 12, \pm 14, 16\pmod{32}$ &(S. 121) = (S. 72)\\
\hline
$\prod_{n\geq 1} 1/(1-q^n); n\not\equiv 0,\pm 4, \pm 6 \pm 10, 16\pmod{32}$&(S. 123) = (S. 69)\\
\hline
$\prod_{n\geq 1} 1/(1-q^n); n\equiv \pm 2, \pm 4, \pm 5, \pm 6 \pmod{18}$ 
&(S. 124)\\ \hline
$\prod_{n\geq 1} 1/(1-q^n); n\equiv \pm 2, \pm 6, \pm 7, \pm 8 \pmod{18}$ 
&(S. 125)\\ \hline
&(S. 127) = (S. 71)\\ & $-q\times$(S. 128)\\ \hline
&(S. 129) = $q^{-2}$\\ & $\times\Big($(S. 128)$-$(S. 68)$\Big)$\\ \hline
\hline
\caption{Products and cross-references in Slater's list}
\end{longtable}
}

\section{False Theta Series Identities}
Noting that Ramanujan's theta series
\[ f(a,b) := \sum_{n=-\infty}^\infty a^{n(n+1)/2} b^{n(n-1)/2} = \
  \sum_{n=0}^\infty a^{n(n+1)/2} b^{n(n-1)/2} +  \sum_{n=1}^\infty a^{n(n-1)/2} b^{n(n+1)/2},  \]
let us define the corresponding \emph{false theta series} as
 \[ \ft(a,b) :=
  \sum_{n=0}^\infty a^{n(n+1)/2} b^{n(n-1)/2} - \sum_{n=1}^\infty a^{n(n-1)/2} b^{n(n+1)/2}.  \]
  Notice that although $f(a,b)=f(b,a)$, in general, $\ft(a,b)\neq \ft(b,a)$.

L.J. Rogers~\cite{R17} studied $q$-series expansions for many instances of
$f(\pm q^\alpha, \pm q^\beta)$ and $\ft(\pm q^\alpha, \pm q^\beta)$, which he
called \emph{theta (resp. false theta) series of order $(\alpha+\beta)/2$}.

A false theta series identity for the series $\ft(\pm q^\alpha, \pm q^\beta)$
arises from the same Bailey pair as the
Rogers-Ramanujan type identity with product
$f(\pm q^\alpha, \pm q^\beta)/\s(-q)$.  A designation of the
form (\ftn{$n$}) means that the identity is the false theta analog of (\sn{$n$}), the
$n$th identity in Slater's list~\cite{S52}.  (Slater did not record any false theta
function identities.)

 \addtocounter{subsection}{-1}
\subsection{False Theta Series Identities of Order $\frac{3}{2}$}
\begin{align}
\ft(q^2,q) &=
 \sum_{n=0}^\infty \frac{(-1)^n q^{n(n+1)/2}}{(-q )_{n}}
    &\mbox{(Rogers~\cite[p. 333, (5)]{R17}; \ftn{2})} \label{E3f} \\
  &=\sum_{n=0}^\infty \frac{ q^{n(2n+1)}}{(-q)_{2n+1}}
      &\mbox{(Ramanujan~\cite[p. 233, Entry 9.4.3]{AB05}; \ftn{5})} \label{G5f} \\
&= 2-\sum_{n=0}^\infty \frac{ q^{n(2n-1)}}{(-q)_{2n}}
       &\mbox{(Ramanujan~\cite[p. 233, Entry 9.4.4]{AB05})}\\
&= \sum_{n=0}^\infty \frac{ (-1)^n q^n}{(-q^2;q^2)_{n}}
     & \mbox{(Ramanujan~\cite[p. 235, Entry 9.4.7]{AB05})}
\end{align}

\subsection{False Theta Series Identity of Order $2$}
\begin{align}
\ft(-q^3,-q) &=
  \sum_{n=0}^\infty \frac{  (-1)^n q^{n(n+1)} (-q;q^2)_{n} }{ (q;q^2)_{n+1} (-q^2;q^2)_n  }
 &\mbox{(M.-S.-Z.~\cite[Eq. (2.12)]{MSZ08}; \ftn{11}) \label{H19f} }
\end{align}

\subsection{False Theta Series Identities of Order $\frac{5}{2}$}
\begin{align}
\ft( q^4, q) &=
 \sum_{n=0}^\infty \frac{(-1)^n q^{n(n+1)}}{(-q )_{2n+1}}
   & \mbox{(Rogers~\cite[p. 334 (7)]{R17}; \ftn{17})} \label{G2f}  \\
\ft( q^3, q^2) &=
 \sum_{n=0}^\infty \frac{(-1)^n q^{n(n+1)}}{(-q )_{2n}}
    & \mbox{(Rogers~\cite[p. 334 (7)]{R17})}
\end{align}

\subsection{False Theta Series Identities of Order $3$}
\begin{align}
1&=
\sum_{n=0}^\infty \frac{(-1)^n q^{n^2}}{(q ; q^2)_{n+1}}
   &\mbox{(Rogers~\cite[p. 333 (4)]{R17}; \ftn{26}) } \label{C6f} \\
 \ft(-q^5,-q) &=
\sum_{n=0}^\infty \frac{(-1)^n q^{n(n+1)}}{(q ; q^2)_{n+1}}
    &\mbox{(Rogers~\cite[p. 333 (4)]{R17}; \ftn{22}; \ftn{28}) \label{F2f} }
    \end{align}
    \begin{align}
 \ft(q^5,q) &=
\sum_{n=0}^\infty \frac{ (-1)^n q^{ 3n(n+1)/2} (q)_{3n+1} }{ (q^3;q^3)_{2n+1} }
   &\mbox{ (Dyson~\cite[p. 9, Eq. (7.8)]{B49})} \\
\ft(q^4,q^2) &=
\sum_{n=0}^\infty \frac{ (-1)^n q^{ 3n(n+1)/2} (q)_{3n} (1-q^{3n+2}) }{ (q^3;q^3)_{2n+1} }
   & \mbox{  (Dyson~\cite[p. 9, Eq. (7.9)]{B49})}
\end{align}

\subsection{False Theta Series Identities of Order $4$}
\begin{align}
\ft(-q^2,-q^6) &=
\sum_{n=0}^\infty \frac{(-1)^n q^{n(n+1)/2} (-q)_n}{(q ; q^2)_{n+1}} \notag\\
    &\mbox{(Rogers~\cite[p. 333 (5)]{R17})} \label{D5f} \\
\ft(-q^6, -q^2) &=
\sum_{n=0}^\infty \frac{(-1)^n q^{n(n+1)/2} (-q^2;q^2)_n }{(q^{n+1} ; q)_{n+1}} \notag\\
    &\mbox{(Rogers~\cite[p. 333 (5)]{R17})} \label{D6f} \\
 \ft(-q^5,-q^3) &=
\sum_{n=0}^\infty \frac{(-1)^n q^{n(n+1)/2}(q)_n (-q;q^2)_n }{(q ; q)_{2n+1}} \notag\\
    &\mbox{(Ramanujan~\cite[p. 257, Eq. (11.5.3)]{AB05}; \ftn{37})}  \label{I17f} \\
  \ft(-q^7,-q) &=
 \sum_{n=0}^\infty \frac{(-1)^n q^{n(n+3)/2} (q)_n (-q;q^2)_n }{(q ; q)_{2n+1}} \notag\\
    & \mbox{(Ramanujan~\cite[p. 257, Eq. (11.5.4)]{AB05}; \ftn{35})}  \label{I18f}  \\
  \ft(q^7,q) &=
  \sum_{n=0}^\infty \frac{(-1)^n q^{2n(n+1)} (q^4;q^4)_n (q;q^2)_{2n+1} }{ (q^4;q^4)_{2n+1} }
  \notag\\
    &\mbox{ (Ramanujan~\cite[p. 257, Eq. (11.5.5)]{AB05})}
\end{align}

\subsection{False Theta Series Identities of Order $5$}
\begin{align}
\ft( -q^7, -q^3 ) &=
\sum_{n=0}^\infty \frac{(-1)^n q^{n(n+1)/2}}{(q ; q^2)_{n+1}}
    &\mbox{(Rogers~\cite[p. 333 (3)]{R17}; \ftn{45})} \label{C3f} \\
\ft(-q^9, -q) &=
\sum_{n=0}^\infty \frac{(-1)^n q^{n(n+3)/2}}{(q ; q^2)_{n+1}}
    & \mbox{(Rogers~\cite[p. 333 (3)]{R17}; \ftn{43})} \label{C4f}
\end{align}

\addtocounter{subsection}{2}
\subsection{False Theta Series Identities of Order $\frac{15}{2}$}
\begin{align}
\ft( q^8, q^7 ) -q \ft (q^2, q^{13}) &=
 \sum_{n=0}^\infty \frac{(-1)^n q^{n(3n-1)/2} (q)_n}{(q )_{2n}}
    & \mbox{(Rogers~\cite[p. 333 (2)]{R17})} \label{A5f}\\
\ft(q^7,q^8) +q \ft(q^2, q^{13}) &=
 \sum_{n=0}^\infty \frac{(-1)^n q^{n(3n+1)/2} (q)_n}{(q )_{2n+1}}
    & \mbox{(Rogers~\cite[p. 333 (2)]{R17}; \ftn{62})} \label{A6f} \\
\ft(q^4,q^{11}) + q \ft(q,q^{14}) &=
 \sum_{n=0}^\infty \frac{(-1)^n q^{3n(n+1)/2} (q)_n}{(q)_{2n+1}}
    &\mbox{(Rogers~\cite[p. 333 (2)]{R17})} \label{A7f}
\end{align}

\subsection{False Theta Series Identities of Order $9$}
\begin{align}
\ft(q^{15},-q^3) &= \sum_{n=0}^\infty \frac{ (-1)^n q^{n(n+3)/2} (-q^3;q^3)_n}
{ (q;q^2)_{n+1} (-q)_n (-q)_{n+1}}   
&\mbox{(M.-S. \cite[p. 768, Eq. (1.32)]{MS07})}
\end{align}
\begin{align}
\ft(q^{12},-q^6) + q^2 \ft(q^{18},-1) &=\sum_{n=0}^\infty \frac{ (-1)^n q^{n(n+1)/2} (-q^3;q^3)_n}
{ (q;q^2)_{n+1} (-q)_n^2}\notag\\
&\mbox{(M.-S. \cite[p. 768, Eq. (1.34)]{MS07})}
\end{align}
\begin{align}
\ft( q^{15}, q^{3} ) &=
\sum_{n=0}^\infty \frac{(-1)^n q^{n(n+3)/2} (q^3;q^3)_n}{(1+q^{n+1}) (q)_{2n+1}} \notag\\
    &\mbox{ (Dyson~\cite[J6; p. 434, Eq. (E1)]{B47}; \ftn{76})}  \label{J6f} \\
  &= \sum_{n=0}^\infty \frac{ (-1)^n q^{n(n+3)/2} (q^3;q^3)_n (1-q^{n+1})}{(q)_{2n+2} }\notag\\
  &\mbox{ (M.-S.-Z.~\cite[Eq. (2.13)]{MSZ08}; \ftn{75})} \label{J3f}  \\
  \ft( q^{12}, q^{6}) &=
\sum_{n=0}^\infty \frac{(-1)^n q^{n(n+1)/2} (q^3;q^3)_n } {(q)_{2n+1}}\notag\\
    & \mbox{ (Dyson~\cite[p. 434, Eq. (E2)]{B47}; \ftn{77})} 
\end{align}

\setcounter{subsection}{10}
\subsection{False Theta Series Identities of Order $\frac{21}{2}$}
\begin{align}
  \ft(q^8, q^{13}) + q^2 \ft(q,q^{20}) &=
 \sum_{n=0}^\infty \frac{(-1)^n q^{n(n+1)/2} (q)_n}{(q )_{2n+1}}
    & \mbox{(Rogers~\cite[p. 332 (1)]{R17}; \ftn{80})} \label{A2f} \\
 \ft( q^{10}, q^{11}) - q \ft( q^4, q^{17}) &=
 \sum_{n=0}^\infty \frac{(-1)^n q^{n(n+1)/2} (q)_n }{(q )_{2n}}
    & \mbox{(Rogers~\cite[p. 332 (1)]{R17}; \ftn{81})} \label{A3f} \\
 \ft( q^5, q^{16} ) + q\ft( q^2, q^{19}) &=
 \sum_{n=0}^\infty \frac{(-1)^n q^{n(n+3)/2} (q)_n }{(q)_{2n+1}}
    &\mbox{(Rogers~\cite[p. 332 (1)]{R17}; \ftn{82})} \label{A4f}
\end{align}

\setcounter{subsection}{15}
\subsection{False Theta Series Identities of Order $16$}
\begin{align}
\ft(-q^8,-q^{24}) &=
  \sum_{n=0}^\infty \frac{ (-1)^{n} q^{n(n+3)/2} (q)_{n+1} (-q^2;q^2)_n }{ (q)_{2n+2}  }
  \notag\\
  &\mbox{(M.-S.-Z. \cite[Eq. (2.14)]{MSZ08}; \ftn{103})} \label{K6f} \\
\ft(q^{22}, q^{10}) + q \ft(q^{26},q^6) &=
  \sum_{n=0}^\infty \frac{ (-1)^{n} q^{n(n+3)/2} (q)_{n} (-q;q^2)_n }{ (q)_{2n+1}  }
  \notag\\
  &\mbox{(M.-S.-Z. \cite[Eq. (2.15)]{MSZ08}; \ftn{106})}  \label{K4f}
\end{align}

\setcounter{subsection}{17}
\subsection{False Theta Series Identities of Order $18$}
\begin{align}
\ft(q^{21},-q^{15}) -q\ft(q^{27},-q^9) &=
\sum_{n=0}^\infty \frac{ (-1)^n q^{n(n+1)} (-q^3;q^6)_n}{(q^2;q^4)_n (-q;q)_{2n+1}}\notag\\
&\mbox{(M.-S. \cite[p. 768, Eq. (1.31)]{MS07})}\\
\ft(q^{21},-q^{15}) +q^3\ft(q^{33},-q^3) &=
\sum_{n=0}^\infty \frac{ (-1)^n q^{n(n+1)} (-q^3;q^6)_n}{(q^2;q^4)_{n+1} (-q;q)_{2n}}\notag\\
&\mbox{(M.-S. \cite[p. 768, Eq. (1.33)]{MS07})}\\
\ft(q^{27},-q^{9}) +q^2\ft(q^{33},-q^3) &=
\sum_{n=0}^\infty \frac{ (-1)^n q^{n(n+3)} (-q^3;q^6)_n}{(q^2;q^4)_{n+1} (-q;q)_{2n}}\notag\\
&\mbox{(M.-S. \cite[p. 768, Eq. (1.35)]{MS07})}\\
\ft(q^{21},q^{15}) +q\ft(q^{27},q^9) &=
\sum_{n=0}^\infty \frac{ (-1)^n q^{n(n+1)} (q^3;q^6)_n}{(q^2;q^4)_n (-q^2;q^2)_{n} (q;q^2)_{n+1}}
\notag\\
&\mbox{(M.-S. \cite[p. 768, Eq. (1.36)]{MS07})}\\
 \ft(q^{21},q^{15}) - q^3 \ft(q^{33},q^{3}) &=
 \sum_{n=0}^\infty \frac{ (-1)^n q^{n(n+1)} (q^3;q^6)_n  }{ (q^{2n+2};q^2)_{n+1} (q;q^2)_n }
 \notag\\
  &\mbox{(M.-S. \cite[p. 769, Eq. (1.38)]{MS07}; \ftn{107})} \label{J4f} \\
  \ft(q^{27},q^{9}) - q^5 \ft(q^{39},q^{-3}) &=
 \sum_{n=0}^\infty \frac{ (-1)^n q^{n(n+3)} (q^3;q^6)_n  }{ (q^{2n+2};q^2)_{n+1} (q;q^2)_n }
 \notag\\
 &\mbox{(M.-S. \cite[p. 769, Eq. (1.40)]{MS07})}
\end{align}

\section{Inter-Dependence Between Identities}

A number of pairs of identities on Slater's list are easily seen
to be equivalent, either by simply replacing $q$ by $-q$ (for
example, \textbf{S24} and \textbf{S30}), replacing $q$ by $q^2$
(for example, \textbf{S2} and \textbf{S7}), specializing a free
parameter in a general series = product identity (for example,
\textbf{S8}) or rearranging the finite $q$-products on the series
side (for example, \textbf{S10} and \textbf{S47}).

After eliminating one of each such pair of identities from the list
of identities, a natural question is: how  independent from each
other are the identities in the remaining set? In this section we
describe a number of  non-trivial ways  in which  pairs, or larger
sets, of identities are dependent.

 \subsection{Series-Equivalent Identities}

 Suppose we have an identity of the form
 \begin{equation}\label{equivid}
 \text{Infinite Series$_{1}$}
 = \text{Infinite Product}\times \text{Infinite Series$_{2}$},
\end{equation}
where each side contains one or more free parameters. It is
immediately clear that if, for particular values of the
parameters, either series has a representation as an infinite
product, then so does the other. We say two identities of
Rogers-Ramanujan type are \emph{series-equivalent}, if one can be
derived from the other by specializing the free parameters in an
identity of the type at \eqref{equivid}.

We next list some general series transformations. Most can be
derived as limiting cases of transformations between basic
hypergeometric series. Let $a$, $b$, $c$, $d$, $\gamma$ and $q \in
\mathbb{C}$, $|q|<1$. Then
\begin{equation}\label{eq0}
\sum_{n=0}^{\infty} \frac{(a,b;q)_{n}q^{n(n-1)/2}\left(-c \gamma /ab
\right)^{n}} {(c,\gamma,q;q)_{n}}= \frac{(c \gamma /ab;
q)_{\infty}}{(\gamma;q)_{\infty}}\sum_{n=0}^{\infty} \frac{(c /a,c
/b;q)_{n}q^{n(n-1)/2}(-\gamma)^{n}} {(c,c \gamma /ab,q;q)_{n}}.
\end{equation}
\begin{equation}\label{eq1}
\sum_{n=0}^{\infty} \frac{(a;q)_{n}q^{n(n-1)/2}\gamma^{n}}
{(b;q)_{n}(q;q)_{n}} = \frac{(-\gamma;
q)_{\infty}}{(b;q)_{\infty}}\sum_{n=0}^{\infty} \frac{(-a \gamma
/b;q)_{n}q^{n(n-1)/2}(-b)^{n}} {(-\gamma;q)_{n}(q;q)_{n}}.
\end{equation}
\begin{equation}\label{eq2}
\sum_{n=0}^{\infty} \frac{(a;q)_{n}q^{n(n-1)/2}\gamma^{n}}
{(q;q)_{n}} = (-\gamma; q)_{\infty}\sum_{n=0}^{\infty} \frac{(-a
\gamma )^{n}q^{n(n-1)}} {(-\gamma;q)_{n}(q;q)_{n}}.
\end{equation}
\begin{equation}\label{eq3}
\sum_{n=0}^{\infty} \frac{(a;q)_{n}q^{n(n-1)/2}\gamma^{n}} {(-a
\gamma;q)_{n}(q;q)_{n}} =\frac{(-\gamma;q)_{\infty}}{(-a
\gamma;q)_{\infty}}.
\end{equation}
\begin{equation}\label{eq4}
\sum_{n=0}^{\infty} \frac{q^{3n(n-1)/2}\gamma^{n}} {(
\gamma;q^2)_{n}(q;q)_{n}} =\frac{1}{(
\gamma;q^2)_{\infty}}\sum_{n=0}^{\infty}
\frac{q^{2n^2-n}\gamma^{n}} {(q^2;q^2)_{n}}.
\end{equation}
\begin{equation}\label{eq5}
\sum_{n=0}^{\infty} \frac{q^{n^2-n}(-\gamma)^{n}} {( \gamma
q;q^2)_{n}(q^2;q^2)_{n}} =\frac{1}{( \gamma
q;q^2)_{\infty}}\sum_{n=0}^{\infty} \frac{q^{n^2-n}(-\gamma)^{n}}
{(q;q)_{n}}.
\end{equation}
\begin{equation}\label{eq6}
\sum_{n=0}^{\infty} \frac{q^{n^2-n}(-\gamma)^{n}} {( \gamma /
q;q^2)_{n}(q^2;q^2)_{n}} =\frac{1}{( \gamma /
q;q^2)_{\infty}}\sum_{n=0}^{\infty} \frac{q^{n^2-2n}(-\gamma)^{n}}
{(q;q)_{n}}.
\end{equation}
\begin{equation}\label{eq7}
\sum_{n=0}^{\infty} \frac{q^{n(n-1)/2}\gamma^{n}} {(
\gamma;q)_{n}(q;q)_{n}} =\frac{1}{(
\gamma;q)_{\infty}}\sum_{n=0}^{\infty}
\frac{q^{2n^2-n}\gamma^{2n}} {(q^2;q^2)_{n}}.
\end{equation}
\begin{equation}\label{eq8}
\sum_{n=0}^{\infty} \frac{(b;q^2)_{n} q^{n(n-1)/2}(-\gamma)^{n}}
{( b;q)_{n}(q;q)_{n}}
 =(
\gamma;q)_{\infty}\sum_{n=0}^{\infty}
\frac{q^{4n^2-2n}(b\gamma^{2})^{n}}
{(q^2;q^2)_{n}(bq;q^2)_{n}(\gamma;q)_{2n}}.
\end{equation}
\begin{equation}\label{eq9}
\sum_{n=0}^{\infty} \frac{ q^{n(n-1)/2}(-\gamma)^{n}} {(
b;q)_{n}(q;q)_{n}}
 =(
\gamma;q)_{\infty}\sum_{n=0}^{\infty}
\frac{q^{(3n^2-3n)/2}(-b\gamma)^{n}}
{(q;q)_{n}(b;q)_{n}(\gamma;q)_{n}}.
\end{equation}
\begin{equation}\label{eq10}
\sum_{n=0}^{\infty} \frac{ q^{n^2}\gamma^{n}} {(
q/b;q)_{n}(q;q)_{n}}
 =(
-\gamma q^2;q^2)_{\infty}\sum_{n=0}^{\infty}
\frac{q^{n^2}\gamma^{n}(-q/b;q)_{2n}}
{(q^2;q^2)_{n}(q^2/b^2;q^2)_{n}(-\gamma q^2;q^2)_{n}}.
\end{equation}
\begin{equation}\label{eq11}
\sum_{n=0}^{\infty} \frac{(a;q)_{n}q^{n^2+n}(b\gamma/a)^{n}}
{(-bq;q)_{n}(-\gamma q;q)_{n}(q;q)_{n}} = \frac{1}{(-\gamma
q;q)_{\infty}}\sum_{n=0}^{\infty} \frac{(-b
q/a;q)_{n}q^{n(n+1)/2}\gamma^{n}} {(-bq;q)_{n}(q;q)_{n}}.
\end{equation}
If $n$ is a positive integer, then
\begin{equation}\label{eq13}
(-b q^n;q^n)_{\infty}\sum_{m=0}^{\infty} \frac{q^{(m^2+m)/2}a^{m}}
{(-bq^n;q^n)_{m}(q;q)_{m}} = (-a q;q)_{\infty}\sum_{m=0}^{\infty}
\frac{q^{n(m^2+m)/2}b^{m}} {(-aq;q)_{nm}(q^n;q^n)_{m}}.
\end{equation}
\begin{equation}\label{eq14}
\sum_{n=0}^{\infty} \frac{q^{n^2+n}a^{n}} {(q^2;q^2)_{n}(1+a
q^{2n+1})} = (-a q^2;q^2)_{\infty}\sum_{n=0}^{\infty}
\frac{q^{(n^2+n)/2}(-a)^{n}} {(-aq;q)_{n}}.
\end{equation}
\begin{equation}\label{eq15}
\sum_{n=0}^{\infty} \frac{(b/a;q)_{n}q^{(n^2+n)/2}a^{n}}
{(q;q)_{n}(aq;q)_{n}} = \frac{(bq;q^2)_{\infty}}{(a q;q)_{\infty}}
\sum_{n=0}^{\infty} \frac{(a^2q/b;q^2)_{n}q^{n^2+n}(-b)^{n}}
{(q^2;q^2)_{n}(bq;q^2)_{n}}.
\end{equation}
\begin{equation}\label{eq16}
\sum_{n=0}^{\infty} \frac{(d;q)_{2n}q^{n^2-n}(-c^2/d^2)^{n}}
{(q^2;q^2)_{n}(c;q)_{2n}}
 =
\frac{(c^2/d^2;q^2)_{\infty}}{( c;q)_{\infty}} \sum_{n=0}^{\infty}
\frac{q^{n^2-n}(-c)^{n}} {(q;q)_{n}(-c/d;q)_{n}}.
\end{equation}
\begin{equation}\label{eq161}
\sum_{n=0}^{\infty} \frac{q^{3n^2-2n}(-a^2)^{n}}
{(q^2;q^2)_{n}(a;q)_{2n}}
 =
\frac{1}{( a;q)_{\infty}} \sum_{n=0}^{\infty}
\frac{q^{n^2-n}(-a)^{n}} {(q;q)_{n}}.
\end{equation}
\begin{equation}\label{eq17}
\sum_{n=0}^{\infty} \frac{(a;q)_{n}q^{n^2-n}(-b)^{n}}
{(q;q)_{n}(ab;q^2)_{n}}
 =
\frac{(b;q^2)_{\infty}}{( ab;q^2)_{\infty}} \sum_{n=0}^{\infty}
\frac{(a;q^2)_{n}q^{n^2-n}(-bq)^{n}} {(q^2;q^2)_{n}(b;q^2)_{n}}.
\end{equation}
\begin{equation}\label{eq18}
\sum_{n=0}^{\infty} \frac{(a^2,b;q)_{n}q^{n^2+n}(-a^2/b)^{n}}
{(q;q)_{n}(a^2q/b;q)_{n}}
 =
\frac{(a^2q;q)_{\infty}}{( -aq;q)_{\infty}} \sum_{n=0}^{\infty}
\frac{(aq/b,-a;q)_{n}q^{(n^2-n)/2}(aq)^{n}} {(a^2q/b,aq,q;q)_{n}}.
\end{equation}
\begin{equation}\label{eq19}
\sum_{n=0}^{\infty} \frac{(a^2;q)_{n}q^{(3n^2+n)/2}a^{2n}}
{(q;q)_{n}}
 =
\frac{(a^2q;q)_{\infty}}{( -aq;q)_{\infty}} \sum_{n=0}^{\infty}
\frac{(-a;q)_{n}q^{(n^2-n)/2}(aq)^{n}} {(aq,q;q)_{n}}.
\end{equation}

Some of the identities above are derived from other identities
above, as limiting cases. However, we list them explicitly to have
the available for what follows below. Several follow from
identities in Andrews' two papers \cite{A66a}, \cite{A66b} and
some also follow from identities in Ramanujan's notebooks.

The transformation at \eqref{eq0} is a limiting case of a
$q$-analogue of the Kummer-Thomae-Whipple formula (see
\cite{GR04}, page 72, equation 3.2.7), which in turn is a limiting
case of Sear's $_{4}\phi_{3}$ transformation formula, \cite{S51}.

The identity at \eqref{eq1} is found in Ramanujan's lost notebook
\cite{R88} and a proof can be found in the recent book by Andrews
and Berndt \cite{AB05}. Equation \ref{eq2} follows from \ref{eq1}
upon letting $b \to 0$. The identity at \eqref{eq3}, which follows
upon setting $b=-a \gamma$ in \eqref{eq1}, is also found in
Ramanujan's notebooks (see \cite{B94}, Chapter 27, Entry 1, page
262). This identity is also equivalent to a result found in
Andrews \cite{A72}, where Andrews attributes it to Cauchy.

Proofs of \eqref{eq4}, \eqref{eq5}, \eqref{eq6} and \eqref{eq7}
can be found in \cite{GS83}, and alternative proofs can be found
in \cite{BMS07}. Identities  \eqref{eq5},  \eqref{eq6} and
\eqref{eq7} were also stated by Ramanujan in the lost notebook
(see Entries \textbf{1.5.1} and \textbf{1.5.2} in \cite{AB07}).
Proofs of \eqref{eq8}, \eqref{eq9} and \eqref{eq10} also are to be
found in \cite{GS83}.  Transformation \ref{eq9} was also stated by
Ramanujan (see Entry \textbf{2.24} of \cite{AB07}).

The transformation at \eqref{eq11} is a limiting case of Jackson's
transformation, \cite{J10} (see also \cite{GR04}, page 14).

A limiting case of a transformation due to Andrews \cite{A66a} leads
to the identity at \eqref{eq13}, which was also given by Ramanujan
in the lost notebook (see \cite{AB07}, Entry \textbf{1.4.12}).

Identity \ref{eq14} can be found in Ramanujan's lost notebook, and
a proof is given in \cite{AB07}, Entry \textbf{1.6.5}. Likewise,
\eqref{eq15} is also from Ramanujan's lost notebook (see
\cite{AB07}, Entry \textbf{1.7.3}).

The transformation at \eqref{eq16} follows from a series
transformation relating two $_8 \phi_7$'s in  \cite{GR04} (
(3.5.4) on pages 77--78, after replacing $c$ with $aq/c$, then
letting $a\to 0$ and finally letting $b \to \infty$). The identity
at \eqref{eq161} follows from that at \eqref{eq16}, upon letting
$d \to \infty$ and then replacing $c$ with $a$.

The identity at \eqref{eq17} follows from a result of Andrews in
\cite{A66b} (see also Corollary 1.2.3 of \cite{AB07}, where it
follows after replacing $t$ by $t/b$, then letting $b \to \infty$
and finally replacing $t$ by $b$).

A special case of Watson's transformation \cite{W29} of a
terminating very-well-poised $_8 \phi_7$ yields \eqref{eq18} (see
also \cite{GR04}, page 43, where the transformation follows upon
letting $n \to \infty$, replacing $a$ by $a^2$, setting $c=a$,
$d=-a$ and finally letting $e \to \infty$). The transformation at
\eqref{eq19} follows from that at \eqref{eq18}, after letting $b
\to \infty$.

 Several identities on Slater's list follow directly from some of the
 transformations above, in that particular values of the
 parameters make one of the series identically equal to 1, so
 that what remains is a ``series = product" identity. We  list
 such identities, along with the transformations from which they
 derive in Table \ref{Ta:t1}. Series-equivalent identities are
 listed in Table \ref{Ta:t2}.

 Note that the case $a=q$ in \eqref{eq3} implies that, for $|q|<1$
 and $\gamma \in \mathbb{C}$,
 \[
 \sum_{n=0}^{\infty} \frac{q^{n(n-1)/2}\gamma^{n}}
 {(-\gamma, q)_{n+1}}=1.
 \]

 The tables are to be
 understood as follows: let the parameters have the specified
 values in the specified transformations (listed in the first column), and then, if applicable, make the indicated base changes in
 either both sides of the transformation or else recognize that one
 of the series equals the series in the corresponding identity,
 after the indicated base change.

For example, from row one of Table \ref{Ta:t2}, if the indicated
substitutions are made in Transformation \ref{eq11}, we get that
\begin{equation*}
\sum_{n=0}^{\infty} \frac{(-1;q)_{n}q^{n^2+n}(1/q)^{n}}
{(-q^{1/2};q)_{n}(q^{1/2};q)_{n}(q;q)_{n}} = \frac{1}{(
q^{1/2};q)_{\infty}}\sum_{n=0}^{\infty}
\frac{(q^{1/2};q)_{n}q^{n(n+1)/2}(-q^{-1/2})^{n}}
{(-q^{1/2};q)_{n}(q;q)_{n}}.
\end{equation*}
From \textbf{S.6} on Slater's list, the left side equals
$(-q,-q^2,q^3;q^3)_{\infty}/(q;q)_{\infty}$. Replace $q$ by $q^{2}$
and we have that
\begin{equation*}
\frac{(-q^2,-q^4;q^6;q^6)_{\infty}(q;q^2)_{\infty}}{(q^2;q^2)_{\infty}}
=\sum_{n=0}^{\infty} \frac{(q;q^2)_{n}q^{n^2}(-1)^{n}}
{(-q;q^2)_{n}(q^2;q^2)_{n}}.
\end{equation*}
Finally, replace $q$ by $-q$ and we have Identity \textbf{S.29}
from Slater's list.

As a second example, from row 5 of the table, if the indicated
value $\gamma = q^{3/2}$ is substituted in Transformation
\ref{eq7}, we get
\begin{equation*}
\sum_{n=0}^{\infty} \frac{q^{n(n+2)/2}} {(q^{3/2};q)_{n}(q;q)_{n}}
=\frac{1}{( q^{3/2};q)_{\infty}}\sum_{n=0}^{\infty}
\frac{q^{2n^2+2n}} {(q^2;q^2)_{n}}.
\end{equation*}
The series on the right is the series in Identity \textbf{S.14}
(with $q$ replaced by $q^2$). Thus it follows that
\begin{equation*}
\sum_{n=0}^{\infty} \frac{q^{n(n+2)/2}} {(q^{3/2};q)_{n}(q;q)_{n}}
=\frac{1}{( q^{3/2};q)_{\infty}(q^4,q^6;q^{10})_{\infty}}.
\end{equation*}
Replacing $q$ by $q^2$ and dividing both sides by $1-q$ leads to
Identity \textbf{S.96}.

The transformations also imply some identities which we believe to
be new. For example, if we set $c=-q^2$, $d=q^{1/2}$ and then
replace $q$ by $q^2$ in \eqref{eq16}, the series on the right side
becomes that in \textbf{S38} (up to a multiple of $1-q$) and leads
to the following identity: {\allowdisplaybreaks
\begin{equation}\label{eq16a}
\sum_{n=0}^{\infty} \frac{(q;q^2)_{2n}q^{2n^2+4n}(-1)^n}
{(q^8;q^8)_{n}(-q^2;q^4)_{n+1}} =
(-q^9,-q^7,q^8;q^8)_{\infty}\frac{ (q^2;q^4)_{\infty}}{
(q^4;q^4)_{\infty}}.
\end{equation}
} A similar pairing of \textbf{S39} and \eqref{eq16} leads to the
following:
 {\allowdisplaybreaks
\begin{equation}\label{eq16b}
\sum_{n=0}^{\infty} \frac{(q;q^2)_{2n}q^{2n^2}(-1)^n}
{(q^8;q^8)_{n}(-q^2;q^4)_{n}} = (-q^3,-q^5,q^8;q^8)_{\infty}\frac{
(q^2;q^4)_{\infty}}{ (q^4;q^4)_{\infty}}.
\end{equation}
These identities  are somewhat reminiscent of the following
identities of Gessel and Stanton \cite{GS83}: {\allowdisplaybreaks
\begin{equation}\label{GS2}
\sum_{n=0}^{\infty} \frac{(-q;q^2)_{2n}q^{2n^2}}
{(q^8;q^8)_{n}(q^2;q^4)_{n}} = (-q^3,-q^5,q^8;q^8)_{\infty}\frac{(
-q^2;q^4)_{\infty}}{( q^4;q^4)_{\infty}}.
\end{equation}
}
 {\allowdisplaybreaks
\begin{equation}\label{GS1}
\sum_{n=0}^{\infty} \frac{(q;q^2)_{2n+1}q^{2n^2+2n}}
{(q^2;q^2)_{2n+1}(-q^2;q^4)_{n+1}} =
\frac{(-q,-q^7,q^8;q^8)_{\infty}( -q^4;q^4)_{\infty}}{
(q^4;q^4)_{\infty}}.
\end{equation}
}

\begin{table}[ht]  \begin{center}
\begin{tabular}{| c | c | c | c |c | c | c |}   \hline
Transform.&Identity  &$a$& $b$  & $\gamma$ & base change     \\
\hline \ref{eq2}& \textbf{S.2}  & 0 &
  & $q$&
 \\
\ref{eq2}& \textbf{S.3}  & 1 &
  & $q$&
 \\
\ref{eq1}& \textbf{S.4}  &$ -q^{1/2}$ & $-q$
  & $-q^{1/2}$&$q\to q^2$
 \\
 \ref{eq2}& \textbf{S.9}  & 1 &
  &$-q^{3}$ & $q\to q^2$
 \\
\ref{eq1}& \textbf{S.10}  & -1 & $q$
  &$q$ &$q\to q^2$
 \\
 \ref{eq1}& \textbf{S.11}  & $-q$ & $q^3$
  &$q^2$ &$q\to q^2$
 \\
 \ref{eq2}& \textbf{S.52}  & 1 & $$
  &$-q^{1/2}$ &$q\to q^2$
 \\
 \ref{eq2}& \textbf{S.47}  & -1 & $$
  &$q^{1/2}$ &$q\to q^2$
 \\
 \hline
    \end{tabular}\phantom{asdf}\\
\vspace{10pt}   \caption{Identities deriving directly from general
transformations.}\label{Ta:t1}
 \end{center}\end{table}

{\allowdisplaybreaks
\begin{table}[!ht]
\begin{center}
\begin{tabular}{| c | c |c | c | c |c | c | c |}   \hline
Transf.&Left &Right &$a$& $b$  & $\gamma$ & base change     \\
\hline \ref{eq11}& \textbf{S6} & \textbf{S29} & -1 & $q^{-1/2}$
  &$-q^{-1/2}$ &$q\to q^2,$$q\to -q$\\
\ref{eq4}& \textbf{S44} & \textbf{S14} &  &
  &$q^{3}$ &
 \\
 \ref{eq5}& \textbf{S17} & \textbf{S14} &  &
  &$-q^{2}$ &
 \\
 \ref{eq6}& \textbf{S16} & \textbf{S14} &  &
  &$-q^{3}$ &
 \\
 \ref{eq9}& \textbf{S16} & \textbf{S97} &  &$-q$
  &$q^{3/2}$ &$q\to q^{2}$,$q\to -q$
 \\
\ref{eq7}& \textbf{S96} & \textbf{S14} &  &
  &$q^{3/2}$ & ($q\to q^2$) $q\to q^2$
 \\
\ref{eq4}& \textbf{S46} & \textbf{S18} &  &
  &$q$ &
 \\
\ref{eq5}& \textbf{S20} & \textbf{S18} &  &
  &$-q$ &
 \\
\ref{eq6}& \textbf{S99} & \textbf{S18} &  &
  &$-q^2$ &
 \\
\ref{eq7}& \textbf{S79} & \textbf{S18} &  &
  &$q^{1/2}$ &($q\to q^2$) $q\to q^2$\\
 \ref{eq9}& \textbf{S20} & \textbf{S19} & $ $ &$-q$
  &$q^{-1/2}$ & $q\to q^2$\\
  \ref{eq11}& \textbf{S22} & \textbf{S50} & $-q$ &$q^{-1/2}$
  &$q^{1/2}$ & $q\to q^2$\\
\ref{eq1}& \textbf{S25} & \textbf{S48} & $-q^{1/2}$ &$-q$
  &$q^{1/2}$ & $q\to q^2$\\
\ref{eq11}& \textbf{S27} & \textbf{S28} & $-q^{1/2}$ &$-q^{1/2}$
  &$1$ & $q\to q^2$\\
\ref{eq1}& \textbf{S28} & \eqref{RamStanton} & $-q$ &$q^{3/2}$
  &$q$ & $q\to q^2$\\
\ref{eq15}& \textbf{S29} & \textbf{S48} & $q^{-1/2}$ &$-1$
  & & $q\to q^{1/2}$\\
\ref{eq15}& \textbf{S50} & \eqref{RamStanton} & $q^{-1/2}$ &$-q$
  & & $q\to q^{1/2}$\\
\ref{eq2}& \textbf{S34} & \textbf{S38} & $-q^{1/2}$ &$$
  &$q^{3/2}$ & $q\to q^2$\\
\ref{eq2}& \textbf{S36} & \textbf{S39} & $-q^{1/2}$ &$$
  &$q^{1/2}$ & $q\to q^2$\\
\ref{eq10}& \textbf{S38} & \textbf{\ref{GS1}} & $$ &$q^{-1/2}$
  &$q$ & $q\to q^2$\\
\ref{eq10}& \textbf{S39} & \textbf{\ref{GS2}} & $$ &$q^{1/2}$
  &$1$ & $q\to q^2$\\
\ref{eq9}& \textbf{S44} & \eqref{RogMod20} & $$ &$q^{3/2}$
  &$-q^{3/2}$ & $q\to q^2$\\
\ref{eq9}& \textbf{S46} & \textbf{S79} & $$ &$q^{1/2}$
  &$-q^{1/2}$ & $q\to q^2$\\
\ref{eq1}& \textbf{S94} & \textbf{S16} & $0$ &$q^{3/2}$
  &$q$ & $q\to q^2$\, $(q \to -q)$\\
\ref{eq9}& \textbf{S94} & \textbf{S97} & $$ &$q^{3/2}$
  &$-q$ & $q\to q^2$\\
\ref{eq9}& \textbf{S96} & \textbf{S44} & $$ &$q^{3/2}$
  &$-q^{3/2}$ & $q\to q^2$\\
\ref{eq1}& \textbf{S99} & \textbf{S20} & $0$ &$q^{1/2}$
  &$q$ & $q\to q^2$\, $(q \to -q)$\\
\ref{eq9}& \textbf{S99} & \textbf{S100} & $$ &$q^{1/2}$
  &$-q$ & $q\to q^2$\\
\ref{eq13}& \textbf{S80} & \textbf{S118} & $1$ &$-q$
  &$$ & ($n=2$)\\
\ref{eq13}& \textbf{S81} & \textbf{S117} & $1$ &$-1/q$
  &$$ & ($n=2$)\\
\ref{eq13}& \textbf{S82} & \textbf{S119} & $q$ &$-q$
  &$$ & ($n=2$)\\
\ref{eq161}& \textbf{2.5.3} & \textbf{S14} & $-q^2$ &$$
  &$$ & \\
\ref{eq161}& \textbf{S19} & \textbf{S18} & $-q^2$ &$$
  &$$ & \\
\ref{eq17}& \textbf{S6} & \textbf{S48} & $-1$ &$q$
  &$$ & $q\to -q$\\
\ref{eq17}& \textbf{S22} &  \eqref{RamStanton} & $-q$ &$-q^2$
  &$$ & $$\\
\ref{eq19}& \textbf{1.3.3} & \textbf{S50} & $q^{1/2}$ &$$
  &$$ & $q\to q^2$\\
& ($a=q$) &  &  &$$
  &$$ & $$\\
  \hline
    \end{tabular}\phantom{asdf}\\
\vspace{10pt}   \caption{Identities equivalent to each other via
some general transformation}\label{Ta:t2}
 \end{center}\end{table}
}

\subsection{Inter-dependence of Identities via the Jacobi Triple
Product Identity}

Let $m \geq 2$ be a positive integer. By considering sums in the $m$
arithmetic progressions $mk+r$, $0 \leq r <m$ and $k \in
\mathbb{Z}$, we easily get that
\begin{equation*}
\sum_{k=-\infty}^{\infty}x^k q^{k(k-1)/2}=
\sum_{r=0}^{m-1}q^{r(r-1)/2}x^{r} \sum_{k=-\infty}^{\infty}\left(
x^{m} q^{(m^2-m+2mr)/2} \right)^{k} q^{m^2(k^2-k)/2},
\end{equation*}
and thus, after applying the Jacobi triple product, that
\begin{multline}\label{xq}
(-x,-q/x,q;q)_{\infty}\\= \sum_{r=0}^{m-1}q^{r(r-1)/2}x^{r} \left(
-x^{m} q^{(m^2-m+2mr)/2}, -x^{-m} q^{(m^2+m-2mr)/2}, q^{m^2};q^{m^2}
\right)_{\infty}.
\end{multline}
The case $m=2$ is of course Bailey's expression
\begin{equation*}
(-x,-q/x,q;q)_{\infty}
 = \left(
-x^{2} q, -x^{-2} q^{3}, q^{4};q^{4} \right)_{\infty} + x \left(
-x^{2} q^{3}, -x^{-2} q, q^{4};q^{4} \right)_{\infty}.
\end{equation*}

If we replace $q$ by $q^3$ and set $x=-q$ in \eqref{xq}, then
\begin{align}\label{qq}
(q:q)_{\infty} &= (-q^5, -q^7, q^{12};q^{12})_{\infty}-q(-q^{11},
-q, q^{12};q^{12})_{\infty}, \phantom{asdaasd}&(m=2)\\
&= (q^{12}, q^{15}, q^{27};q^{27})_{\infty}-q(q^{21}, q^{6},
q^{27};q^{27})_{\infty} \notag\\
&\phantom{asdasdasdasdasdasasdadda}- q^{2}(q^{3}, q^{24},
q^{27};q^{27})_{\infty},\,\,&(m=3)\notag\\
&= (-q^{22}, -q^{26}, q^{48};q^{48})_{\infty}-q(-q^{34}, -q^{14},
q^{48};q^{48})_{\infty} \notag\\
&\phantom{as}- q^{2}(-q^{10}, -q^{38}, q^{48};q^{48})_{\infty}+
q^{5}(-q^{46}, -q^{2}, q^{48};q^{48})_{\infty},&(m=4)\notag\\
&\phantom{asasdadsdasdadasdadasd}\vdots  & \vdots
.\phantom{as}\notag
\end{align}
Note that the right sides required some elementary manipulations in
the cases $m=3$ and $m=4$.

Now suppose we have a finite set of ``Series $S_i$ = Product $P_i$"
identities ($1 \leq i \leq n$). Clearly
\begin{equation}\label{S=P}
\sum_{i=1}^{n} \alpha_i S_i = \sum_{i=1}^{n} \alpha_i P_i
\end{equation}
holds for any set of complex constants $\alpha_{i}$. If, for a
certain set of $\alpha_i$,  \eqref{S=P} follows as a consequence
of the Jacobi triple product identity, then we say the identities
$S_i=P_i$, $1 \leq i \leq n$ are \emph{JTP-dependent}, in the
sense that any one identity can be derived from the other $n-1$
identities taken together with the Jacobi triple product identity.
As an example, consider the following three identities from
Slater's list:
\begin{align}\label{ta3ex}
1
+\sum_{n=1}^{\infty}\frac{(q^3;q^3)_{n-1}q^{n^2}}{(q;q)_{2n-1}(q;q)_{n}}
&= \frac{(q^{12},q^{15},q^{27};q^{27})_{\infty}}{(q;q)_{\infty}}
&\tag{\textbf{S93}}\\
\sum_{n=0}^{\infty}\frac{(q^3;q^3)_{n}q^{n(n+2)}}{(q;q)_{2n+2}(q;q)_{n}}
&= \frac{(q^{6},q^{15},q^{21};q^{27})_{\infty}}{(q;q)_{\infty}}
&\tag{\textbf{S91}} \\
\sum_{n=0}^{\infty}\frac{(q^3;q^3)_{n}q^{n(n+3)}}{(q;q)_{2n+2}(q;q)_{n}}
&= \frac{(q^{3},q^{24},q^{27};q^{27})_{\infty}}{(q;q)_{\infty}}.
&\tag{\textbf{S90}}
\end{align}
One easily checks that the sum side of \textbf{S93} -
$q\times$\textbf{S91} - $q^2\times$\textbf{S90} is identically 1,
while the product side being identically 1 follows from the $m=2$
case of \eqref{qq}. Thus the identities \textbf{S90}, \textbf{S91}
and \textbf{S93} are JTP-dependent.

Before compiling a table of sets of identities which are
JTP-dependent, we first exhibit two other $q$-products which, like
$(q;q)_{\infty}$ at \eqref{qq},  have infinitely many expressions as
sums of triple products.

Firstly, if we replace $q$ by $q^4$ and set $x=-q$, we get that
\begin{align}\label{q2-q} \frac{(q^2:q^2)_{\infty}}{(-q:q^2)_{\infty}}
&= (-q^6, -q^{10}, q^{16};q^{16})_{\infty}-q(-q^{14},
-q^2, q^{16};q^{16})_{\infty}, &(m=2)\\
&= (q^{15}, q^{21}, q^{36};q^{36})_{\infty}-q(q^{27}, q^{9},
q^{36};q^{36})_{\infty} \notag\\
&\phantom{asdasdasdasdasdasas}- q^{3}(q^{3}, q^{33},
q^{36};q^{36})_{\infty},\,\,&(m=3)\notag\\
&= (-q^{28}, -q^{36}, q^{64};q^{64})_{\infty}
   -q(-q^{44}, -q^{20}, q^{64};q^{64})_{\infty} \notag\\
&- q^{3}(-q^{12}, -q^{58}, q^{64};q^{64})_{\infty}+
q^{6}(-q^{60}, -q^{4}, q^{64};q^{64})_{\infty},&(m=4)\notag\\
&\phantom{asasddadasdadasd}\vdots  &\vdots .\phantom{as} \notag
\end{align}
Here also the right sides required some elementary manipulations
in the cases $m=3$ and $m=4$.

Secondly, replace $q$ by $q^2$ and set $x=-q$ to get
{\allowdisplaybreaks
\begin{align}\label{q-q} \frac{(q:q)_{\infty}}{(-q:q)_{\infty}}
&= (-q^4, -q^{4}, q^{8};q^{8})_{\infty}-2q(-q^{8},
-q^8, q^{8};q^{8})_{\infty}, \phantom{asdaasd}&(m=2)\\
&= (q^{9}, q^{9}, q^{18};q^{18})_{\infty}-2q(q^{15}, q^{3},
q^{18};q^{18})_{\infty} &(m=3)\notag\\
&= (-q^{16}, -q^{16}, q^{32};q^{32})_{\infty}
   -2q(-q^{24}, -q^{8}, q^{32};q^{32})_{\infty} \notag\\
&\phantom{aasASasASasaSasasaas} +
2q^{4}(-q^{32}, -q^{32}, q^{32};q^{32})_{\infty},&(m=4)\notag\\
&\phantom{asasdadsdasdadasdada}\vdots  & \vdots .\phantom{as}\notag
\end{align}
}

The table is to be understood as follows: in each case it is easily
checked that  the combinations of $q$-series from the indicated
identities sums to the indicated value; that the same combination of
the corresponding products in the indicated identities sums to the
same value follows from the stated identity multiplied by the given
$q$-product, for the stated value of $m$. See the example at
\eqref{ta3ex} above for more details.

Applying the Jacobi triple product in the ways described above can
also lead to new identities. For example, considering
\textbf{S53}(with $q$ replaced by $-q$) $-q\times$\textbf{S57} in
conjunction with the $m=2$ case of \eqref{qq} leads to the following
identity:
\begin{align*}
1+ \sum_{n=1}^{\infty} \frac{q^{(2n-1)^2}(-q;q^2)_{2n-1}(-1+q^{4
n-1}+q^{8 n}+q^{8 n-2})}{(q^4;q^4)_{2n}}
=\frac{(q;q)_{\infty}}{(q^4;q^4)_{\infty}}.
\end{align*}

\begin{longtable}{| c | c | c | }   \hline
Series Identity&Product Identity       \\
\hline
 \textbf{S58}-$q\times$\textbf{S56}=1
 &$\displaystyle{\eqref{qq}|_{ m=2}\times \frac{1}{(q;q)_{\infty}}}$ \\
 \textbf{S54}+$q\times$\textbf{S49}=\textbf{S47}
 &$\displaystyle{\eqref{qq}|_{ m=2, q\to -q
 }\times \frac{1}{(q;q)_{\infty}}}$ \\
 \textbf{S93}-$q\times$\textbf{S91}-$q^{2}\times$\textbf{S90}=1
 &$\displaystyle{\eqref{qq}|_{m=3}\times \frac{1}{(q;q)_{\infty}}}$ \\
 \textbf{S120}-$q^2\times$\textbf{S122}=1
 &$\displaystyle{\eqref{qq}|_{m=4}\times \frac{1}{(q;q)_{\infty}}}$ \\
 \textbf{S72}-$q\times$\textbf{S69}=1
 &$\displaystyle{\eqref{q2-q}|_{m=2}\times\frac{(-q;q^2)_{\infty}}{(q^2;q^2)_{\infty}}}$ \\
 \textbf{S121}-$q\times$\textbf{S123}=1
 &$\displaystyle{\eqref{q2-q}|_{m=2}\times\frac{(-q;q^2)_{\infty}}{(q^2;q^2)_{\infty}}}$ \\
 \textbf{S114}-$q\times$\textbf{S115}-$q^3\times$\textbf{S116}=1
 &$\displaystyle{\eqref{q2-q}|_{m=3}\times\frac{(-q;q^2)_{\infty}}
 {(q^2;q^2)_{\infty}}}$ \\
 \textbf{S111}+\textbf{S113}-\textbf{S114}=1
 &$\displaystyle{\eqref{q2-q}|_{m=3}\times\frac{(-q;q^2)_{\infty}}
 {(q^2;q^2)_{\infty}}}$ \\
 \textbf{S113}-$q\times$\textbf{S115}=1
 &$\displaystyle{\eqref{q2-q}|_{m=3}\times\frac{(-q;q^2)_{\infty}}
 {(q^2;q^2)_{\infty}}}$ \\
 \textbf{S111}-$q^3\times$\textbf{S116}=1
 &$\displaystyle{\eqref{q2-q}|_{m=3}\times\frac{(-q;q^2)_{\infty}}
 {(q^2;q^2)_{\infty}}}$ \\
 \textbf{S126}-$q\times$\textbf{S128}=1
 &$\displaystyle{\eqref{q2-q}|_{m=4}\times\frac{(-q;q^2)_{\infty}}{(q^2;q^2)_{\infty}}}$ \\
 \textbf{S127}-$q^3\times$\textbf{S129}=1
 &$\displaystyle{\eqref{q2-q}|_{m=4}\times\frac{(-q;q^2)_{\infty}}
 {(q^2;q^2)_{\infty}}}$ \\
 \textbf{S78}-$2q\times$\textbf{S76}=1
 &$\displaystyle{\eqref{q-q}|_{m=3}\times\frac{(-q;q)_{\infty}}
 {(q;q)_{\infty}}}$ \\
 \textbf{S104}-$q\times$\textbf{S103}=1
 &$\displaystyle{\eqref{q-q}|_{m=4}\times\frac{(-q;q)_{\infty}}
 {(q;q)_{\infty}}}$ \\
 \hline
\caption{Sets of JTP-dependent identities.}\label{Ta:t3}
\end{longtable}
\vskip 5mm
A similar consideration of \textbf{S53} $+q\times$\textbf{S55}
yields

\begin{align*}
\sum_{n=0}^{\infty} \frac{q^{4n^2}(-q;q^2)_{2n}(1-q^{4 n+1}-q^{8
n+2}-q^{8 n+4})}{(q^4;q^4)_{2n+1}}
=\frac{(q;q)_{\infty}}{(q^4;q^4)_{\infty}}.
\end{align*}

Likewise, considering \textbf{S40}, \textbf{S41} and \textbf{S42}
(each with $q$ replaced by $q^3$), together with the $m=3$ case of
\eqref{qq}, yields
\begin{equation*}
 \sum_{n=0}^{\infty} \frac{q^{9 n^2}(q^3;q^3)_{3n}
 (1 - q^{1 + 9n} - q^{2 + 9n} + q^{5 + 18n} + q^{7 + 18n} - q^{9 + 18n})}
 {(q^9;q^9)_{n}(q^9;q^9)_{2n+1}}
=\frac{(q;q)_{\infty}}{(q^9;q^9)_{\infty}}.
\end{equation*}

\subsection{Inter-dependence Via an Identity of Weierstrass}
The identity in question is given in the  lemma below. We first
define
\[
[x;q]_{\infty}:=(x,q/x;q)_{\infty}, \hspace{20pt} [x_1,\dots
x_n;q]_{\infty}=[x_1;q]_{\infty}\dots [x_n;q]_{\infty},
\]
and note that $[x^{-1};q]_{\infty}=-x^{-1}[x;q]_{\infty}$.

\begin{lemma}
Let $a_1$, $a_2$, $\dots$, $a_n$; $b_1$, $b_2$, $\dots$, $b_n$ be
non-zero complex numbers such that \\
i) $a_i\not = q^n a_j$, for all $i\not = j$ and all $n \in
\mathbb{Z}$,\\
 ii) $a_1a_2\dots a_n= b_1b_2\dots b_n$. Then
\begin{equation}\label{Weq}
\sum_{i=1}^{n} \frac{\prod_{j=1}^{n}[a_i b_j^{-1};q]_{\infty}} {
\prod_{j=1,j\not=i}^{n}[a_i a_j^{-1};q]_{\infty}}=0.
\end{equation}
\end{lemma}
This lemma appears in \cite{WW27} (page 45) and a proof can be found
in \cite{L91}. If $n=3$,  then \eqref{Weq} can be written as
\begin{equation}\label{Weq3}
\frac{\left [a_1/b_1,a_1/b_2,a_1/b_3; q
 \right ]_{\infty}}
 {\left[
a_1/a_2,a_1/a_3; q
 \right ]_{\infty}}
+\frac{\left [a_2/b_1,a_2/b_2,a_2/b_3; q
 \right ]_{\infty}}{\left
[ a_2/a_1,a_2/a_3; q
 \right ]_{\infty}
 }+
 \frac{\left
[a_3/b_1,a_3/b_2,a_3/b_3; q
 \right ]_{\infty}}{\left
[ a_3/a_1,a_3/a_2; q
 \right ]_{\infty}
 }=0.
\end{equation}

The dependence of identities via \eqref{Weq} is more tenuous than
dependence via series-equivalence or the Jacobi triple product. We
give the following example to illustrate the concept. If we
replace $q$ by $q^{10}$ in \eqref{Weq3} and set
\[
\{a_1,a_2,a_3;b_1,b_2,b_3\}=\{1,-q^5,q;q^{-1},-q^{-2}, q^9\},
\]
then after a little simplification we get that
\begin{equation}\label{wex1a}
q[q,q,-q^2,-q^4;q^{10}]_{\infty} +[q,q^3,-q^4,-q^4;q^{10}]_{\infty}
- [q^2,q^2,-q^3,-q^5;q^{10}]_{\infty}=0.
\end{equation}
This identity can be rearranged to give
\begin{equation}\label{wex1aa}
(-q,-q,q^2;q^2)_{\infty} =
\frac{(q^{10};q^{10})_{\infty}}{(q,q^9,-q^2,-q^8;q^{10})_{\infty}} +
\frac{q(q^{10};q^{10})_{\infty}}{(q^3,q^7,-q^4,-q^6;q^{10})_{\infty}}
\end{equation}

Similarly, if we replace $q$ by $q^{10}$ in \eqref{Weq3} and set
\[
\{a_1,a_2,a_3;b_1,b_2,b_3\}=\{1,-q^5,q^6;q^{-1},-q^{-2}, q^{14}\},
\]
we get that
\begin{equation}\label{wex1b}
-q[q,-q,-q^2,q^4;q^{10}]_{\infty} +[-q,q^3,q^4,-q^4;q^{10}]_{\infty}
- [q^2,-q^2,q^3,-q^5;q^{10}]_{\infty}=0.
\end{equation}
This identity can likewise be rearranged to give
\begin{equation}\label{wex1bb}
(-q^5,-q^5,q^{10};q^{10})_{\infty} =
\frac{(q^{10};q^{10})_{\infty}}{(q,q^9,-q^2,-q^8;q^{10})_{\infty}} -
\frac{q(q^{10};q^{10})_{\infty}}{(q^3,q^7,-q^4,-q^6;q^{10})_{\infty}}
\end{equation}
Finally, if we replace $q$ by $-q$ and subtract, we get that
\begin{equation}\label{wex1ab}
(q^5,q^5,q^{10};q^{10})_{\infty}- (q,q,q^2;q^2)_{\infty} =
\frac{2q(q^{10};q^{10})_{\infty}}{(-q^3,-q^7,-q^4,-q^6;q^{10})_{\infty}}
\end{equation}
See the paper by Cooper and Hirschhorn \cite{CH01} for these and
similar identities.

We now exhibit two identities of Rogers-Ramanujan type which are
related via \eqref{Weq}. The first appears in \cite{BMS07}:
\begin{equation}\label{c15eq}
\sum_{n=0}^{\infty}\frac{q^{(n^2+3n)/2}(-q;q)_{n}}
{(q;q^2)_{n+1}(q;q)_{n+1}} =\frac{(q^{10};q^{10})_{\infty}}
{(q;q)_{\infty}(q;q^2)_{\infty}
(-q^3,-q^4,-q^6,-q^7;q^{10})_{\infty}}.
\end{equation}

The second is an identity of Rogers (see \cite{S06}):
\begin{equation}\label{tseq}
\sum_{n=0}^{\infty}\frac{q^{n(n+1)/2}(-1;q)_{n}}{(q;q)_{n}(q;q^2)_{n}}
=\frac{(q^5,q^5,q^{10};q^{10})_{\infty}}
{(q;q)_{\infty}(q;q^2)_{\infty}}.
\end{equation}
One easily checks (by re-indexing the second sum) that
\[
\sum_{n=0}^{\infty}\frac{q^{n(n+1)/2}(-1;q)_{n}}{(q;q)_{n}(q;q^2)_{n}}
-2 q \sum_{n=0}^{\infty}\frac{q^{(n^2+3n)/2}(-q;q)_{n}}
{(q;q^2)_{n+1}(q;q)_{n+1}}=1.
\]
That
\[
\frac{(q^5,q^5,q^{10};q^{10})_{\infty}}
{(q;q)_{\infty}(q;q^2)_{\infty}}-2
q\frac{(q^{10};q^{10})_{\infty}} {(q;q)_{\infty}(q;q^2)_{\infty}
(-q^3,-q^4,-q^6,-q^7;q^{10})_{\infty}} =1
\]
follows from \eqref{wex1ab}. We say that pairs of identities which
are dependent via \eqref{Weq} are \emph{W-dependent}.

We give one other example of a pair of identities are
$W$-dependent. The first is \textbf{S.56}:
\begin{equation}\label{Weq3ca}
\sum_{n=0}^{\infty}\frac{(-q;q)_nq^{n(n+2)}}{(q;q^2)_{n+1}(q;q)_{n+1}}
= \frac{(-q,-q^{11},q^{12};q^{12})_{\infty}}{(q;q)_{\infty}}.
\end{equation}
The second is
\begin{equation}\label{Weq3cb}
\sum_{n=0}^{\infty}\frac{(-1;q)_nq^{n^2}}{(q;q^2)_{n}(q;q)_{n}} =
\frac{(q^3,q^{3},q^{6};q^{6})_{\infty}(-q;q)_{\infty}}{(q;q)_{\infty}}.
\end{equation}
It is easy to check (by re-indexing, as above) that
\[
\sum_{n=0}^{\infty}\frac{(-1;q)_nq^{n^2}}{(q;q^2)_{n}(q;q)_{n}}-2q\sum_{n=0}^{\infty}\frac{(-q;q)_nq^{n(n+2)}}{(q;q^2)_{n+1}(q;q)_{n+1}}
=1.
\]
That
\[
\frac{(q^3,q^{3},q^{6};q^{6})_{\infty}(-q;q)_{\infty}}{(q;q)_{\infty}}
-2q\frac{(-q,-q^{11},q^{12};q^{12})_{\infty}}{(q;q)_{\infty}}=1
\]
follows from the following identity (proved by Cooper and
Hirschhorn in \cite{CH01}, again making use of \eqref{Weq3}):
\[
(q^3,q^{3},q^{6};q^{6})_{\infty}-(q,q,q^{2};q^{2})_{\infty}=
2q(-q,-q^{11},q^{12};q^{12})_{\infty}(q;q^2)_{\infty}.
\]
It is possible that new, as yet unknown  identities of the
Rogers-Ramanujan type may be derivable from existing identities
through $W$-dependence.

\section{Bailey pairs}

\subsection{The Bailey Pairs of Slater}
Rogers~\cite{R17} categorized the series transformations he discovered into
seven groups labeled A through G.   With hindsight, each of these series transformations
can be seen to correspond to a Bailey pair.  Accordingly, when Slater~\cite{S51, S52} tabulated
the Bailey pairs she used to produce her list of Rogers-Ramanujan type identities, she
retained Rogers' designations, and added Groups H through M which contained
the new Bailey pairs she had discovered; those that did not correspond to Rogers'
series transformations.  The following table summarizes the use of Slater's Bailey pairs
to produce identities via insertion into various limiting cases of Bailey's lemma~\eqref{PBL}--\eqref{FBL}.
\pagebreak

\begin{longtable}{|c|c|c|c|c|c|}
\hline\hline
 BP & \eqref{PBL} &\eqref{TBL}& \eqref{S1BL} & \eqref{S2BL} & \eqref{FBL} \\
\hline\hline
A1 & S98 & S117 &  &         &     \\ 
A2 & S94 &            &  &S80 & \eqref{A2f}\\
A3 &  S99  & S118             &  & S81        &  \eqref{A3f}   \\
A4 & S96         & S119           &  & S82 &    \eqref{A4f}  \\
A5 & S83    &  S100          &  &  \eqref{A5s2}       &  \eqref{A5f}   \\
A6 & S84  &            &  &  S62       &  \eqref{A6f} \\
A7 & S85  &   S95         &  &         &    \eqref{A7f}  \\
A8 &          &    S97        &  &   S63      &   \\
\hline  
B1 & S18   &  S36   & S12 & S13     &     \\
B2 &  S14  &  S34          &  &         &     \\
B3 &          &            &  &   S8      &     \\
\hline
C1 &   S61 &   S79     & \eqref{C1s1}  &         &     \\
C3 &   S60       &            &  &   S45      & \eqref{C3f}   \\
C4 &  S59        &     \eqref{RogMod20}       &  &   S43  &     \eqref{C4f} \\
C5 &  \eqref{C5p} &  S52  &  &            &     \\
C6 &  S46         &            &  &   S26      &   \eqref{C6f} \\
C7 &   S44       &            &  &   S22      &  \eqref{F2f}  \\
\hline
D5 &          &            &  &        &   \eqref{D5f}  \\
D6 &          &            &  &  \eqref{D6s1}    &    \eqref{D6f}  \\
\hline
E1 &   S3-      &    S25        &  &        &   \\
E2 &   S3       &     S4   &  &       &     \\
E3 &   S7    &     \eqref{RamStanton}       &  &    S2     & \eqref{E3f}    \\
\hline
F1 &  S39 & S29, S23-           &  &         &     \\
F2 &  S38        &            &  &   S28     & \eqref{F2f}    \\
F3 &   S9       &  \eqref{F3t}          &  &         &     \\
\hline
G1 &   S33       &     S20,     S21-   &  &         &     \\
G2 &  S31        &            &  &    S17     & \eqref{G2f}   \\
G3 &   S32       &    S16        &  &         &     \\
G4 &   S19       &            &  &         &     \\
G5 &          &            &  &   S5  &  \eqref{G5f}  \\
G6 & S15 &&&&\\
\hline
H1 &   S6    &     S10       &  &         &     \\
H2 &   S39-      &   S23-         &  &         &     \\
H17 &    S1    &            &  &         &     \\
H19 &   S27      &            &  &      S11   &   \eqref{H19f}  \\
\hline
I7 &  S47 &  S66   &  &         &     \\
I8 &  S48    & S67   &  &         &     \\
I9 &  S58 &  S72          &  &         &     \\
I14 & S54 &  S71          &  &         &     \\
I15 & S87      &            &  &  S64  &     \\
I16 &          &            &  &   S65&     \\
I17 &  S56  S69      &            &  &  S37       & \eqref{I17f}    \\
I18 &   S50  & S68           &  & S35       &    \eqref{I18f}  \\
\hline
J1 &   S93       &      S114      &  & S73 &     \\
J2 &  S88  & S108, S113           &  & S74 &     \\
J3 &   S89 & S111           &  &  S75  &   \eqref{J3f}  \\
J4 & S124    & S109           &  &    S107     & \eqref{J4f}   \\
J5 & S125   & S110           &  &         &     \\
J6 &   S90  & S112           &  & S76       &  \eqref{J6f}  \\
\hline
K1 &  S121        & S126         &  &    S101     &     \\
K2 &  S120        &  S127          &  &  S104       &     \\
K3 &  S51       &  S130          &  & S105        &     \\
K4 &          & S70          &  &  S106       &  \eqref{K4f}   \\
K5 &  S123        & S128           &  &  S102       &     \\
K6 &  S122        & S129           &  &  S103       &   \eqref{K6f}  \\
\hline
L2 &  S53        &  \eqref{gs8}        &  &         &     \\
\hline
M2 & S57        &            &  &         &     \\
M3 &  S55        &            &  &         &     \\
\hline
\caption{Slater's Bailey pairs and associated identities}
 \end{longtable}

\subsection{Duality of Bailey pairs}
We recall that Andrews showed in \cite{A84}  that if
$(\alpha_{n}(a,q), \beta_{n}(a,q))$ is a Bailey pair relative to
$a$, then $(\alpha_{n}^*(a,q), \beta_{n}^*(a,q))$ is also a Bailey
pair relative to $a$, where
\begin{align*}
\alpha_{n}^*(a,q) &= a^n q^{n^2}\alpha_{n}(1/a,1/q),\\
 \beta_{n}^*(a,q) &= a^{-n}q^{-n^2-n}\beta_n(1/a,1/q).
\end{align*}

The pair $(\alpha_{n}^*(a,q), \beta_{n}^*(a,q))$ is called \emph{the
dual} of $(\alpha_{n}(a,q), \beta_{n}(a,q))$. Note that the dual of
$(\alpha_{n}^*(a,q), \beta_{n}^*(a,q))$ is $(\alpha_{n}(a,q),$
$\beta_{n}(a,q))$. Identities that arise from Bailey pairs that are
duals of each other may be regarded in some sense as being related,
so for completeness sake we give a table that matches Bailey pairs
with their duals.

Before giving this table, we make some comments. Firstly, Slater's
list of Bailey pairs contains several  pairs whose duals are not
listed, so we list the new Bailey pairs that are dual to these.
Secondly, some pairs, such as  \textbf{H7} and \textbf{E2} are in
fact the same pair, and we indicate the occurrence of such identical
pairs. Thirdly, Slater's list of Bailey pairs contains a good many
typographic errors. Sometimes retracing Slater's steps (making the
same choices that she did for the parameters in the identities she
used to derive the Bailey pairs) will recover the correct Bailey
pair, but this does not always work as in several cases Slater
employs an additional unstated manipulation to derive the Bailey
pair(for example, the same choice of parameters give rise to pairs
\textbf{B1} and \textbf{B2}, but only \textbf{B1} follows directly
from substituting these values into Equation (4.1) of \cite{S51}).

In this latter case we can insert the stated sequence
$\alpha_n(a,q)$ (under the assumption that the stated value
 for $\alpha_n(a,q)$ is correct, and that the error is in the stated value for $\beta_n(a,q)$) into
 the equation
 \begin{equation}\label{BPeq}
 \beta_n(a,q)=\sum_{r=0}^{n}\frac{1}{(q;q)_{n-r}(aq;q)_{n+r}}\alpha_r(a,q),
 \end{equation}
 use \emph{Mathematica} to compute $\beta_n(a,q)$ for low values
 of $n$ and hopefully recognize the  form of $\beta_n(a,q)$ for
 arbitrary $n$. If this does not work, then we try substituting the stated value
 for $\beta_n(a,q)$ into \eqref{BPeq}, and use \emph{Mathematica} to
 compute successive values for $\alpha_n(a,q)$, again try to find
 the formula for $\alpha_n(a,q)$. Of course there is the possibility
 that the pair derived in this way is not the pair Slater intended.
 This could happen if, for example, Slater incorrectly stated $\alpha_n(a,q)$, but this
 $\alpha_n(a,q)$ is part of an entirely different Bailey pair, so that the attempt to find
 Slater's correct pair would instead find the other Bailey pair (for example, Slater
 lists the same $\alpha_n(a,q)$ for both \textbf{F1} and \textbf{H1}, and assuming that
 the $\alpha_n(a,q)$ in \textbf{H1} is correct would lead to a Bailey pair different from what Slater
 had in mind for \textbf{H1}).

\begin{center}
{\bf Slater's Bailey Pairs with their respective duals}
\begin{longtable}{|l|l||l|l|}\hline
Slater's BP & Dual Pair&Slater's BP & Dual Pair
\\\hline
A.1 & A.5 &H.1* & H.1*\\
A.2 & A.8 &H.2 & \ref{pair5}\\
A.3 & A.7&H.5 & H.9 \\
A.4 & A.6&H.6 & H.6 \\
B.1 & H.4 &H.10 & H.11\\
B.2 & H.3&H.12 & H.14\\
B.3 &\ref{pair1}&H.13 & H.13\\
B.4 & \ref{pair2}&H.15 & H.16 \\
C.1 & C.5&H.17 & H.17\\
C.2 & \ref{pair3}&H.18 & H.18\\
C.3 & C.7&H.19=I.15*& H.19=I.15*\\
C.4 & C.6&I.1 & I.4\\
E.1 & E.2 = H.7&I.2*& I.3\\
E.3 & E.7*&I.5& I.6\\
E.4 & E.5= H.8&I.7 & I.8*\\
E.6* & \ref{pair4}&I.9 & I.9\\
F.1 & F.3&I.10 & I.11\\
F.2 & F.4&I.12& I.13\\
G.1 & G.4&I.14 & I.14\\
G.2* & G.5*&I.16*&I.16*\\
G.3 & G.6& I.17 & I.17\\
J.1* & J.1*& I.18 & I.18\\
J.2 & J.3&M.1 & M.1 \\
J.4 & J.5&M.2*&M.2*\\
J.6 & J.6&M.3* & M.3*\\
K.1 & K.2&L.1 & L.3\\
K.3 & K.4&L.2 & L.2\\
K.5 & K.6&L.4* & L.6\\
&&L.5* & L.5*\\\hline
\caption{A * indicates a typographic error in Slater's Paper,
corrected versions of these Bailey pairs are given below. The five new
Bailey pairs \eqref{pair1}---\eqref{pair5} are also stated below.}
\label{table_slater_duals}
\end{longtable}
\end{center}

The new Bailey pairs found (by deriving the duals of Bailey pairs
given by Slater whose duals were not given by her) are the
following.
\begin{flushleft}
\begin{align}\label{pair1}
 \alpha_n & =\frac{(-1)^nq^{\frac{-n(n+3)}{2}}(1-q^{2n+1})}{(1-q)}, \phantom{10pt}(a=q)\\
 \beta_n & =\frac{(-1)^n q^{-n(n+3)/2}}{(q;q)_n}. \notag
\end{align}
\begin{align}\label{pair2}
\alpha_n & =\frac{(-1)^n(q^{(-n^2+n+4)/2}(1-q^{n})+q^{-n(n+5)/2}(1-q^{n+1})) }{(1-q)}, \phantom{10pt}(a=q)\\
 \beta_n & =\frac{(-1)^n q^{\frac{-n(n+5)}{2}}}{(q;q)_n}.\notag
\end{align}
\begin{align}\label{pair3}
\alpha_{2n} & = (-1)^nq^{n^2}\left(q^{n}+ q^{-n}\right), \phantom{10pt}(a=1)\\
\alpha_{2n+1} & = (-1)^{n}q^{n^2-n-1}(1-q^{4n+2})\notag\\
 \beta_n & =\frac{q^{\frac{n(n-3)}{2}}}{(q;q)_n(q;q^2)_n}.\notag
\end{align}
\begin{align}\label{pair4}
\alpha_n & =\frac{(-1)^n q^{-2n}(1- q^{4n+2})}{(1-q^2)},\phantom{10pt} (a=q)\\
 \beta_n & =\frac{(-1)^nq^{-2n}}{(-q,q;q)_n}.\notag
\end{align}
\begin{align}\label{pair5}
\alpha_n & =(-1)^n(q^{\frac{n}{2}}+q^{-\frac{n}{2}}), \phantom{10pt}(a=1)\\
\beta_n & =\frac{(-1)^n q^{-\frac{n}{2}}}{(-q^{\half},q;q)_n}.\notag
\end{align}
\end{flushleft}

We next give corrected versions of the Bailey pairs (``BP" in the
tables below) in the Slater papers \cite{S51, S52}, which have
typographic errors.
\pagebreak

\begin{center}
{\bf Corrected Bailey pairs}\\
\begin{longtable}{|l|c|c|} \hline
BP & $\alpha_n $& $\beta_n$ \\ \hline
&&\\
E.6*&$\frac{(-1)^n q^{n^2-n}(1- q^{4n+2})}{(1-q^2)}$ &$\frac{q^n}{(-q,q;q)_n}$\\
&&\\ \hline
&&\\
E.7*&$(-1)^nq^{-n}\frac{\left(1 - q^{2n+1}\right)}{(1-q)}$ &$\frac{(-1)^nq^{-n}}{(-q,q;q)_n}$\\
&&\\ \hline
&&\\
H.1*&$q^{n^2}(q^{\frac{n}{2}}+q^{-\frac{n}{2}})$ &$\frac{1}{(q^{\half},q;q)_n}$\\
&&\\ \hline
&&\\
I.15*&$q^{\frac{n^2}{2}}\frac{(1+q^{n+\frac{1}{2}})}{(1+q^\half)}$ &$\frac{(-q^\half;q)_n}{(q^{\frac{3}{2}};q)_n(q^2;q^2)_n}$\\
&&\\ \hline
&&\\
L.5*&$q^{\frac{n(n-1)}{2}}\left(1 + q^{n}\right)$ &$\frac{(-1;q)_{n}}{(q;q)_n (q;q^2)_{n}}$\\
&&\\ \hline
&&\\
M.2*&$q^{\frac{n(2n+1)}{4}}\frac{\left(1 + q^{\frac{(2n+1)}{4}}\right)}{(1+q^{\frac{1}{4}})}$ &$\frac{(-q^{\frac{3}{4}};q^\half)_{2n}}{(q^2;q)_{2n}}$\\
&&\\ \hline
&&\\
M.3* &$(-1)^nq^{\frac{n(2n+1)}{4}}\frac{\left(1 - q^{\frac{(2n+1)}{4}}\right)}{(1-q^{\frac{1}{4}})}$ &$\frac{(q^{\frac{3}{4}};q^\half)_{2n}}{(q^2;q)_{2n}}$\\
& & \\ \hline
\end{longtable}
\end{center}

\begin{center}
{\bf Corrected Bailey pairs}\\
\begin{longtable}{|l|c|c|c|} \hline
BP & $\alpha_{2n}$ & $\alpha_{2n+1}$ & $\beta_n$ \\ \hline
&&&\\
G.2* & $\frac{q^{\frac{6n^2+n}{2}}\left(1-q^{2n+\frac{1}{2}}\right)}{(1-\sqrt{q})} $ & $ \frac{q^{\frac{6n^2+11n+5}{2}}\left(1-q^{-2n-\frac{3}{2}}\right)}{(1-\sqrt{q})} $ & $\frac{1}{(-q^{\frac{3}{2}};q)_n(q^2;q^2)_n} $ \\
&&&\\\hline
&&&\\
 G.5* & $\frac{q^{\frac{2n^2-n}{2}}\left(1-q^{2n+\frac{1}{2}}\right)}{(1-\sqrt{q})} $ & $ \frac{q^{\frac{2n^2+5n+3}{2}}\left(1-q^{-2n-\frac{3}{2}}\right)}{(1-\sqrt{q})}$ & $ \frac{(-1)^n q^{\frac{n^2}{2}}}{(-q^{\frac{3}{2}};q)_n(q^2;q^2)_n}$ \\
&&&\\\hline
&&&\\
I.2* & $(-1)^nq^{n^2}\left(q^{\frac{n}{2}}+ q^{-\frac{n}{2}}\right) $ & $(-1)^{n+1}q^{n^2}\left(q^{\frac{n}{2}}- q^{\frac{3n+1}{2}}\right) $ & $\frac{q^{\frac{n^2}{2}}}{(\sqrt{q};q)_n(q^2;q^2)_n} $ \\
&&&\\\hline
&&&\\
I.8* & $ (-1)^nq^{2n^2}\left(q^{n}+ q^{-n}\right)$& $(-1)^nq^{2n^2}\left(q^{3n+1}- q^{n}\right) $ & $\frac{q^n(-1;q^2)_n}{(q,q^2;q^2)_n} $\\
&&&\\\hline
&&&\\
I.16* & $ (-1)^nq^{2n^2}\frac{(1+q^{2n+\frac{1}{2}})}{(1+q^{\frac{1}{2}})}$ & $(-1)^{n+1}q^{2n^2+2n + \frac{1}{2}}\frac{(1+q^{2n+\frac{3}{2}})}{(1+q^{\frac{1}{2}})} $ & $\frac{(-q;q^2)_{n}}{(q^2;q^2)_n(-q^{\frac{1}{2}},q^{\frac{3}{2}};q)_{n}} $\\
&&&\\\hline
&&&\\
L.4* & $(-1)^nq^{n(n-1)}\left(1 + q^{2n}\right) $ & $0 $ & $\frac{q^{\frac{n(n-1)}{2}}}{(-q^\half;q)_n (q^\half;q^\half)_{2n}} $\\
&&&\\\hline
\end{longtable}
\end{center}

\begin{center}
{\bf Corrected Bailey pair}\\
\begin{longtable}{|l|c|c|c|c|} \hline
BP&$\alpha_{3n-1}$&$\alpha_{3n}$&$\alpha_{3n+1}$&$\beta_{n}$\\\hline
&&&&\\
J.1*&$ 0$& $(-1)^n q^{\frac{3n(3n-1)}{2}}(1+q^{3n}) $ & $0 $ & $\frac{(q^3;q^3)_{n-1}}{(q;q)_{n}(q;q)_{2n-1}} $\\
&&&&\\\hline
\end{longtable}
\end{center}

\section{Conclusion}
We have only touched on a small part of the Rogers-Ramanujan story in this survey.
The main goal has been to present an expanded version of Slater's list with the
earliest known reference to each identity in the literature.
Slater's list contained only a few references to the
earlier literature, and of course, Ramanujan's lost notebook was unknown to
the mathematical community in 1952.
Accordingly, the authors believe it was a useful endeavor to bring together
Slater's list with Ramanujan's lost notebook, and the dozens of additional
identities of similar type which have been scattered throughout the literature
over the years.  Since Slater's main tool was Bailey's lemma and Bailey pairs,
we included an exposition of this material in the introduction.

We have not attempted to address the numerous advances in the theory of $q$-series
which have occurred in the past few decades (e.g. Andrews's discovery of the
``Bailey Chain," and its application to multisum--product identities;
the extension of the Bailey chain to the WP Bailey chain by Andrews, Berkovich;
the extension of the Bailey chain to elliptic hypergeometric series by Spiridonov and
Warnaar;
Lepowsky, Milne, and Wilson's connection of $q$-series to Lie algebras and
vertex operator algebras; the work of Baxter, Berkovich, Forrester, McCoy, Melzer, Warnaar
connecting $q$-series to models in statistical mechanics; the work on finite
analogs by Andrews, Berkovich, Forrester, McCoy, Melzer, Paule, Riese,
the second author, Warnaar, Wilf, and Zeilberger.
Further, we have not considered any of the combinatorial consequences of
these identities, of which there are too many to even begin to list the most
important contributors.

\section*{Acknowledgements}
As with any long list of identities, keyed into \LaTeX \ by hand,
there are bound to be typographical errors, and
errors of omission.  Accordingly, the authors would like to thank in advance the
kind readers who will bring such errors, omissions, and updates to our attention,
so that we can make this survey as useful to our readers as possible.

In particular, we thank Michael Somos for finding and helping us correct 
numerous errors in the first published version.

\end{document}